\documentclass[a4paper,12pt]{article}
\usepackage{amsmath, amsfonts, amssymb, amsthm}
\usepackage[english]{babel}

\newcounter{num}[section]

\newenvironment{theorem}
{\refstepcounter{num}%
\bigskip\noindent\nopagebreak[4]{\bf Theorem~\arabic{section}.\arabic{num}. }\it}

\newenvironment{proposition}
{\refstepcounter{num}%
\bigskip\noindent\nopagebreak[4]{\bf Proposition~\arabic{section}.\arabic{num}. }\it}

\newenvironment{corollary}
{\refstepcounter{num}%
\bigskip\noindent\nopagebreak[4]{\bf Corollary~\arabic{section}.\arabic{num}. }\it}

\newenvironment{remark}
{\refstepcounter{num}%
\bigskip\noindent\nopagebreak[4]{\bf Remark~\arabic{section}.\arabic{num}. }}

\newcommand{\N}{{\mathbb{N}}}
\newcommand{\Zbb}{{\mathbb{Z}}}

\newcommand{\LL}{{\mathcal{L}}}
\newcommand{\Hom}{{\mathrm{Hom}}}
\newcommand{\ucl}{{\mathbf{ucl}}}
\newcommand{\qvar}{{\mathbf{qvar}}}

\newcommand{\Rad}{{\mathrm{Rad}}}

\newcommand{\Ss}{{\mathbf{S}}}
\newcommand{\V}{{\mathrm{V}}}

\newcommand{\lb}{{\langle}}
\newcommand{\rb}{{\rangle}}

\newcommand{\om}{{\omega}}
\newcommand{\pr}{{\prime}}

\newcommand{\al}{{\alpha}}

\newcommand{\Acal}{{\mathcal{A}}}
\newcommand{\Bcal}{{\mathcal{B}}}
\newcommand{\Ccal}{{\mathcal{C}}}

\newcommand{\Fcal}{{\mathcal{F}}}

\newcommand{\Tcal}{{\mathcal{T}}}

\newcommand{\Cbf}{{\mathbf{C}}}

\newcommand{\Nbf}{{\mathbf{N}}}

\newcommand{\Qbf}{{\mathbf{Q}}}

\newcommand{\Ubf}{{\mathbf{U}}}

\newcommand{\qq}{{\mathbf{q}_\omega}}
\newcommand{\uu}{{\mathbf{u}_\omega}}

\newcommand{\1}{{^{-1}}}

\newcommand{\Wiki}{{${}^{\mathrm{Wiki}}$}}

\begin{document}
\title{Lectures notes in universal algebraic geometry}
\author{Artem N. Shevlyakov}

\maketitle

\tableofcontents


\section{Introduction}

{\it Universal algebraic geometry} studies the general properties of equations over arbitrary algebraic structures. Such general principles really exist, and there are at least two attempts to explain them. 
\begin{enumerate}
\item In the papers by B.I. Plotkin (with co-authors)~\cite{Plot1}--\cite{Plot9} such principles were explained in the language of Lindebaum algebras.  
\item In the papers~\cite{uni_Th_I,uni_Th_II,uni_Th_III,uni_Th_IV} written by E. Daniyarova, A. Miasnikov, V. Remeslennikov (DMR) the universal algebraic geometry is studied by the methods of model theory and universal algebras.
\end{enumerate}

The main achievement of DMR is the pair of so-called Unifying theorems (see Theorems~\ref{th:unify_approx},~\ref{th:unify_discr} of our paper). The Unifying theorems give the connections between equations and other six (!) notions in algebra. Thus, the universal algebras geometry becomes a set of six equivalent approaches for solving equations over arbitrary algebraic structures. Therefore, any researcher may choose the most useful approach for him to solve equations over the given algebraic structure.

Moreover, the statements of Unifying theorems give the connection between equations and elementary theory of an algebraic structure $\Acal$. Therefore, the study of equations over $\Acal$ is the starting point in the study of the elementary theory $\mathrm{Th}(\Acal)$. This strategy was applied by O. Kharlampovich and A. Miasnikov in~\cite{KM1,KM2,KM3}, where the description of the solution sets of equations over free non-abelian groups allowed to solve positively the famous Tarski problem in~\cite{KM4}.

Unfortunately, there are difficulties in the reading of DMRs` papers~\cite{uni_Th_I}--~\cite{uni_Th_V}, since the authors use the abstract language of model theory and rarely give examples. We hope our survey elementary explains the ideas of DMR and gives numerous examples for universal algebraic geometry.  

In the series of the papers by DMR we extract the main results, prove them by elementary methods and illustrate by numerous examples. Thus, {\it each section of our survey consists of two parts. In the first part we give definitions and theorems of universal algebraic geometry, and in the second one we explain them with examples from the following classes of algebraic structures:}
\begin{itemize}
\item abelian groups (here we follow the paper~\cite{MR2});
\item free monoids and semigroups;
\item the additive monoid of natural numbers (here we follow the papers~\cite{shevl_over_N_irred,shevl_over_N_qvar});
\item semilattices;
\item rectangular bands;  
\item lest zero semigroups;
\item unars (unar is a set equipped with a unary function $f(x)$).
\end{itemize} 

In the first sections we consider all algebraic structures above in the ``standard'' languages (for example, any group is an algebraic structure of the language $\LL_g=\{\cdot,^{-1},1\}$). It follows that we study equations with no constants (parameters). The general case (equations with constants) is studied in Section~\ref{sec:appearing_of_const}.

The survey is concluded by Section~\ref{sec:researches}, where we refer to papers devoted to equations in various classes of algebraic structures. 

We tried to write an elementary survey. The reader needs only the simple facts in linear algebra and group theory. Some notions in our text are marked by the symbol\Wiki. For example, the expression ``cyclic group\Wiki ''means that the survey does not contain the definition of a cyclic group, but we recommend the reader to find it in Wikipedia.

\section{Algebraic structures}

A set of functional symbols $\LL=\{f_1^{(n_1)},f_2^{(n_2)},\ldots\}$, where $f_i^{n_i}$ denotes a $n_i$-ary function, is a {\it functional language} (language, for shortness). A functional $0$-ary symbol  $f_i^{(0)}$ is called a~\textit{constant symbol}, and in the sequel we shall denote constant symbols with no upper index. An {\it algebraic structure} $\Acal=\lb A\mid \LL\rb$ of a language $\LL$ ($\LL$-\textit{algebra} for shortness) is an nonempty set $A$, where each functional symbol $f_i^{(n_i)}\in\LL$ corresponds to a function $f_i^\Acal\colon A^{n_i}\to A$ (clearly, each constant symbol $f_i^{(0)}$ corresponds to some element of $A$). The function $f_i^\Acal$ is called the~\textit{interpretation} of a functional symbol $f_i^{(n_i)}\in\LL$. The set $A$ of an $\LL$-algebra $\Acal$ is the  \textit{universe} of $\Acal$. If $|A|=1$ an algebraic structure $\Acal$ is called \textit{trivial}. 

We shall use the following denotations: an expression $a\in\Acal$ ($(a_1,a_2,\ldots,a_n)\in\Acal^n$) means $a\in A$ (respectively, $(a_1,a_2,\ldots,a_n)\in A^n$). 

Let us define main classes of algebraic structures which are used for our examples.
 
\begin{enumerate}
\item 
An algebra $\Acal$ of the language $\LL_s=\{\cdot^{(2)}\}$ is a {\it semigroup}, if $\Acal$ satisfies the axiom of the associativity
\begin{equation}
\forall x\forall y\forall z \; (xy)z=x(yz).
\label{eq:multiplication_associativity_axiom}
\end{equation}
The binary function $\cdot$ is called a \textit{multiplication}.

\item 
An algebra $\Acal$ of the language $\LL_m=\{\cdot^{(2)},1\}$ is a  {\it monoid}, if $\Acal$ satisfies~(\ref{eq:multiplication_associativity_axiom}) and
\begin{equation}
\forall x (x\cdot 1=x),\; \forall x (x\cdot 1=x).
\label{eq:identity_axiom}
\end{equation}
The element $1$ is the {\it identity} of a monoid $\Acal$. 

\item 
An algebra $\Acal$ of the language $\LL_g=\{\cdot^{(2)}, {{}^{-1}}^{(1)},1\}$ is a {\it group}, if $\Acal$ satisfies~(\ref{eq:multiplication_associativity_axiom},~\ref{eq:identity_axiom}) and
\begin{equation}
\forall x (xx^{-1}=1),\; \forall x (x^{-1}x=1)
\label{eq:inversion_axiom}
\end{equation}
The unary function ${}^{-1}$ is called an \textit{inversion}.

\item A group is \textit{commutative}, if it satisfies
\[
\forall x\forall y \; (xy=yx).
\]
Commutative groups (abelian groups) are considered as algebraic structures of the additive language $\LL_{+g}=\{+,-,0\}$, where the binary operation $+$ is called an {\it addition} and $0$ is the {\it zero element}. Similarly, one can define commutative semigroups (monoids) as algebraic structures of the language  $\LL_{+s}=\{+^{(2)}\}$ (respectively, $\LL_{+m}=\{+^{(2)},0\}$).

\item Any algebraic structure of the language  $\LL_{u}=\{f^{(1)}\}$ is said to be \textit{unar}.
\end{enumerate}

\subsection{Free objects}

Let us give the definition of a 
{\it term of a language $\LL$ ($\LL$-term for shortness) in variables $X=\{x_1,x_2,\ldots,x_n\}$}  as follows: 
\begin{enumerate}
\item any variable $x_i\in X$ is an $\LL$-term;
\item if $f^{(m)}$ is a functional symbol of a language $\LL$ and $t_1(X),t_2(X),\ldots,t_m(X)$ are $\LL$-terms, then the expression $f^{(m)}(t_1(X),t_2(X),\ldots,t_m(X))$ is an $\LL$-term.
\end{enumerate}

Since any constant symbol of a language $\LL$ is a $0$-ary functional symbol, all constant symbols are $\LL$-terms. 

\begin{remark}
Roughly speaking, an $\LL$-term is a composition of variables $X$, constant and functional symbols of $\LL$.
\end{remark}

According to the definition, $\LL$-terms may be very complicated expressions. For example, the following expressions are terms of the group language $\LL_g$:
\[
x_1,\; x_1(x_1)\1,\; ((x_1x_2)(x_3)\1),\; (((x_1)\1)\1)\1,\; ((1\1)(1\1))\1.
\]
In Section~\ref{sec:simplification_of_equations} we discuss the ways to simplify $\LL$-terms in some classes of $\LL$-algebras. 

A formula\Wiki of the language $\LL$ of the form
\[
\forall x_1\forall x_2\ldots\forall x_n \; t(x_1,x_2,\ldots,x_n)=s(x_1,x_2,\ldots,x_n), 
\]
where $t,s$ are terms of $\LL$, is called an~\textit{identity}. The class of $\LL$-algebras defined by a set of identities is called a~\textit{variety}. Obviously, the classes of all semigroups and monoids are defined by the identities of the languages $\LL_s$, $\LL_m$ respectively, therefore such classes form varieties. By Birkgof`s theorem\Wiki any variety contains a free\Wiki object. Thus, there exist~\textit{free semigroups $FS_n$ and free monoids $FM_n$} of rank\Wiki $n$. Below we give constrictive definitions for $FS_n$ and $FM_n$.   

The class of all algebraic structures of the language $\LL_u$ forms the variety of unars. Let us denote the free unar of rank $n$ by $FU_n$. It is easy to prove that $FU_n$ is isomorphic\Wiki (we give the definition of an isomorphism in Section~\ref{sec:subalgebras_direct_powers_hom}) to the disjoint union of $n$ copies of natural numbers $\N_1,\N_2,\ldots,\N_n$, $\N_i=\{0_i,1_i,2_i,\ldots,k_i,\ldots\}$, and the function $f$ is defined as follows $f(k_i)=(k+1)_i$.

\subsection{Abelian groups}

We shall use the following denotations.
\begin{enumerate}
\item $\Zbb$ is the abelian group of integers $\{0,\pm 1,\pm 2,\ldots\}$ relative to the operation of the addition. 
\item $\Zbb_{n}$ is the cyclic\Wiki group of order $n$. Elements of $\Zbb_n$ are denoted by natural numbers $\{0,1,2,\ldots,n-1\}$ relative to the addition modulo $n$.  
\end{enumerate}

\subsection{Semigroups}

In the variety of semigroups we deal only with the following classes of semigroups.
\begin{enumerate}
\item A {\it semilattice} $S$ is a semigroup, where all elements are idempotents ($\forall x\; xx=x$) and commute ($\forall x\forall y\; xy=yx$). Any semilattice admits a partial order 
\[
x\leq y\Leftrightarrow xy=x.
\]

If a partial order over $S$ is linear\Wiki, the semigroup is said to be {\it linearly ordered}. A linearly ordered semilattice with exactly $n$ elements  is denoted by $L_n$.

\item A \textit{left zero semigroup} is a semigroup which satisfies the identity
\[
\forall x\forall y\; xy=x.
\]
It is easy to prove that there exists a unique (up to isomorphism) left zero semigroup with $n$ elements. Let us denote such semigroup by $LZ_n$.  

\item A \textit{rectangular band} is a semigroup which satisfies the identity
\[
\forall x\; xx=x,\; \forall x\forall y\forall z\; xyz=xz.
\]
By semigroup theory, any rectangular band is isomorphic to a set of pairs $RB(n,m)=\{(i,j)\mid 1\leq i\leq n,1\leq j\leq m\}$ under multiplication $(i,j)(i^\pr,j^\pr)=(i,j^\pr)$. Obviously, $RB(n,m)$ is a left zero semigroup if $m=1$.

\item The \textit{free semigroup} $FS_n$ of rank $n$ is the set of all words in an alphabet $A_n=\{a_1,a_2,\ldots,a_n\}$ under the operation of a word concatenation.

\item The \textit{free monoid} $FM_n$ of rank $n$ is obtained from $FS_n$ by the adjunction of the empty word $1$ which satisfies the identity axiom~(\ref{eq:identity_axiom}).
 
\item For $n=1$ the free monoid $FM_n$ is isomorphic to the monoid of natural numbers $\N=\{0,1,2,\ldots\}$ relative to the addition. In the sequel we consider the monoid $\N$ in the language $\LL_{+m}$.

\item A semigroup $S$ has {\it cancellations} if each equality $xy=xz$, $yx=zx$ implies $y=z$. It is directly checked that the free monoids $FM_n$ has cancellations.

\end{enumerate}

By the definition, the class of all semilattices (respectively, rectangular bands, left zero semigroups) forms a variety.

\section{Subalgebras, direct products, homomorphisms}
\label{sec:subalgebras_direct_powers_hom}

Let $\Acal=\lb A\mid \LL\rb$ be an algebraic structure of a language $\LL$. An $\LL$-algebra $\Bcal=\lb B\mid\LL\rb$ is called a~\textit{subalgebra} of $\Acal$ if 
\begin{enumerate}
\item $B\subseteq A$;
\item all functions over $B$ are restrictions of corresponding functions over $A$
\[
f^\Bcal(b_1,b_2,\ldots,b_n)=f^\Acal(b_1,b_2,\ldots,b_n),
\]
\item $B$ is closed under all functions $f^\Bcal$.
\end{enumerate}

By the definition, any subalgebra of an $\LL$-algebra $\Acal$ should contain all constants from the set $A$, since constants are unary functions over $A$. Obviously, in the variety of groups (semigroups) subalgebras of a group (respectively, semigroup) $\Acal$ are exactly subgroups\Wiki (respectively, subsemigroups\Wiki) of $\Acal$. 

An $\LL$-algebra $\Acal=\lb A\mid \LL\rb$ is~\textit{finitely generated} if there exists a finite set of elements $S=a_1,a_2,\ldots,a_n$ such that any non-constant element of $\Acal$ is a value of a composition of functions over $\Acal$ at a point $P=(p_1,p_2,\ldots,p_m)$ where for each $i$ either $p_i\in S$ or $p_i$ is a constant of an $\LL$-algebra $\Acal$.

Remark that an $\LL$-algebra $\Acal$ is always finitely generated if {\it each} element of $\Acal$ is defined as constant. 

One calls a variety $V$ of $\LL$-algebras {\it locally finite} if every finitely generated algebra from $V$ has finite cardinality. For example, the varieties of semilattices, left zero semigroups and rectangular bands are locally finite.   

Let $\Acal=\lb A\mid \LL\rb$, $\Bcal=\lb B\mid \LL\rb$ be $\LL$-algebras. A map $\phi\colon A\to B$ is called a~\textit{homomorphism} if for any functional symbol $f^{(n)}\in\LL$ it holds 
\[
\phi(f^\Acal(a_1,a_2,\ldots,a_n))=f^\Bcal(\phi(a_1),\phi(a_2),\ldots,\phi(a_n)). 
\] 
It follows that for any constant symbol $c\in\LL$ we have the equality $\phi(c^\Acal)=c^\Bcal$. In other words, any homomorphism should map constants into constants. The set of all homomorphisms between $\LL$-algebras $\Acal,\Bcal$ is denoted by $\Hom(\Acal,\Bcal)$.

Let us apply the general definition of a homomorphism to groups, semigroups and unars.  
\begin{enumerate}
\item A map $\phi\colon \Acal\to\Bcal$ is a homomorphism between two semigroups $\Acal,\Bcal$ of the language $\LL_s$ if for all $a_1,a_2\in\Acal$ it holds
\begin{equation}
\label{eq:hom_for_semigroups}
\phi(a_1\cdot a_2)=a_1\cdot a_2
\end{equation}
\item A map $\phi\colon \Acal\to\Bcal$ is a homomorphism between two monoids $\Acal,\Bcal$ of the language $\LL_m$ if for all $a_1,a_2\in\Acal$ it holds~(\ref{eq:hom_for_semigroups}) and  
\begin{equation}
\label{eq:hom_for_monoids}
\phi(1)=1
\end{equation}
\item A map $\phi\colon \Acal\to\Bcal$ is a homomorphism between two groups $\Acal,\Bcal$ of the language $\LL_g$ if for all $a_1,a_2\in\Acal$ it holds~(\ref{eq:hom_for_semigroups},~\ref{eq:hom_for_monoids}) and  
\begin{equation}
\label{eq:hom_for_groups}
\phi(a_1^{-1})=\phi(a_1)^{-1}.
\end{equation}
\item A map $\phi\colon \Acal\to\Bcal$ is a homomorphism between two unars $\Acal,\Bcal$ of a language $\LL_u$ if for any $a\in\Acal$ we have $\phi(f(a))=f(\phi(a))$.
\end{enumerate}

A homomorphism $\phi\in\Hom(\Acal,\Bcal)$ is an {\it embedding} if $\phi$ is injective, i.e. for any pair of distinct elements $a_1,a_2\in\Acal$ it holds $\phi(a_1)\neq \phi(a_2)$.

Let us describe homomorphism sets and embeddings  for some algebras $\Acal,\Bcal$  (in all examples below $p,q$ are distinct primes).

\begin{enumerate}
\item The abelian group $\Zbb_{p^n}$ is embedded into $\Zbb_{p^m}$ iff $n\leq m$. For example, $\Zbb_4=\{0,1,2,3\}$ is embedded into $\Zbb_8=\{0,1,\ldots,7\}$ by $\phi(x)=2\cdot x$.
\item Any homomorphism $\phi\in\Hom(\Zbb_{p^n},\Zbb_{p^m})$ ($n>m$) maps an element $p^m\in \Zbb_{p_n}$ into $0\in \Zbb_{p^m}$.
\item The unique homomorphism of the set $\Hom(\Zbb_{p^n},\Zbb_{q^m})$ ($p\neq q$) is the trivial homomorphism $\phi(x)=0$. 
\item The unique homomorphism of the set $\Hom(\Zbb_{p^n},\Zbb)$ is the trivial homomorphism $\phi(x)=0$.

\item Any homomorphism between $\Zbb$ and $\Zbb_{p^n}$ should map the element $p^n\in\Zbb$ into $0\in\Zbb_{p^n}$.
 
\item The left zero semigroup $LZ_n$ is embedded into $LZ_m$ iff $n\leq m$.

\item Clearly, a rectangular band $RB(n,m)$ is embedded into $RB(n^\pr,m^\pr)$ iff $n\leq n^\pr$, $m\leq m^\pr$. 

\end{enumerate}

A homomorphism $\phi\in\Hom(\Acal,\Bcal)$ of two $\LL$-algebras $\Acal,\Bcal$ is an {\it isomorphism} if $\phi$ is bijective. A pair of isomorphic $\LL$-algebras $\Acal,\Bcal$ is denoted by $\Acal\cong\Bcal$.

Let $\{\Acal_i\mid i\in I\}$ be a family of $\LL$-algebras. The \textit{direct product} $\Acal=\prod_{i\in I}\Acal_i$  is an $\LL$-algebra of all tuples
\[
(a_1,a_2,\ldots,a_n,\ldots),\; a_i\in\Acal_i,
\]
and the value of any function $f$ over $\Acal$ is the tuple of values of $f$ in algebras $\Acal_i$. For example, a binary function $f(x,y)$ over $\Acal$ is computed as follows
\[
f^\Acal((a_1,a_2,\ldots,a_n,\ldots),(a_1^\pr,a_2^\pr,\ldots,a_n^\pr,\ldots))=
(f^{\Acal_1}(a_1,a_1^\pr),f^{\Acal_2}(a_2,a_2^\pr),\ldots,f^{\Acal_n}(a_n,a_n^\pr),\ldots).
\]
By the definition of a direct product, we have the following interpretation of constant symbols. Suppose $a_i\in\Acal_i$ is an interpretation of a constant symbol $c\in\LL$ in $\Acal_i$, then the element $(a_1,a_2,\ldots,a_i,\ldots)$ is the interpretation of $c$ in the algebra $\Acal$.

If all algebras $\Acal_i$ are isomorphic to an algebra $\Bcal$, the direct product $\prod_i \Acal_i$ is called a~\textit{direct power} of an algebra $\Bcal$.

Let $\Acal=\prod_{i\in I}\Acal_i$. By $\pi_i\colon\Acal\to\Acal_i$ we shall denote the \textit{projection} of $\Acal$ onto the subalgebra $\Acal_i$. In other words,
\[
\pi_i((a_1,a_2,\ldots,a_i,\ldots))=a_i.
\]
It is directly checked that any projection $\pi_i$ is a homomorphism between $\Acal$ and $\Acal_i$.
 
Let us give examples of direct products in various classes of algebras.

\begin{enumerate}
\item In the category of abelian groups direct products are called~\textit{direct sums} and denoted by  $\oplus$.For example, one can prove that the cyclic group $\Zbb_6$ is isomorphic to the direct sum $\Zbb_2\oplus \Zbb_3$.

By the theory of abelian groups, any finitely generated abelian group $A$ is isomorphic to a direct sum
\begin{equation}
\label{eq:abelian_group_presentation}
\Zbb\oplus\Zbb\oplus\ldots\oplus\Zbb\oplus\Zbb_{p_1^{n_1}}\oplus\Zbb_{p_2^{n_2}}\oplus\ldots\oplus\Zbb_{p_l^{n_l}},
\end{equation}
where $p_i$ are primes. The subgroup $\Zbb_{p_1^{n_1}}\oplus\Zbb_{p_2^{n_2}}\oplus\ldots\oplus\Zbb_{p_l^{n_l}}$ is the \textit{torsion} of a group $A$ and denoted by $T(A)$.
\item The proof of the following two facts is refer to the reader. The direct product of two left zero semigroups $LZ_n,LZ_m$ is isomorphic to $LZ_{n\cdot m}$. The direct product of two rectangular bands $RB(n,m), RB(n^\pr,m^\pr)$ is isomorphic to $RB(nn^\pr,mm^\pr)$. 
\item Remark that nontrivial direct products of free semigroups $FS_n$ and free monoids $FM_n$ are not free.
\end{enumerate}

\section{Simplifications of terms}
\label{sec:simplification_of_terms_in_var}

As we mentioned above, $\LL$-terms may be very complicated expressions. 
Fortunately, for many varieties of algebras one can simplify term expressions.

Let $V$ be a variety of $\LL$-algebras. $\LL$-terms $t(X),s(X)$ are called {\it equivalent over $V$} if the identity $\forall x_1\forall  x_2\ldots \forall x_n \; t(X)=s(X)$ holds in each algebra of $V$. 
Let us give examples which show the simplification of terms in many varieties of algebras (in all examples below terms depend on variables $X=\{x_1,x_2,\ldots,x_n\}$).

\begin{enumerate}
\item $V=${\bf the variety of all semigroups}. Since the multiplication is associative in every semigroup,  any $\LL_s$-term is equivalent over $V$ to a term with no brackets:
\begin{equation}
x_{i_1}^{k_1}x_{i_2}^{k_2}\ldots x_{i_m}^{k_m} \; (k_j\in\N).
\label{eq:semigroup_term}
\end{equation}

\item $V=${\bf the variety of all groups}. Using the axioms of group theory, one can obtain that every $\LL_{g}$-term is equivalent over $V$ to a product of variables in integer powers
\begin{equation}
x_{i_1}^{k_1}x_{i_2}^{k_2}\ldots x_{i_m}^{k_m} \; (k_j\in\Zbb).
\label{eq:group_term}
\end{equation}

\item  $V=${\bf the variety of all idempotent semigroups}. Recall that a semigroup is~\textit{idempotent} if it satisfies the identity $\forall x\; xx=x$. Thus, any $\LL_s$-term is equivalent over $V$ to an expression of the form
\begin{equation}
x_{i_1}x_{i_2}\ldots x_{i_k}.
\label{eq:band_term}
\end{equation}

\item $V=${\bf the variety of all semilattices}.  Since a semilattice is a commutative idempotent semigroup, one can sort variables in~(\ref{eq:band_term}) and obtain
\begin{equation}
x_{i_1}x_{i_2}\ldots x_{i_k}\; (i_1<i_2<\ldots<i_k).
\label{eq:semilattice_term}
\end{equation}

\item  $V=${\bf the variety of all rectangular band}. Recall that a rectangular band is an idempotent semigroup satisfying the identity $\forall x\forall y\forall z \;(xyz=xz)$. Therefore the expression~(\ref{eq:band_term}) is equivalent over $V$ to $x_{i_1}x_{i_k}$. In other words, any $\LL_s$-term is equivalent over $V$ to a term in at most two letters.

\item  $V=${\bf the variety of all left zero semigroups}. The identity $\forall x\forall y\; (xy=x)$ holds in $V$. Thus, any $\LL_s$-term is equivalent over $V$ to a term in one letter.

\item $V=${\bf the variety of all abelian groups}. It is directly checked that every $\LL_{+g}$-term is equivalent over $V$ to an expression of the form
\begin{equation}
\label{eq:abelian_group_term}
k_1x_1+k_2x_2+\ldots+ k_nx_n,
\end{equation}
where $k_i\in\Zbb$, and expressions $k_ix_i$ are defined as follows
\[
k_ix_i=\begin{cases}
\underbrace{x_i+x_i+\ldots+x_i}_{\mbox{$k_i$ times}}\mbox{ if $k_i> 0$}\\
-(\underbrace{x_i+x_i+\ldots+x_i}_{\mbox{$k_i$ times}})\mbox{ if $k_i< 0$}\\
0\mbox{ if $k_i=0$}
\end{cases}
\]

\item $V=${\bf the variety of all unars}. The language $\LL_u$ has a very simple set of terms, therefore we write only a compact form of  $\LL_u$-terms:
\[
f^n(x)=\begin{cases}
\underbrace{f(f(f(\ldots(x))\ldots))}_{\mbox{$n$-times}} \mbox{ if } n>0\\
x \mbox{ if} n=0
\end{cases}
\]
\end{enumerate}

\section{Equations and solutions}

An {\it equation over a language $\LL$ ($\LL$-equation)} is an equality of two $\LL$-terms.
\[
t(X)=s(X).
\]

\begin{remark}
According to model theory\Wiki, $\LL$-equation is actually an atomic\Wiki formula of $\LL$.
\end{remark}

A system of $\LL$-equations ($\LL$-system for shortness) is an arbitrary set of $\LL$-equations. {\it Notice that we shall consider only systems which depend on a finite set of variables $X$}. An $\LL$-system in variables $X$ is denoted by $\Ss(X)$. 

Let $\Acal$ be an $\LL$-algebra and $\Ss$ be an $\LL$-system in variables $X=\{x_1,x_2,\ldots,x_n\}$. A point $P=(p_1,p_2,\ldots,p_n)\in \Acal^n$ is called a {\it solution} of $\Ss$ if the substitution $x_i=p_i$ reduces each equation of $\Ss$ to a true equality of $\Acal$. Let us denote the set of all solutions of $\Ss$ in $\Acal$ by $\V_\Acal(\Ss(X))\subseteq \Acal^n$. 

If $\V_\Acal(\Ss)=\emptyset$ an $\LL$-system $\Ss$ is called \textit{inconsistent} over $\Acal$.

\begin{proposition}
For an $\LL$-algebra $\Acal$ and $\LL$-systems $\Ss_i$ it holds: 

\begin{enumerate}
\item $\Ss_1\supseteq\Ss_2\Rightarrow \V_\Acal(\Ss_1)\subseteq\V_\Acal(\Ss_2)$;
\item 
\[
\V_\Acal(\bigcup_{i\in I}\Ss_i)=\bigcap_{i\in I}\V_\Acal(\Ss_i)
\]
\end{enumerate}
\end{proposition}
\begin{proof}
Both statements immediately follows from a principle: ``a big system has a small solution set''. 
\end{proof}

An $\LL$-system $\Ss_1$  is {\it equivalent over} an $\LL$-algebra $\Acal$  to an $\LL$-system $\Ss_2$ if 
\[
\V_\Acal(\Ss_1)=\V_\Acal(\Ss_2),
\]
and it is denoted by $\Ss_1\sim_\Acal\Ss_2$.

\subsection{Simplifications of equations}
\label{sec:simplification_of_equations}

One can simplify equations in many classes of $\LL$-algebras. 

\begin{enumerate}
\item Any $\LL_g$-equation $t(X)=s(X)$ is equivalent over any group $G$ to
\begin{equation}
\label{eq:group_equation}
w(X)=1,
\end{equation}
where $w(X)=t(X)s^{-1}(X)$.

Respectively,  any $\LL_{+g}$-equation is equivalent over an abelian group $A$ to 
\begin{equation}
\label{eq:abelian_group_equation}
k_1x_1+k_2x_2+\ldots+ k_nx_n=0.
\end{equation}

\item By the definition, an equation over a semigroup $S$ is an equality 
\begin{equation}
\label{eq:semigroup_equation}
t(X)=s(X),
\end{equation}
where  $t(X),s(X)$ are semigroup terms~(\ref{eq:semigroup_term}). However, the equation~(\ref{eq:semigroup_equation}) does not always allow further simplifications. Nevertheless, for a commutative monoids with cancellations any $\LL_{+m}$-equation can be reduced to 
\begin{equation}
\label{eq:cancellation_semigroup_equation}
\sum_{i\in I}k_ix_i=\sum_{j\in J}k_jx_j,
\end{equation}
where $I,J\subseteq \{1,2,\ldots,n\}$, and $I\cap J=\emptyset$, i.e. there is not a variable occurring in both parts of equation.

\item Any $\LL_{u}$-equation over an unar $U$ is written as 
\begin{equation}
\label{eq:unar_equation}
f^n(x_i)=f^m(x_j),
\end{equation} 
где $x_i,x_j\in X$.

If a function $f$ is injective the equation~(\ref{eq:unar_equation}) can be reduced to an expression of the form
\begin{equation}
\label{eq:unar_equation_for_bijective_f}
f^k(x_i)=x_j.
\end{equation} 
\end{enumerate}

Let us give examples of equations and their solutions over some algebraic structures.

\begin{enumerate}
\item Let us consider an equation $xy=yx$ in the free semigroup $FS_n$ ($n>1$). By the properties of $FS_n$, elements $x,y\in FS_n$ commute iff $x,y$ are powers of the same element $w\in FS_n$. Thus,
\[
\V_{FS_n}(xy=yx)=\{(w^k,w^l)\mid w\in FS_r, k,l>0\}
\]
The solution set of  $xy=yx$ in the free monoid $FM_n$ ($n>1$) is
\[
\V_{FM_n}(xy=yx)=\{(w^k,w^l)\mid w\in FS_r, k,l\geq 0\},
\]
where the expression $w^0$ equals $1$. 

\item Let $S=RB(n,m)$ be a rectangular band. The solution set of an equation $xy=y$ in $S$ is the following set of pairs $\{([t,t_1],[t,t_2])\mid 1\leq t\leq n, 1\leq t_1,t_2\leq m\}$.

\end{enumerate}

\section{Algebraic sets and radicals}
\label{sec:algebraic_sets_and_radicals}

A set $Y\subseteq\Acal^n$ is called {\it algebraic} over an $\LL$-algebra $\Acal$ if there exists an $\LL$-system $\Ss$ in variables  $X=\{x_1,x_2,\ldots,x_n\}$ such that $\V_\Acal(\Ss)=Y$.

Let us give examples and properties of algebraic sets.

\begin{enumerate}
\item An {\it intersection of an arbitrary number of algebraic sets 
\[
Y=\bigcap_{i\in I}Y_i,\; Y_i\subseteq \Acal^n
\]
is algebraic over every $\LL$-algebra $\Acal$.} Indeed, if each $Y_i$ is the solution set of an $\LL$-system $\Ss_i$, then the system $\Ss=\bigcup_{i\in I}\Ss_i$ has the solution set $Y$.

\item {\it A set $Y=\Acal^n$ is algebraic over any $\LL$-algebra $\Acal$ and every natural $n>0$.} Indeed, the $\LL$-system $\Ss=\{x_i=x_i|1\leq i\leq n\}$ has the solution set $Y$.

\item Let $G$ be a group, and $Y_n=\{g\mid \mbox{the order\Wiki of $g$ divides $n$}\}\subseteq G$. The set $Y_n$ is algebraic, since $Y_n=\V_G(x^n=1)$. 

\item {\it Every algebraic set $Y\subseteq G^n$ over a group $G$ contains the point $e_n=(1,1,\ldots,1)$}, since any equation $w(X)=1$ satisfies $e_n$ over the group $G$. Thus, {\it the empty set is not algebraic over any group $G$.} 

\item The set of idempotents $E$ in an arbitrary semigroup $S$ is algebraic, since $E=\V_S(x^2=x)$.

\item Let $S$ be an idempotent semigroup (recall that left zero semigroups, semilattices and rectangular bands are idempotent). Let us prove that every algebraic set $Y\subseteq S^n$ over a semigroup $S$ contains the set of points $\{(a,a,\ldots,a)\mid a\in S\}$. Indeed, let us consider an $\LL_s$-equation $t(X)=s(X)$. The idempotency of $S$ gives $t(a,a,\ldots,a)=a$, $s(a,a,\ldots,a)=a$, therefore $(a,a,\ldots,a)\in\V_S(t(X)=s(X))$. Thus, the point $(a,a,\ldots,a)$ satisfies any  $\LL_s$-system $\Ss$, and every algebraic set $Y$ should contain all points of the form $(a,a,\ldots,a)$. Finally, we obtain that {\it the empty set is not algebraic set over any idempotent semigroup}.

\item Left zero semigroups $LZ_m$ have a simple structure, and we can explicitly describe the whole class of algebraic sets over $LZ_m$. According to Section~\ref{sec:simplification_of_terms_in_var}, any equation over $LZ_m$ is the equality of two variables $x_i=x_j$. Therefore, every $\LL_s$-system  $\Ss$ in variables $X=\{x_1,x_2,\ldots,x_n\}$ over $LZ_m$ is the set of equalities between variables of $X$. Thus, {\it for any pair of variables $x_i,x_j$ it holds that either the values of $x_i$ and $x_j$ coincide at any point of $\V_{LZ_m}(\Ss)$ or for all $a_i,a_j\in LZ_m$ the set $\V_{LZ_m}(\Ss)$ contains a point, where $x_i=a_i$, $x_j=a_j$}. 

Let $Y$ be a subset of $LZ_m^n$. Using the reasonings above, we give conditions for $Y$ to be algebraic. Let $\pi_{ij}(Y)\subseteq LZ_m^2$ denote the projection of the set $Y$ onto the $i$-th and $j$-th coordinates, i.e.
\[
(a_i,a_j)\in \pi_{ij}(Y)\Leftrightarrow (a_1,a_2,\ldots,a_n)\in Y.
\] 

Thus, {\it a set $Y\subseteq LZ_m^n$ is algebraic over $LZ_m$ iff for any pair $1\leq i\neq j \leq n$ the projection $\pi_{ij}(Y)$ is equal either $LZ_m^2$ or $\pi_{ij}(Y)=\{(a,a)\mid a\in LZ_m\}$.}

\item The empty set is algebraic over the free unar $FU_n$ of rank $n$, since there exists an inconsistent equation $f(x)=x$ over $FU_n$.

\end{enumerate}

\subsection{Radicals and congruent closures}

Let $Y\subseteq\Acal^n$ be an algebraic set over an $\LL$-algebra $\Acal$. The {\it radical} of $Y$ is defined as follows
\[
\Rad_\Acal(Y)=\{t(X)=s(X)| Y\subseteq \V_\Acal(t(X)=s(X))\}.
\]
The {\it radical of an $\LL$-system $\Ss$ over $\Acal$} is defined as the radical of the set $\V_\Acal(\Ss)$:  
\[
\Rad_\Acal(\Ss)=\Rad_\Acal(\V_\Acal(\Ss)).
\]
Equations from the radical $\Rad_\Acal(\Ss)$ are called {\it consequences} of an $\LL$-system $\Ss$. 
 
Obviously, $\Ss\subseteq \Rad_\Acal(\Ss)$ (a system $\Ss$ is a consequence of its equations) for any $\LL$-algebra $\Acal$. Moreover, any radical $\Rad_\Acal(\Ss)$ contains the following types of equations.
\begin{enumerate}
\item {\it All equations obtained from equations of $\Ss$ by the transitive and symmetric closure of the equality relation.} For example, if $t(X)=s(X)\in \Ss$ then $s(X)=t(X)\in\Rad_\Acal(\Ss)$, and $t(X)=s(X),s(X)=u(X)\in\Ss$ imply $t(X)=u(X)\in\Rad_\Acal(\Ss)$.
\item {\it All equations obtained from a functional composition of equations of $\Ss$}. For example, for an $\LL_{s}$-system $\Ss=\{t_1(X)=s_1(X),t_2(X)=s_2(X)\}$ of semigroup equations the equation $t_1(X)\cdot t_2(X)=s_1(X)\cdot s_2(X)$ belongs to the radical $\Rad_S(\Ss)$ over any semigroup $S$.
\end{enumerate}

Above we describe two types of elementary consequences of an $\LL$-system $\Ss$. Such equations form a main subset $[\Ss]$ of the radical $\Rad_\Acal(\Ss)$. Let us give a formal definition of the set $[\Ss]$.

A set of $\LL$-equations $[\Ss]$ is a \textit{congruent closure} of an $\LL$-system $\Ss$, if $[\Ss]$ is the least set containing $\Ss$ and satisfying the following axioms (below $t_i(X),s_i(X)$ are $\LL$-terms):
\begin{enumerate}
\item $t_1(X)=t_1(X)\in[\Ss]$ for every $\LL$-term $t_1(X)$;
\item if $t_1(X)=t_2(X)\in[\Ss]$ then $t_2(X)=t_1(X)\in[\Ss]$;
\item if $t_1(X)=t_2(X)\in[\Ss]$ and $t_2(X)=t_3(X)\in[\Ss]$ then $t_1(X)=t_3(X)\in[\Ss]$;
\item if $t_i(X)=s_i(X)\in[\Ss]$ ($1\leq i\leq n$) and $\LL$ contains an $n$-ary functional symbol $f^{(n)}$ then we have $f(t_1(X),t_2(X),\ldots,t_n(X))=f(s_1(X),s_2(X),\ldots,s_n(X))\in[\Ss]$. 
\end{enumerate}

It is easy to check that for any $\LL$-system $\Ss$ it holds the following inclusion
\begin{equation}
\label{eq:congruent_closure_vs_radical}
[\Ss]\subseteq\Rad_\Acal(\Ss). 
\end{equation}

The converse inclusion of~(\ref{eq:congruent_closure_vs_radical}) does not holds for many systems $\Ss$ and algebras $\Acal$, since the radical may contain nontrivial and unexpected equations. In the following examples we find some equations which do not belong to congruent closures of given systems.  

\begin{enumerate}
\item The radical of any system over an $\LL$-algebra $\Acal$ should contain all identities which are true in $\Acal$.

\item Let $\Zbb$ be the abelian group of integers. The radical of an $\LL_{+g}$-equation $nx=0$ ($n\in\N$) contains the consequence $x=0$. Obviously, in this example one can take an arbitrary abelian group with no torsion instead of $\Zbb$. 

\item Let $\N$ be the monoid of natural numbers. The radical of the equation $x+y=0$ contains $x=0,y=0$.

\item Let $FS_r$ be the free semigroup of rank $r$. Let us consider a $\LL_s$-system $\Ss=\{x_1x_2=x_2x_1,x_2x_3=x_3x_2\}$. By the properties of free semilattices, elements $x,y$ commute iff $x,y$ are powers of the same element $w$. Therefore, $x_1=w^{n_1}$, $x_2=w^{n_2}$, $x_3=w^{n_3}$ for some $w\in FS_r$, and we obtain that the equation $x_1x_3=x_3x_1$ belongs to the radical of $\Ss$. 

However in the free monoid $FM_r$ of rank $r>1$ the equation $x_1x_3=x_3x_1$ does not belong to the radical of  $\Ss$, since $(a_1,1,a_2)\in \V_{F_r}(\Ss)$, but $a_1a_2\neq a_2a_1$, where $a_i$ are free generators of $FM_r$.
 
\item The equation $x_1x_3=x_2x_3$ over the free monoid $FM_r$ implies $x_1=x_2\in\Rad_{FM_r}(x_1x_3=x_2x_3)$, since $FM_r$ has cancellations.

\item The equation $f(x)=f^2(y)$ over the free unar $FU_n$ deduces the equation $x=f(y)\in\Rad_{FU_n}(f(x)=f^2(y))$, since the unary function $f$ is injective over the set $FU_n$.
\end{enumerate}

The following statement contain the list of all elementary properties of radicals.
 
\begin{proposition}
Let $\Acal$ be an $\LL$-algebra.
\begin{enumerate}
\item For algebraic sets $Y_1,Y_2$ it holds 
\begin{equation}
Y_1\subseteq Y_2\Rightarrow\Rad_\Acal(Y_1)\supseteq\Rad_\Acal(Y_2).
\label{eq:about_radical_1}
\end{equation}
In particular,
\[
Y_1= Y_2\Leftrightarrow\Rad_\Acal(Y_1)=\Rad_\Acal(Y_2).
\]

\item For $\LL$-systems $\Ss_1,\Ss_2$ we have
\begin{equation}
\Ss_1\subseteq \Ss_2\Rightarrow\Rad_\Acal(\Ss_1)\subseteq\Rad_\Acal(\Ss_2).
\label{eq:about_radical_2}
\end{equation}

\item Suppose an algebraic set $Y$ is a union of algebraic sets $\{Y_i\mid i\in I\}$, then
\begin{equation}
\Rad_\Acal(Y)=\bigcup_{i\in I}\Rad_\Acal(Y_i).
\label{eq:about_radical_3}
\end{equation}

\end{enumerate}
\end{proposition}
\begin{proof}
The first statement follows from the informal principle ``a small set $Y$ has a big class of equations which satisfy all points of $Y$''. The second statement immediately follows from the first one.

Let us prove the third statement. The inclusion $\Rad(Y)\subseteq\bigcup_{i\in I}\Rad_\Acal(Y_i)$ follows from~(\ref{eq:about_radical_1}), therefore it is sufficient to prove the converse inclusion. Let $t(X)=s(X)$ be an $\LL$-equation from all radicals $\Rad_\Acal(Y_i)$ ($i\in I$). Therefore, we have $Y_i\subseteq\V_\Acal(t(X)=s(X))$, and hence $Y=\bigcup_{i\in I}Y_i\subseteq\V_\Acal(t(X)=s(X))$ or equivalently  $t(X)=s(X)\in\Rad_\Acal(Y)$. 

\end{proof}

\subsection{Irreducible algebraic sets}
\label{sec:irreducible_algebraic_sets}

A nonempty algebraic set $Y\subseteq\Acal^n$ is~\textit{irreducible} if $Y$ is not a finite union of algebraic sets
\[
Y_1\cup Y_2\cup\ldots\cup Y_k,\; Y_i\neq Y.
\]

Below we give examples of irreducible algebraic sets over different algebraic structures.

\begin{enumerate}

\item Let $\Acal=\lb A\mid\LL\rb$ be a finite $\LL$-algebra and $A=\{a_1,a_2,\ldots,a_n\}$. {\it Any algebraic set $Y=\Acal^m$ for $m>n$ is reducible}. Indeed, $Y$ is the solution set of a system $\Ss=\{x_1=x_1,x_2=x_2,\ldots,x_m=x_m\}$. Since the values of $x_1,x_2,\ldots,x_m$ belong to $\Acal$, there exists a pair $x_i,x_j$ with the same values at a given point. Thus, the set $Y$ is a union
\[
Y=\bigcup_{i\neq j}\V_\Acal(\Ss\cup\{x_i=x_j\}).
\] 

\item Let $L_n$ be a linear ordered semilattice and $n>1$. Then 
\[
L^2_n=\V_{L_n}(\{x=x,y=y\})=\V_L(x\leq y)\cup\V_L(y\leq x).
\]
Remark that the set $L_n=\V_{L_n}(x=x)$ is irreducible.

\item Let $S=RB(n,m)=\{(i,j)\mid 1\leq i\leq n,1\leq j\leq m\}$ be a rectangular band. {\it Then the algebraic set $Y=S^l=\V_S(\Ss)=\V_S(x_1=x_1,x_2=x_2,\ldots,x_l=x_l)$ is reducible for $l>\min(n,m)$}. Let us prove this statement for $\min(n,m)=n$. 

There exists a pair $x_i,x_j$ whose values have the same first coordinate in the presentation of $S$. Thus, $Y$ is a union
\[
Y=\bigcup_{1\leq i\neq j\leq n}\V_S(\Ss\cup\{x_ix_j=x_j\}).
\]

\item Let $FM_r$ be the free monoid of rank $r>1$. Then the solution set of $\Ss=\{xy=yx,\; yz=zy\}$ admits a presentation 
\[
\V_{FM_r}(\Ss)=\V_{FM_r}(\Ss\cup\{xz=zy\})\cup\V_{F_r}(\Ss\cup\{y=1\}).
\]
(see Section~\ref{sec:simplification_of_equations} for details).

Remark that $\Ss$ has the irreducible solution set over the free semigroup $FS_r$ (see Section~\ref{sec:irred_coord_properties} for the proof).

\end{enumerate}

There are algebraic structures, where every nonempty algebraic set is irreducible (following~\cite{uni_Th_IV}, such algebraic structures are called co-domains, and we study their properties in  Section~\ref{sec:co-domains}).

\section{Equationally Noetherian algebras}
\label{sec:noeth}

An $\LL$-algebra $\Acal$ is {\it equationally Noetherian} if any infinite $\LL$-system $\Ss$ is equivalent over $\Acal$ to some finite subsystem $\Ss^\pr\subseteq \Ss$.

\begin{remark}
By the definition, any infinite inconsistent system $\Ss$ over an equationally Noetherian $\LL$-algebra $\Acal$ should contain a finite inconsistent subsystem $\Ss^\pr$. However, the property ``any consistent system $\Ss$ is equivalent to some finite subsystem $\Ss^\pr\subseteq \Ss$'' is not equivalent to the definition of equationally Noetherian algebra (see the definition of the class $\Nbf_c$ in Section~\ref{sec:compactness_classes}). 
\end{remark}

For equationally Noetherian algebras we have the following criterion (its proof is easily deduced from the definition of equationally Noetherian algebras).

\begin{proposition}
An $\LL$-algebra $\Acal$ is equationally Noetherian iff for each $n>1$ the set $\Acal^n$ does not contain an infinite chain of algebraic sets 
\[
Y_1\supset Y_2\supset\ldots\supset Y_n\supset\ldots
\]
\label{pr:no_algebraic_chains_in_noeth}
\end{proposition}

In examples below of equationally Noetherian algebras we use the following obvious fact:  {\it any subalgebra $\Acal$ of an equationally Noetherian $\LL$-algebra $\Bcal$ is equationally Noetherian}. 
\begin{enumerate}

\item {\it Any finite $\LL$-algebra $\Acal$ is equationally Noetherian.} Indeed, the solution set of every $\LL$-system $\Ss$ is finite, so there is not any infinite chain of algebraic sets, and Proposition~\ref{pr:no_algebraic_chains_in_noeth} concludes the proof.

\item {\it Any $\LL$-algebra $\Acal$ from a locally finite variety $V$ is equationally Noetherian.} In this case the free algebra $F(X)$ in $V$ generated by $X=\{x_1,x_2,\ldots,x_n\}$ is finite. Since one can consider any $\LL$-term as an element of $F(X)$, the set of different terms is finite. Therefore, the number of all nonequivalent equations are also finite and there is not any infinite $\LL$-system over $\Acal$. 

\item By the previous item, all semilattices, rectangular bands and left zero semigroups are equationally Noetherian. 

\item {\it Each abelian group is equationally Noetherian.} Let us explain this statement by the following example. Let us consider an infinite $\LL_{+g}$-system 
\[
\Ss=\begin{cases}
x+2y=0,\\
3x+4y=0,\\
5x+6y=0,\\
\ldots\\
(2i+1)x+(2i+2)y=0,\\
\ldots
\end{cases}
\]
over an abelian group $A$. By the Gauss elimination process, one can reduce $\Ss$ to an upper triangular form as follows. Multiplying the first equations on $2i+1$ and subtracting the $i$-th equation from the first one, we obtain the following system of equations
\[
\Ss_1=\begin{cases}
x+2y=0,\\
2y=0,\\
4y=0,\\
\ldots\\
(2i)y=0,\\
\ldots
\end{cases}
\]  
Since there exist abelian groups where the equation $(2i)y=0$ do not imply $y=0$, the system $\Ss_1$ is not necessarily equivalent to  $\{x+2y=0,y=0\}$. However abelian group theory states that a system of the form $\Ss_x=\{n_1x=0,\; n_2x=0,\; n_3x=0,\ldots\}$ is equivalent to an equation $nx=0$, where $n$ is the greatest common divisor of  $n_1,n_2,n_3,\ldots$. By the properties of the greatest common divisor,  there exists a number $k$ such that $\mbox{gcd}(n_1,n_2,\ldots,n_k)=\mbox{gcd}(n_1,n_2,\ldots)$. Therefore the system $\Ss_x$ is equivalent to its finite subsystem $\{n_1x=0,\; n_2x=0,\; \ldots n_kx=0\}$. Thus, the system $\Ss_1$ is equivalent to its finite subsystem $\Ss_1^\pr=\{x+2y=0,2y=0\}$. Replacing all equations of $\Ss_2^\pr$ to the corresponding equations of the original system $\Ss$, we obtain a finite subsystem $\Ss^\pr=\{x+2y=0,\; 3x+4y=0\}\subseteq\Ss$ with $\Ss^\pr\sim_A\Ss$.

\item {\it Every commutative monoid $M$ with cancellations is equationally Noetherian.} By semigroup theory, such monoids are embedded into abelian groups. Since all abelian groups are equationally Noetherian (see above), each submonoid of an abelian group is also equationally Noetherian.

\item {\it Any group isomorphic to a group of invertible matrices relative to the operation of matrix multiplication} (following group theory, such groups are called linear) is equationally Noetherian. Let us prove this statement for a group $G$ of $2\times 2$-matrices with entries in a field $F$.

Let $\Ss$ be an arbitrary system in variables $x_1,x_2,\ldots,x_n$. Since elements of $G$ are $2\times 2$-matrices, we put  
\[
x_i=\begin{pmatrix}
x_i^{(1)}&x_i^{(2)}\\
x_i^{(3)}&x_i^{(4)}
\end{pmatrix}
\]
for some $x_i^{(j)}\in F$. The identity of $G$ is the identity matrix
\[
1=\begin{pmatrix}
1&0\\
0&1
\end{pmatrix}
\]

Replacing all variables in $\Ss$ to their matrix presentations and computing all multiplications and inverse elements, we obtain a system of polynomial equations $\Ss_{poly}$ over the field $F$. For example, the group equation $x_ix_j=1$ is equivalent to the following polynomial system
\begin{multline*}
x_ix_j=1\Rightarrow \begin{pmatrix}
x_i^{(1)}&x_i^{(2)}\\
x_i^{(3)}&x_i^{(4)}
\end{pmatrix}\begin{pmatrix}
x_j^{(1)}&x_j^{(2)}\\
x_j^{(3)}&x_j^{(4)}
\end{pmatrix}=
\begin{pmatrix}
1&0\\
0&1
\end{pmatrix}
\Rightarrow\\
\begin{pmatrix}
x_i^{(1)}x_j^{(1)}+x_i^{(2)}x_j^{(3)}&x_i^{(1)}x_j^{(2)}+x_i^{(2)}x_j^{(4)}\\
x_i^{(3)}x_j^{(1)}+x_i^{(4)}x_j^{(3)}&x_i^{(3)}x_j^{(2)}+x_i^{(4)}x_j^{(4)}
\end{pmatrix}=
\begin{pmatrix}
1&0\\
0&1
\end{pmatrix}
\Rightarrow
\begin{cases}
x_i^{(1)}x_j^{(1)}+x_i^{(2)}x_j^{(3)}=1,\\
x_i^{(1)}x_j^{(2)}+x_i^{(2)}x_j^{(4)}=0,\\
x_i^{(3)}x_j^{(1)}+x_i^{(4)}x_j^{(3)}=1,\\
x_i^{(3)}x_j^{(2)}+x_i^{(4)}x_j^{(4)}=0
\end{cases}
\end{multline*} 

According to commutative algebra (Hilbert`s basic theorem\Wiki), the system of polynomial equations  $\Ss_{poly}$ is equivalent over the field $F$ to its finite subsystem $\Ss^\pr_{poly}$. Applying inverse transformations to $\Ss_{poly}^\pr$, one can obtain a system of group equations $\Ss^\pr$ with $\Ss^\pr\sim_G\Ss$.

\item Let us prove that the free monoids $FM_n$ are equationally Noetherian. We need the following facts about the theory of the free group\Wiki. It is known that the free group $FG_n$ of rank $n$ contains  (as a submonoid) the free monoid of rank $n$, and moreover all free groups $FG_n$ are embedded into the free group $FG_2$ of rank $2$. Thus, $FG_2$ contains a submonoid $FM_n$ of an arbitrary finite rank $n$ and it is sufficient to prove that $FG_2$ is equationally Noetherian.

By the Sanov theorem\Wiki there exists a presentation of the free generators $a_1,a_2$ of $FG_2$ by the following matrices
\[
\begin{pmatrix}
1&2\\
0&1
\end{pmatrix},
\;
\begin{pmatrix}
1&0\\
2&1
\end{pmatrix}.
\]
Thus, the group $FG_2$ is linear and  by the previous statement $FG_2$ is equationally Noetherian. Remark that we actually prove that all free groups are equationally Noetherian.

\item Let us show that the free unar $FU_n$ of rank $n$ is equationally Noetherian. Since any equation over $FU_n$ contains at most two variables, every system  $\Ss$ in variables $X=\{x_1,x_2,\ldots,x_m\}$ is a union
\[
\Ss=\bigcup_{1\leq i\leq j\leq m}\Ss_{ij},
\] 
where the system $\Ss_{ij}$ depends only on the variables $x_i,x_j$. It is sufficient to prove that each $\Ss_{ij}$ is equivalent to a finite subsystem. 

According to~(\ref{eq:unar_equation_for_bijective_f}), each equation of $\Ss_{ij}$ has the form $f^k(x_i)=x_j$. Suppose $S_{ij}$ is infinite. Let us consider all possible cases if $\Ss_{ij}$ contains the following equations.
\begin{enumerate}
\item If $f^k(x_i)=x_j,f^l(x_i)=x_j\in\Ss_{ij}$ ($k<l$) then 
\[
f^k(x_i)=f^l(x_i)\in\Rad_{FU_n}(\Ss_{ij})\Leftrightarrow f^{l-k}(x_i)=x_i\in\Rad_{FU_n}(\Ss_{ij}).
\]
 
\item  If $f^k(x_i)=x_j,x_i=f^l(x_j)\in\Ss_{ij}$ ($k<l$) we have 
\[
f^{k+l}(x_j)=x_j\in\Rad_{FU_n}(\Ss_{ij}).
\]
\end{enumerate} 
In both cases the last equation is inconsistent in the free unar, $\Ss_{ij}$ is also inconsistent and it is equivalent to the given two equations.

\end{enumerate}

There are algebras which are not equationally Noetherian.
\begin{enumerate}
\item Following~\cite{harju_compactness_property}, let us consider a monoid with the presentation $\Bcal=\lb a,b\mid ab=1\rb$. The monoid $\Bcal$ is called {\it the bicyclic}, and each element $s\in \Bcal$ is a product of the form $s=b^na^m$ ($m,n\in\N$). The multiplication in $\Bcal$ is defined as follows
\[
b^{n_1}a^{m_1}b^{n_2}a^{m_2}=
\begin{cases}
b^{n_1+n_2-m_1}a^{m_2},\mbox{ if }m_1\leq n_2\\
b^{n_1}a^{m_1-n_2+m_2},\mbox{ otherwise }
\end{cases}
\] 

The elements $b^na^n$ are idempotents of $\Bcal$:
\[
(b^na^n)(b^na^n)=b^n(a^nb^n)a^n=b^n1a^n=b^na^n.
\]

Let us consider an infinite system of equations 
\[
\Ss=\begin{cases}
xyz=z,\\
x^2y^2z=z,\\
x^3y^3z=z,\\
\ldots
\end{cases}
\] 
Let us assume that  $\Ss$ is equivalent over $\Bcal$ to a finite subsystem  $\Ss^\pr$. Without loss of generality, we put $\Ss^\pr=\{x^iy^iz=z\mid 1\leq i \leq n\}$. The system $\Ss^\pr$ satisfies the point $(b,a,b^na^n)$:
\[
b^ia^ib^na^n=b^i(a^ib^n)a^n=b^ib^{n-i}a^n=b^na^n.
\] 
However $(b,a,b^na^n)\notin \V_\Bcal(x^{n+1}y^{n+1}z=z)$:
\[
b^{n+1}a^{n+1}b^na^n=b^{n+1}(a^{n+1}b^n)a^n=b^{n+1}a^{1}a^n=b^{n+1}a^{n+1}\neq b^na^n.
\] 
Thus, $(b,a,b^na^n)\notin \V_\Bcal(\Ss)$, and we obtain $\Ss\nsim_\Bcal\Ss^\pr$, a contradiction.

\item Let $\LL_{u\infty}=\{f_i^{(1)}\mid i\in\N\}$ be a language of a countable number of unary functional symbols, and $\Acal$ be an $\LL_{u\infty}$-algebra with the universe $\N$ (natural numbers) and the functions $f_i$ are defined in $\Acal$ as follows
\[
f^\Acal_i(x)=\begin{cases}
0,\mbox{ if $x=i$ }\\
x,\mbox{ otherwise}
\end{cases}
\] 

Let us consider an infinite system $\Ss=\{f_i(x)=x\mid i\in\N\}$. By the definition of the functions $f_i$ it follows $\V_\Acal(\Ss)=\{0\}$. However, any finite subsystem $\Ss^\pr=\{f_i(x)=x\mid 1\leq i\leq n\}$ has the nonempty solution set  $\V_\Acal(\Ss^\pr)=\{i\mid i>n\}$. 

\item Following \cite{harju_compactness_property}, any finitely generated non-hopfian\Wiki group is not equationally Noetherian. In particular, the well-know Baumslag-Solitar group $B(2,3)=\lb a,b\mid a^{-1}b^2a=b^3 \rb$ is not equationally Noetherian.
\end{enumerate}

For the class of equationally Noetherian algebras we have the following statements.

\begin{enumerate}
\item {\it Any subalgebra of an equationally Noetherian $\LL$-algebra $\Acal$ is also equationally Noetherian} (we used this property above many times).
\item {\it An arbitrary direct power of an equationally Noetherian $\LL$-algebra $\Acal$ is equationally Noetherian.} 
\item  {\it A finite direct product of equationally Noetherian algebras  is also equationally Noetherian.}
\end{enumerate}

The class of equationally Noetherian algebras is not closed under the following operations.
\begin{enumerate}
\item {\it An infinite direct product of equationally Noetherian algebras is not always equationally Noetherian.} In~\cite{lothaire} it was defined an example of infinite direct product $S=\prod_i S_i$ of equationally Noetherian semigroups $S_i$ such that $S$ is not equationally Noetherian. Below we consider a simpler example of a direct product of linear ordered semilattices in a non-standard language.

Let $\LL_{s\infty}=\{\cdot,c_1,c_2,\ldots,c_n,\ldots\}$ be a semigroup language extended by a countable set of constants $c_i$. Let $\Acal_n=\lb\{1,2,\ldots,n\}\mid\LL\rb$ be an $\LL_{s\infty}$-algebra such that $x\cdot y=\min\{x,y\}$ and the constants are defined in $\Acal$ as follows
\[
c_i^{\Acal_n}=\begin{cases}
i\mbox{ if } i\leq n\\
n\mbox{ otherwise }
\end{cases}
\]
Obviously, $\Acal_n$ is a linear ordered semilattice. Since all $\Acal_n$ are finite, they are equationally Noetherian. 

Let $\Acal$ denote the direct product $\prod_{i=1}^\infty\Acal_i$. By the definition of a direct product, a constant symbol $c_n$ has an interpretation in $\Acal$ as follows
\[
c_n^\Acal=(1,2,3,\ldots,n,n,\ldots,n,\ldots).
\]  
According to the definition of a partial order, we have the following chain of constants
\[
c_1^\Acal<c_2^\Acal<\ldots<c_n^\Acal<\ldots
\]
     
Let us consider an infinite system $\Ss=\{x\geq c_n^\Acal\mid n\geq 1\}$ over $\Acal$. Let $\Ss^\pr$ be a finite subsystem of $\Ss$. Without loss of generality one can assume that $\Ss^\pr=\{x\geq c_n\mid 1\leq n\leq N\}$. It is directly checked that the point $c_N^\Acal$ is a solution of $\Ss^\pr$, however $c_N^\Acal\ngeq c_{N+1}^\Acal$, i.e. $c_N^\Acal\notin\V_\Acal(\Ss)$. Thus,  $\Ss^\pr\nsim_\Acal\Ss$, and $\Acal$ is not equationally Noetherian.  

\item {\it A homomorphic image of an equationally Noetherian algebra is not always equationally Noetherian.} The Baumslag-Solitar group $B(2,3)$ is a homomorphic image of the equationally Noetherian free group $FG_2$ of rank $2$ (by group theory, any two-generated group is a homomorphic image of $FG_2$), but above we showed $B(2,3)$ is not equationally Noetherian. 

\end{enumerate}

An algebraic set $Y$ over an $\LL$-algebra $\Acal$ can be decomposed into a union of other algebraic sets $Y=Y_1\cup Y_2\cup\ldots\cup Y_n$, and one can iterate this process for each  $Y_i$, ets... This process may be finite (if we shall obtain a decomposition of $Y$ into a finite union of irreducible sets) or infinite. The equationally Noetherian property provides the finiteness of the process.

\begin{theorem}
If an $\LL$-algebra $\Acal$ is equationally Noetherian, then any algebraic set $Y$ over $\Acal$ admits a decomposition 
\begin{equation}
Y=Y_1\cup Y_2\cup\ldots\cup Y_m,
\label{eq:irreducible_decomposition}
\end{equation}
where $Y_i$ are irreducible, $Y_i\nsubseteq Y_j$ ($i\neq j$), and this decomposition is unique up to permutations of $Y_i$.
\label{th:irreducible_decomposition}
\end{theorem}

The algebraic sets $Y_i$ from Theorem~\ref{th:irreducible_decomposition} are called {\it irreducible components} of $Y$.

\begin{remark}
Let us prove that the converse statement of Theorem~\ref{th:irreducible_decomposition} does not hold. It is sufficient to define an $\LL$-algebra $\Bcal$ such that $\Bcal$ is not equationally Noetherian and {\it all} nonempty algebraic sets over $\Bcal$ are irreducible (in this case the decomposition~(\ref{eq:irreducible_decomposition}) obviously holds for $\Bcal$).

Algebras with no nonempty reducible sets are called co-domains (see Section~\ref{sec:co-domains}). Thus, $\Bcal$ should be a co-domain. According to the results of Section~\ref{sec:co-domains}, an infinite direct power $\Bcal=\prod_{i=1}^\infty \Acal$ of an arbitrary $\LL$-algebra $\Acal$ is co-domain. Taking a non-equationally Noetherian algebra $\Acal$, we obtain $\Bcal$ with required properties.
\end{remark}

\subsection{Comments}

Firstly, let us show the terminology difference. In B. Plotkin`s papers~\cite{Plot1,Plot2,Plot3} equationally Noetherian algebras are called {\it logically Noetherian}, and in~\cite{harju_c_s_semigroups,harju_compactness_property,lothaire} the equationally Noetherian property is called the Compacteness property.

In~\cite{uni_Th_II} there is a list of equationally Noetherian algebras. For example, the following algebras are equationally Noetherian  
\begin{itemize}
\item any Noetherian commutative ring $R$ (a ring $R$ is called Noetherian if it does not contain infinite ascending chains of ideals);
\item any torsion-free hyperbolic group~\cite{Sela3};
\item any free solvable group~\cite{GRom};
\item any finitely generated metabelian (nilpotent) Lie algebra~\cite{Daniyarova1};
\item since any linear group is equationally Noetherian (see above), any polycyclic and any finitely generated metabelian group is equationally Noetherian~\cite{BMR1,
Bryant, Guba}.
\end{itemize}
Moreover, the paper~\cite{RomSh} is devoted to the Noetherian property of the universal enveloping algebras; in~\cite{harju_c_s_semigroups} there were described equationally Noetherian completely simple semigroups.

On the other hand, the following algebras are not equationally Noetherian (see~\cite{uni_Th_II} for the complete list of such algebras):
\begin{itemize}
\item infinitely generated nilpotent groups~\cite{MR2};
\item the wreath product $A\wr B$ of a non-abelian group $A$ and infinite group $B$~\cite{BMRom};
\item the min-max algebras  $\mathcal{M}_\mathbb{R}=\langle \mathbb{R}; {\max},
{\min}, {\cdot}, \allowbreak {+}, {-}, 0, 1\rangle$ and
$\mathcal{M}_\mathbb{N}=\langle \mathbb{N}; {\max}, {\min},
\allowbreak {+}, 0, 1\rangle$~\cite{DvorKot};
\item finitely generated semigroups with an infinite descending chains of idempotents (in particular, all free inverse semigroups)~\cite{harju_compactness_property}. 
\end{itemize}

One can recommend the book~\cite{lothaire} which contains many examples of (non-) equationally Noetherian semigroups.

Nevertheless there are many open problems in this area. Let us formulate only two of them.

\noindent{\bf Problem.} Is the free product of two equationally Noetherian groups equationally Noetherian?

\noindent{\bf Problem.} Is the free anti-commutative (respectively, the free associative algebra, the free Lie algebra) is equationally Noetherian?

\section{Coordinate algebras}
\label{sec:coord_algebras}

By the definition, the set of all $\LL$-terms is closed under compositions of all functional symbols of the language $\LL$. Thus, the set of all $\LL$-terms in variables $X$ is an $\LL$-algebra which is denoted by $\Tcal_\LL(X)$ and called the {\it the termal $\LL$-algebra generated by  $X$}.

Let $Y$ be an algebraic set over an $\LL$-algebra $\Acal$ and $Y$ is defined by an $\LL$-system in variables $X$. Let us define an equivalence relation  $\sim_Y$ over $\Tcal_\LL(X)$ as follows:
\[
t(X)\sim_Y s(X)\Leftrightarrow t(P)=s(P) \mbox{ for each point $P\in Y$}\Leftrightarrow t(X)=s(X)\in\Rad_\Acal(Y).
\] 

Let $[t(X)]_Y$ denote the equivalence class of a term $t(X)$ relative to the relation $\sim_Y$.
The family  of all equivalence classes 
\[
\Gamma_\Acal(Y)=\Tcal(X)/\sim_Y
\]
is called {\it the coordinate $\LL$-algebra of an algebraic set $Y$} (precisely, $\Gamma_\Acal(Y)$ is a factor-algebra\Wiki of $\Tcal(X)$ by the relation $\sim_Y$). The result of an $n$-ary function $f$ over an $\LL$-algebra  $\Gamma_\Acal(Y)$ is defined as follows
\[
f([t_1(X)]_Y,[t_2(X)]_Y,\ldots,[t_n(X)]_Y)=[f(t_1(X),t_2(X),\ldots,t_n(X))]_Y.
\]
For example, if $\LL$ contains a binary functional symbol $\cdot$, the result of the  function $\cdot$ in $\Gamma_\Acal(Y)$ is $[t(X)]_Y\cdot[s(X)]_Y=[t(X)\cdot s(X)]_Y$. 

\begin{remark}
Let $I$ be the set of all identities which hold in an $\LL$-algebra $\Acal$. Since the radical of any algebraic set $Y\subseteq\Acal^n$ contains all identities from $I$, the coordinate algebra of $Y$ satisfies all identities of $I$  $\Acal$. In particular, the coordinate algebra of an algebraic set over a group (semigroup, unar) is always a group (respectively, semigroup, unar).  
\label{rem:coord_alg_from_the_same_variety}
\end{remark}

According to the definition, the coordinate algebra of an algebraic set $Y\subseteq\Acal^n$ defined by a system in variables $X$ admits a presentation 
\begin{equation}
\label{eq:coord_algebra_presentation}
\Gamma_\Acal(Y)=\lb X\mid \Rad_\Acal(Y)\rb.
\end{equation}

By~(\ref{eq:coord_algebra_presentation}) it follows that any coordinate algebra is finitely generated.

Let us compute coordinate algebras for some algebraic sets.

\begin{enumerate}
\item Suppose $Y=\emptyset$ is an algebraic set over an $\LL$-algebra $\Acal$, and $Y$ is defined by an $\LL$-system $\Ss$ in variables $X$. Obviously, {\it arbitrary terms $t(X),s(X)$ has equal values at each point of the empty set}. It follows that the relation $\sim_\emptyset$ defines a unique equivalence class in $\Tcal_\LL(X)$. Therefore, the factor-algebra $\Gamma_\Acal(\emptyset)$ consists of a single element. In other words, the algebra $\Gamma_\Acal(\emptyset)$ is {\it trivial}.

\item Let us consider a system of equations $\Ss=\{2x+z=y+u,x+u=3y+z\}$ over the commutative monoid of natural numbers $\N$ and denote $Y=\V_\N(\Ss)$. Subtracting from the first equation the second one, we obtain the equation $x+2y=0$ which is obviously belongs to the radical of $\Ss$. From the last equation it follows $x=0,y=0\in\Rad_\N(Y)$. Applying $x=0,y=0$, to the equations of  $\Ss$, we have $z=u\in\Rad_\Acal(Y)$. Since the substitutions $x:=0,y:=0, z:=u$ reduce the system $\Ss$ to trivial equalities, the radical of $Y$ is generated by the equations $x=0, y=0, z=u$. Therefore 
\[
\Gamma_\N(Y)\cong\lb x,y,z,u\mid x=0, y=0, z=u\rb\cong\lb u\rb\cong\N
\]

\item Let $A$ be a torsion-free abelian group, and $Y$ be an arbitrary algebraic set over $A$. By Remark~\ref{rem:coord_alg_from_the_same_variety}, the coordinate group $\Gamma_A(Y)$ is an abelian group. Suppose that $\Gamma_A(Y)$ has a nontrivial torsion, i.e. there exists a nonzero element $\gamma\in\Gamma_A(Y)$ with $n\gamma=0$ for some $n\in\N$. By the definition of a coordinate group, there exists an $\LL_{+g}$-term $t(X)$ such that $\gamma=[t(X)]_Y$ and  $nt(X)=0\in\Rad_A(Y)$. Since the equality $nt(X)=0$ implies $t(X)=0$ in any torsion-free group, we have $\gamma=[0]$, a contradiction. 

Thus, we proved that all coordinate groups over an abelian torsion-free group $A$ have not torsion subgroups. Since the coordinate group $\Gamma_A(Y)$ is always finitely generated, the group $\Gamma_A(Y)$ is isomorphic to $\Zbb^n$ for some $n\in \N$ (here we apply the decomposition~(\ref{eq:abelian_group_presentation})).

\begin{proposition}
A group $G$ is the coordinate group of an algebraic set over an abelian torsion-free group iff $G$ is isomorphic to $\Zbb^n$. 
\label{pr:coord_groups_over_torsion_free_group}
\end{proposition}
\begin{proof}
The ``only if'' statement was proved above. Let us prove the ``if'' part of the theorem: one should prove that $\Zbb^n$ is isomorphic to a coordinate group of some algebraic set over $A$. Let us consider an algebraic set $Y=A^n=\V_A(\{x_1=x_1,x_2=x_2,\ldots,x_n=x_n\})$. Its coordinate group is 
\[
\Gamma_A(Y)=\lb x_1,x_2,\ldots,x_n\mid x_i=x_i\rb\cong\lb x_1,x_2,\ldots,x_n\rb\cong\Zbb^n.
\]
\end{proof}

\item Let $FM_r$ be the free monoid of rank $r>1$. The coordinate monoid of the algebraic set  $Y=\V_{FM_r}(xy=yx)$ is the monoid with two generators $x,y$ and $x$ commutes with $y$. Thus, $\Gamma_{FM_r}(Y)$ is isomorphic to the commutative free monoid of rank $2$: $\Gamma_{FM_r}(Y)\cong\N^2$.

\item Let us consider the equation $f(x)=f^2(y)$ over the free unar $FU_r$. Since the function $f$ is injective on $FU_r$, we have $x=f(y)\in\Rad_{FU_r}(f(x)=f^2(y))$. Therefore the coordinate unar of the algebraic set $Y=\V_{FU_r}(f(x)=f^2(y))$ has the presentation
\[
\Gamma_{FU_n}(Y)\cong \lb x,y\mid x=f(y)\rb\cong\lb y\rb\cong FU_1.
\]        
\end{enumerate}

{\it There exist distinct algebraic sets $Y\subseteq\Acal^n,Z\subseteq\Acal^m$ with isomorphic coordinate algebras.} Indeed, the algebraic sets
\[
Y=\Acal=V_\Zbb(x=x),\;Z=\V_\Zbb(\{x=x,y=x\})\subseteq \Zbb^2
\]
over the abelian group $\Zbb$ has the following coordinate groups
\[
\Gamma_\Zbb(Y)=\lb x\mid \Rad_\Zbb(x=x)\rb\cong\Zbb.
\]
\[
\Gamma_\Zbb(Z)=\lb x,y\mid \Rad_\Zbb(x=x,y=x)\rb\cong\lb x\mid \Rad_\Zbb(x=x)\rb\cong \Zbb.
\]  

Algebraic sets with isomorphic coordinate algebras were described in~\cite{uni_Th_I}. Precisely, it was defined the notion of an {\it isomorphism of algebraic sets} (see~\cite{uni_Th_I} for more details) and proved the following statement.

\begin{proposition}
Algebraic sets $Y\subseteq\Acal^n,Z\subseteq\Acal^m$ have isomorphic coordinate algebras iff the sets $Y,Z$ are isomorphic. Moreover, an irreducible algebraic set is isomorphic only to irreducible algebraic sets.
\label{pr:when_coord_alg_are_isomorphic} 
\end{proposition}

The coordinate algebra of an irreducible algebraic set is called {\it irreducible}. By Proposition~\ref{pr:when_coord_alg_are_isomorphic}, this definition is well-defined.

\section{Main problems of universal algebraic geometry}

Now we are able to explain the problems studied in universal algebraic geometry. For a fixed $\LL$-algebra $\Acal$ one can pose the following problems.

\begin{enumerate}
\item Describe all algebraic sets over $\Acal$. When a given set $Y\subseteq\Acal^n$ is algebraic over $\Acal$? What are irreducible sets over $\Acal$? 
\item Describe all coordinate algebras over $\Acal$. When a given $\LL$-algebra $\Bcal$ is isomorphic to the coordinate algebra of some algebraic set over $\Acal$? What are irreducible coordinate algebras over  $\Acal$? 
\end{enumerate}   

Actually, the first problem is complicated and it is completely solved only for algebras $\Acal$ of a simple structure (for example, in Section~\ref{sec:algebraic_sets_and_radicals} we completely describe algebraic sets over left zero semigroups). 

The second problem is simpler than the first one, since the description of coordinate algebras gives the description of algebraic sets up to isomorphism. However, this problem can be solved for a wide class of algebras, and in the next sections we shall study only coordinate algebras.

\section{Properties of coordinate algebras}
\label{sec:coord_alg_properties}

Let $\Gamma=\Gamma_\Acal(Y)$ be the coordinate algebra of an algebraic set over an $\LL$-algebra $\Acal$, and $\phi\colon\Gamma\to \Acal$ be a homomorphism between $\Gamma$ and $\Acal$. By~(\ref{eq:coord_algebra_presentation}), we may assume that $\Gamma$ has a presentation $\lb X\mid \Rad_\Acal(Y)\rb$, where $X=\{x_1,x_2,\ldots,x_n\}$ is the generating set of $\Gamma$. The homomorphism $\phi$ should map each $x_i$ into an element of $\Acal$, and the vector $(x_1,x_2,\ldots,x_n)$ is mapped into a point $P=(\phi(x_1),\phi(x_2),\ldots,\phi(x_n))\in\Acal^n$. Since $\phi$ should preserve all relation $\Rad_\Acal(Y)$ of $\Gamma$, we obtain that the point $P$ should satisfy all equations from $\Rad_\Acal(Y)$ and therefore $P\in Y$. Thus, we proved that any homomorphism between $\LL$-algebras $\Gamma,\Acal$ is actually a substitution $x_i\mapsto p_i$ of coordinates of some point $P=(p_1,p_2,\ldots,p_n)\in Y$. This fact can be written in the following form 
\begin{equation}
\phi([t(X)]_Y)=t(P)
\label{eq:hom_of_substitution}
\end{equation}

A homomorphism $\phi\colon\Bcal\to\Acal$ of the type~(\ref{eq:hom_of_substitution}) is denoted below by $\phi_P$.

An $\LL$-algebra $\Acal$ {\it approximates} an $\LL$-algebra $\Bcal$ if for an arbitrary pair of distinct elements $b_1,b_2\in\Bcal$ there exists a homomorphism $\phi\colon\Bcal\to\Acal$ such that $\phi(b_1)\neq\phi(b_2)$.

\begin{theorem}
A finitely generated $\LL$-algebra $\Bcal$ is the coordinate $\LL$-algebra of a nonempty algebraic set  $Y$ over an $\LL$-algebra $\Acal$ iff $\Bcal$ is approximated by $\Acal$.
\label{th:coord_iff_approx}
\end{theorem}
\begin{proof}
Let us prove the ``only if'' part of the theorem. Suppose $\Bcal\cong\Gamma_\Acal(Y)$ for some algebraic set $Y$ over $\Acal$, and $[t(X)]_Y,[s(X)]_Y$ are distinct elements of $\Bcal$. Since the terms $t(X),s(X)$ are not $\sim_Y$-equivalent, there exists a point $P\in Y$ with $t(P)\neq s(P)$. Thus, the homomorphism $\phi_P$ gives $\phi_P([t(X)]_Y)\neq \phi_P([s(X)]_Y)$, and the algebra $\Bcal$ is approximated by $\Acal$.

Let us prove the ``if'' part of the statement. Suppose that a finitely generated $\LL$-algebra $\Bcal$ has a presentation $\lb X\mid R\rb$, where $X=\{x_1,x_2,\ldots,x_n\}$ is the set of generators and  $R=\{t_i(X)=s_i(X)\mid i\in I\}$ are defining relations. One can consider $R$ as an $\LL$-system over an $\LL$-algebra $\Acal$. 

The images of homomorphisms $\phi\in \Hom(\Bcal,\Acal)$ are the subset in $\Acal^n$:
\[
Y=\{(\phi(x_1),\phi(x_2),\ldots,\phi(x_n))\mid \phi\in\Hom(\Bcal,\Acal)\}.
\]  

A map $\phi\colon\Bcal\to\Acal$ is a homomorphism of $\Bcal$ iff $\phi$ preserves all defining relations $R$, or equivalently  $(\phi(x_1),\phi(x_2),\ldots,\phi(x_n))\in \V_\Acal(R)$. Thus, $Y=\V_\Acal(R)$, i.e. $Y$ is an algebraic set. It is sufficient to prove that $\Bcal$ is the coordinate algebra of the set $Y$. According to~(\ref{eq:coord_algebra_presentation}), one should prove the equality of the sets $[R]=\Rad_\Acal(Y)$.

Since $Y=\V_\Acal(R)$, for any $t(X)=s(X)\in [R]$ it holds $Y\subseteq\V_\Acal(t(X)=s(X))$, therefore $t(X)=s(X)\in\Rad_\Acal(Y)$. Let us prove the converse inclusion $\Rad_\Acal(Y)\subseteq [R]$. Assume there exists an equation $t(X)=s(X)\in\Rad_\Acal(Y)\setminus[R]$. It follows that the elements $t(X),s(X)$ are distinct in the algebra $\Bcal$, but $t(P)=s(P)$ for all $P\in Y$. By the definition of the set $Y$, we have $\phi(t(X))=\phi(s(X))$ for any homomorphism $\phi\colon\Bcal\to\Acal$. Thus, we obtain a contradiction, since $\Bcal$ is approximated by $\Acal$. 
\end{proof}

Let us apply Theorem~\ref{th:coord_iff_approx} to various algebras.
\begin{enumerate}
\item Let $\Acal=\Zbb_{p^n}$ be the cyclic group of order $p^n$ for a prime $p$. Let us find all  coordinate groups of algebraic sets over $\Acal$. The list of nontrivial subgroup of $\Acal$ is $Sub(\Acal)=\{\Zbb_{p^m}\mid 1\leq m\leq n\}$. Obviously, each group from $Sub(\Acal)$ is approximated by $\Acal$, therefore all groups $\{\Zbb_{p^m}\mid m\leq n\}$ are coordinate groups of algebraic sets over $\Acal$.

Let is prove that any direct sum of groups from $Sub(\Acal)$ is also a coordinate group over $\Acal$. Let 
\[
\Bcal=\bigoplus_{i=1}^m\Zbb_{p^{n_i}},
\]    
where $1\leq n_i\leq n$, and we allow $n_i=n_j$ for distinct $i,j$. Therefore the elements $b,b^\pr\in \Bcal$, $b\neq b^\pr$ admit unique presentations as the following sums
\[
b=b_1+b_2+\ldots+b_m,\; b^\pr=b^\pr_1+b^\pr_2+\ldots+b^\pr_m,
\]
where $b_i,b^\pr_i\in \Zbb_{p^{n_i}}$. Since $b\neq b^\pr$, there exists a number $i$ such that $b_i\neq b_i^\pr$. Then the projection $\pi_{i}$ of $\Bcal$ onto the subgroup $\Zbb_{p^{n_i}}\subseteq\Acal$ gives $\pi_i(b_i)\neq\pi_i(b_i^\pr)$, and $\Bcal$ is approximated by $\Acal$. 

Thus, we proved that the group $\Bcal$ is the coordinate group of some algebraic set $Y$ over $\Acal$. Moreover, one can write a system of equations $\Ss$ such that $\Gamma_\Acal(\V_\Acal(\Ss))\cong \Bcal$.
Indeed, for
\[
\Ss=\bigcup_{i=1}^m\{p^{n_i}x_i=0\},
\] 
the coordinate group of the set  $Y=\V_\Acal(\Ss)$ has the presentation
\[
\Gamma_\Acal(Y)=\lb x_1,x_2,\ldots,x_m\mid p^{n_i}x_i=0\rb \; (1\leq i\leq m),
\]
and we obtain $\Gamma_\Acal(Y)\cong \Bcal$.

Let us consider groups which are not coordinate groups over $\Acal=\Zbb_{p^n}$ (below we shall use the results of Section~\ref{sec:subalgebras_direct_powers_hom} about homomorphisms of abelian groups).
\begin{enumerate}
\item Let $\Bcal=\Zbb$. Any homomorphism $\phi\in\Hom(\Bcal,\Acal)$ maps the element $p^n\in\Zbb$ in $0\in\Zbb_{p^n}$, hence the group $\Acal$ does not approximate $\Bcal$ (there is not a homomorphism  $\psi\in\Hom(\Bcal,\Acal)$ with $\psi(0)\neq\psi(p^n)$).   
\item Let $\Bcal=\Zbb_{q^m}$, where $q$ is prime and $q\neq p$. The set of homomorphisms $\Hom(\Bcal,\Acal)$ contains only the trivial homomorphism $\phi(x)=0$ for all $x\in \Zbb_{q^m}$. Obviously, such group $\Bcal$ is not approximated by $\Acal$.   
\item Suppose $\Bcal=\Zbb_{p^m}$, where $m>n$. Each homomorphism $\phi\in\Hom(\Bcal,\Acal)$ maps $p^n\in\Zbb_{p^m}$ in the zero of $\Zbb_{p^n}$. Thus, the group $\Acal$ does not approximate $\Bcal$,  since there is not a homomorphism $\psi\in\Hom(\Bcal,\Acal)$ with $\psi(0)\neq\psi(p^n)$.

\end{enumerate}

The obtained results allow us to formulate the following proposition.

\begin{proposition}
\label{pr:coord_groups_for_cyclic_group}
Let $Sub(\Acal)=\{\Zbb_{p^k}\mid 1\leq k\leq n\}$ be the set of all nonzero subgroups of the group $\Acal=\Zbb_{p^n}$ for a prime $p$. An abelian group $\Bcal$ is the coordinate group of a nonempty algebraic set over $\Acal$ iff $\Bcal$ is a direct sum of a finite number of groups from $Sub(\Acal)$.
\end{proposition}

In Theorem~\ref{th:coord_abelian_groups} we shall prove Proposition~\ref{pr:coord_groups_for_cyclic_group} for any finitely generated abelian group $\Acal$.

\item Let $\Acal=LZ_n=\{a_1,a_2,\ldots,a_n\}$ be a left zero semigroup consisting of $n>1$ elements. Let us show that an \textit{arbitrary} finite left zero semigroup $\Bcal=LZ_m=\{b_1,b_2,\ldots,b_m\}$ is the coordinate semigroup of some nonempty algebraic set over $\Acal$. Let $b_i,b_j$ be two distinct elements of $\Bcal$. Then the homomorphism 
\[
\phi(b)=\begin{cases}
a_1,\mbox{ if }b=b_i\\
a_2,\mbox{ otherwise}
\end{cases}
\]
satisfies $\phi(b_i)\neq \phi(b_j)$. Thus, the semigroup $\Bcal$ is approximated by $\Acal$ and by Theorem~\ref{th:coord_iff_approx} $\Bcal$ is the coordinate semigroup of a nonempty algebraic set over $\Acal$.

\item Let $\Acal=RB(n,m)$ be  a rectangular band. If $n=1$ ($m=1$) the semigroup $\Acal$ is a left (respectively, right) zero semigroup and its coordinate semigroups were described above. Thus, we assume $n,m>1$. Let us show that an \textit{arbitrary} finite rectangular band $\Bcal=RB(n^\pr,m^\pr)$ is the coordinate semigroup of a nonempty algebraic set over $\Acal$. Without loss of generality one can assume $\Acal=RB(2,2)$. Let $b_1=(i_1,j_1)$, $b_2=(i_2,j_2)$ be two distinct elements of $\Bcal$, and $i_1\neq i_2$ (similarly one can consider the case $j_1\neq j_2$). It is directly checked that the map 
\[
\phi((i,j))=\begin{cases}
(1,1),\mbox{ if }i=i_1\\
(1,2),\mbox{ otherwise }
\end{cases}
\]
is a homomorphism and  $\phi(b_1)=(1,1)\neq (1,2)=\phi(b_2)$. Thus, the semigroup $\Bcal$ is approximated by $\Acal$, therefore $\Bcal$ is the coordinate band of an algebraic set over $\Acal$.

\item Let $\Acal=L_2$ be a linear ordered semilattice consisting of the elements $0,1$. Let us show that an  \textit{arbitrary} finite semilattice $\Bcal$ is the coordinate semilattice of a nonempty algebraic set over $\Acal$.   

Let us consider two distinct elements $b_1,b_2\in\Bcal$. One can assume $b_2\nleq b_1$ (otherwise, we assign $b_1:=b_2$, $b_2:=b_1$ and obtain the condition $b_2\nleq b_1$). Let us define a map $\phi\colon\Bcal\to\Acal$:
\[
\phi(b)=\begin{cases}
1,\mbox{ if } b\geq b_2,\\
0\mbox{ otherwise}
\end{cases}
\]  

Let us prove the map $\phi$ is a homomorphism of $\Bcal$. Assume the converse: $\phi(xy)\neq\phi(x)\phi(y)$ for some $x,y\in\Bcal$. We have the following cases:
\begin{enumerate}
\item $\phi(xy)=0$, $\phi(x)=\phi(y)=1$;
\item $\phi(xy)=1$, $\phi(x)=0$;
\item $\phi(xy)=1$, $\phi(y)=0$.
\end{enumerate}
In the first case we have $b_2\leq x,b_2\leq y$. Therefore, $b_2\leq xy$, and the definition of  $\phi$ gives $\phi(xy)=1$, a contradiction. In the second case we obtain $xy\geq b_2$ that implies $x\geq b_2$, and the definition of $\psi$ gives $\phi(x)=1$, a contradiction. The proof of the third case is similar to the second one. 

Thus, $\phi\in\Hom(\Bcal,\Acal)$ and $\phi(b_1)\neq\phi(b_2)$. We have that $\Bcal$ is approximated by $L_2$. Since each nontrivial semilattice  $\Acal$ contains $L_2$ as a subsemilattice, we obtain the following statement.

\begin{theorem}
\label{th:any_semilattice_is_coord}
Any finite semilattice $\Bcal$ is the coordinate semilattice of a nonempty algebraic set over a nontrivial semilattice $\Acal$.  
\end{theorem} 

\item Let us describe the coordinate unars of algebraic sets of nonempty algebraic sets over the free unar $\Acal=FU_n$ of rank $n$. Let us consider a finitely generated unar $\Bcal$ with the set of generators $b_1,b_2,\ldots,b_m$. Therefore, each element $b\in\Bcal$ is equal to an expression $b=f^k(b_i)$ for some numbers $i,k$.

If there exists numbers $i,k$ such that $f^k(b_i)=b_i$ then the set of homomorphisms $\Hom(\Bcal,\Acal)$ is empty (it is impossible to map $b_i$ into $\Acal$, since there is not any element $a\in\Acal$ with $f^k(a)=a$). Therefore, $\Bcal$ is not a coordinate unar of a nonempty algebraic set over $\Acal$. 

Thus, for any $i$ all elements in sequence $b_i,f(b_i),f^2(b_i),\ldots$ are different. Moreover, one can assume there are not generators $b_i,b_j$ with $f^k(b_i)=b_j$ (otherwise,one can remove $b_j$ from the set of generators).

Assume there exist different elements $b,b^\pr\in \Bcal$ such that $f(b)=f(b^\pr)$. In this case any homomorphism $\phi\in\Hom(\Bcal,\Acal)$ gives $\phi(b)=\phi(b^\pr)$, since $f$ is injective on $\Acal$. Thus, $\Bcal$ is not approximated by $\Acal$ in this case.

Finally, we have all elements of the set $\{f^k(b_i)\}$ ($k\in\N$, $1\leq i\leq m$) are pairwise distinct, and it follows that {\it the unar $\Bcal$ is isomorphic to the free unar of rank $m$}. Thus, we obtain the following statement.

\begin{theorem}
A finitely generated unar $\Bcal$ is the coordinate unar of a nonempty algebraic set over $\Acal=FU_n$ ($n\geq 1$) iff $\Bcal\cong FU_m$ for some $m\geq 1$.
\label{th:unars_coord}
\end{theorem} 
\begin{proof}
The ``only if'' part of the statement was proved above. Let us prove the ``if'' part. Let $\Bcal=FU_m$, where $\Bcal$ is freely generated by the set of generators $b_1,b_2,\ldots,b_m$, and $a_1$ be an arbitrary free generator of $\Acal$. Let us take two different elements $b,b^\pr\in\Bcal$. If $b=f^k(b_i)$, $b^\pr=f^{k^\pr}(b_i)$ (i.e. $b,b^\pr$ lies in the same subunar $FU_1\subseteq\Bcal$ generated by the element $b_i$), then  the homomorphism
\[
\phi(b_j)=a_1\; \mbox{ for all $1\leq j\leq m$}
\] 
has the property 
\[
\phi(b)=f^k(a_1)\neq f^{k^\pr}(a_1)=\phi(b^\pr).
\]

In the case of $b=f^k(b_i)$, $b=f^{k^\pr}(b_{i^\pr})$, the homomorphism $\phi$ is defined as follows
\[
\phi(b_j)=a_1,
\]
if $k\neq k^\pr$, and 
\[
\phi(b_j)=\begin{cases}
a_1,\mbox{ if } j\neq i^\pr\\
f(a_1),\mbox{ otherwise } j=i^\pr
\end{cases}
\]
if $k=k^\pr$. One can directly check that $\phi(b)\neq \phi(b^\pr)$.

Thus, we proved the unar $\Bcal$ is approximated by $\Acal$, and according to Theorem~\ref{th:coord_iff_approx} $\Bcal$ is the coordinate unar of a nonempty algebraic set over $\Acal$. 
\end{proof}

\end{enumerate}

One can describe coordinate algebras with embeddings.

\begin{theorem}
An $\LL$-algebra $\Bcal$ is the coordinate $\LL$-algebra of a nonempty algebraic set over an $\LL$-algebra $\Acal$ iff $\Bcal$ is embedded into a direct power of $\Acal$.
\label{th:coord_iff_embeds_into_power}
\end{theorem}
\begin{proof}
Let us prove the ``only if'' part of the statement. Suppose $\Bcal\cong\Gamma_\Acal(Y)$ for some algebraic set $Y$ over the $\LL$-algebra $\Acal$. For each $b\in \Bcal$ let us define a vector $v(b)$ such that
\begin{enumerate}
\item coordinates of $v(b)$ are indexed by points $P\in Y$;
\item the value of the $P$-th coordinate is $\phi_P(b)$.
\end{enumerate}
By Theorem~\ref{th:coord_iff_approx}, $\Bcal$ is approximated by $\Acal$, i.e. for any pair $b_1,b_2\in \Bcal$ there exists a homomorphism $\phi_P$ with $\phi_P(b_1)\neq \phi_P(b_2)$. Thus, the vectors $v(b_1),v(b_2)$ are distinct in the $P$-th coordinate, and the map $v$ is an embedding of $\Bcal$ into a direct power of $\Acal$.

Let us prove the ``if'' of the theorem. According to Theorem~\ref{th:coord_iff_approx}, it is sufficient to prove that each subalgebra of a direct power $\prod_{i\in I}\Acal$ is approximated by $\Acal$. Let
\[
\bar{b}_1=(b_{11},b_{12},\ldots,b_{1n},\ldots),\;\bar{b}_2=(b_{21},b_{22},\ldots,b_{2n},\ldots)
\]
be two distinct elements a direct power of $\Acal$ ($b_{ij}\in \Acal$). Therefore there exists a coordinate with index $j$ such that $b_{1j}\neq b_{2j}$. The projection 
\[
\pi_j\colon \prod_{i\in I}\Acal\to \Acal
\]
of a direct power onto the $j$-th coordinate is a homomorphism, and $\pi_j(\bar{b}_1)\neq\pi_j(\bar{b}_2)$. Thus, any subalgebra of a direct power of $\Acal$ is approximated by $\Acal$.
\end{proof}

Let us apply Theorem~\ref{th:coord_iff_embeds_into_power} for descriptions of coordinate algebras.

\begin{enumerate}
\item Theorem~\ref{th:coord_iff_embeds_into_power} allows us to describe the coordinate groups of algebraic sets over a finitely generated abelian group $\Acal$. According to theory of abelian groups, a group $\Acal$ is a direct sum~(\ref{eq:abelian_group_presentation}). Let $Sub_\oplus(\Acal)$ denote the set of all nonzero subgroups of all direct summands of $\Acal$. For example, the group 
\begin{equation}
\Acal=\Zbb\oplus\Zbb\oplus\Zbb_{8}\oplus\Zbb_3\oplus\Zbb_3
\label{eq:example_of_abelian_group}
\end{equation}
gives $Sub_\oplus(\Acal)=\{\Zbb,\Zbb_8,\Zbb_4,\Zbb_2,\Zbb_3\}$. The following result holds.

\begin{theorem}
A finitely generated group $\Bcal$ is the coordinate group of a nonzero algebraic set over an abelian group $\Acal$ iff $\Bcal$ is isomorphic to a finite direct sum of groups from $Sub_\oplus(\Acal)$. 
\label{th:coord_abelian_groups}
\end{theorem}
\begin{proof}
Let us explain the idea of the proof at the group $\Acal$ given by~(\ref{eq:example_of_abelian_group}). By Theorem~\ref{th:coord_iff_embeds_into_power}, the group $\Bcal=\Zbb_8\oplus\Zbb_8\oplus\Zbb_8$ is coordinate over $\Acal$, since $\Bcal$ is embedded into the direct power $\Acal\oplus\Acal\oplus\Acal$. Similarly, one can embed into a direct power of $\Acal$ any sum of groups from $Sub_\oplus(\Acal)$. 

If a group $\Bcal$ has a direct summand $G\notin Sub_\oplus(\Acal)$ in its presentation~(\ref{eq:abelian_group_presentation}), the subgroup $G$ is not embeddable into any direct power of $\Acal$. For example, the group $\Bcal=\Zbb\oplus\Zbb_{9}$ contain the subgroup $\Zbb_9$ which is not embedded into direct powers of $\Acal$. Thus, $\Bcal$ is not a coordinate group over $\Acal$. 
\end{proof}

\item Applying Theorems~\ref{th:any_semilattice_is_coord} and~\ref{th:coord_iff_embeds_into_power}, one can obtain the well-known result in semilattice theory.

\begin{proposition}
Any finite (actually, one can prove this statement also for infinite semilattices) semilattice $\Bcal$ is embedded into a direct power of $L_2=\{0,1\}$.
\end{proposition} 

\item Let $\Acal=\N$ be the additive monoid of natural numbers. According to Theorem~\ref{th:coord_iff_embeds_into_power}, a finitely generated monoid $\Bcal$ is a coordinate monoid over $\Acal$ iff $\Bcal$ is embedded into a direct power of $\N$. Remark that the monoid $\N^n$ has a rich class of submonoids. For example, it contains monoids which are not isomorphic to a direct powers of $\N$.   

\end{enumerate}

\section{Coordinate algebras of irreducible algebraic sets}
\label{sec:irred_coord_properties}

An $\LL$-algebra $\Acal$ {\it discriminates} an $\LL$-algebra $\Bcal$ if for any finite set of distinct elements $b_1,b_2,\ldots,b_n$ there exists a homomorphism  $\phi\colon\Bcal\to\Acal$ such that $\phi(b_i)\neq\phi(b_j)$ for all $i\neq j$, i.e. $\phi$ is {\it injective} on the set $\{b_1,b_2,\ldots,b_n\}$. Obviously, the approximation is a special case ($n=2$) of the discrimination. 

The following theorem shows the role of the discrimination in the description of {\it irreducible} coordinate algebras.

\begin{theorem}
A finitely generated $\LL$-algebra $\Bcal$ is the coordinate $\LL$-algebra of an irreducible algebraic set $Y$ over an  $\LL$-algebra $\Acal$ iff $\Bcal$ is discriminated by $\Acal$.
\label{th:coord_iff_discr}
\end{theorem}
\begin{proof}
Let us prove the ``only if'' part of the statement. Suppose $\Bcal\cong\Gamma_\Acal(Y)$ for a nonempty algebraic set $Y$ over $\Acal$. Let us assume that $\Bcal$ is not discriminated by $\Acal$, i.e. there exist distinct elements $[t_1(X)]_Y,[t_2(X)]_Y,\ldots,[t_m(X)]_Y\in\Bcal$ such that for any homomorphism $\phi\in\Hom(\Bcal,\Acal)$ there are indexes $i\neq j$ with  $\phi([t_i(X)]_Y)=\phi([t_j(X)]_Y)$. Since any homomorphism of coordinate algebras has the form~(\ref{eq:hom_of_substitution}), the assumption above is formulated as follows: for each point $P\in Y$ there exist indexes $i\neq j$ such that $t_i(P)=t_j(P)$. 
Therefore
\begin{equation}
Y=\bigcup_{i< j}(Y\cap\V_\Acal(t_i(X)=t_j(X))).
\label{eq:1111}
\end{equation}

According to Theorem~\ref{th:coord_iff_approx}, $\Bcal$ is approximated by $\Acal$, i.e. for all $i\neq j$ there exists a point $P_{ij}\in Y$ with $t_i(P_{ij})\neq t_j(P_{ij})$. It follows that $Y_{ij}\subset Y$ for all $i<j$. Thus,~(\ref{eq:1111}) is a proper decomposition of the set $Y$, and we obtain a contradiction with the irreducibility of the set $Y$.    

Let us prove the ``if'' part of the statement. Since the approximation is a special case of the discrimination, $\Bcal$ is approximated by $\Acal$. According to  Theorem~\ref{th:coord_iff_approx}, there exists an algebraic set over $\Acal$ with $\Gamma_\Acal(Y)\cong \Bcal$. Assume that $Y$ is reducible and it admits a decomposition
\[
Y=Y_1\cup Y_2\cup\ldots \cup Y_m\; (Y_i\subset Y,\; Y_i\nsubseteq Y_j \mbox{ for all }i\neq j).
\] 
By~(\ref{eq:about_radical_1}), it follows  $\Rad_\Acal(Y)\subset \Rad_\Acal(Y_i)$, and there exist equations $t_i(X)=s_i(X)\in \Rad_\Acal(Y_i)\setminus \Rad_\Acal(Y)$. Let us show that an arbitrary homomorphism $\phi\in\Hom(\Bcal,\Acal)$ is not injective on the set  $M=\{[t_i(X)]_Y,[s_i(X)]_Y\mid 1\leq i\leq m\}$. Following~(\ref{eq:hom_of_substitution}), $\phi=\phi_P$ for some point  $P\in Y$. Since $Y$ is a union of $Y_i$, the have $P\in Y_k$ for some $1\leq k\leq m$. Therefore $t_k(P)=s_k(P)$ and 
\begin{equation}
\phi_P([t_k(X)]_Y)=\phi_P([s_k(X)]_Y).
\label{eq:2222}
\end{equation}

Since $t_k(X)=s_k(X)\notin \Rad_\Acal(Y)$, the elements $[t_k(X)]_Y,[s_k(X)]_Y$ are distinct in the $\LL$-algebra $\Bcal$. So $[t_k(X)]_Y,[s_k(X)]_Y$ are distinct elements of the set $M$. However, by~(\ref{eq:2222}) the homomorphism $\phi_P$ is not injective on $M$, a contradiction with the discrimination of  $\Bcal$.   
\end{proof}

\begin{theorem}
\label{th:discr_is_embedding_for_finite}
Let $\Acal$ be a finite $\LL$-algebra. A finitely generated    
$\LL$-algebra $\Bcal$ is the coordinate algebra of an irreducible algebraic set over $\Acal$ iff $\Bcal$ is embedded into $\Acal$. 
\end{theorem}
\begin{proof}
Let us prove the ``only if'' part of the statement. If $\Bcal$ contains at least $n>|\Acal|$ different elements $b_1,b_2,\ldots,b_n$, there is not any homomorphism $\phi\colon \Bcal\to\Acal$ which is injective on the set $\{b_i\}$, and we obtain a contradiction with Theorem~\ref{th:coord_iff_discr}. Therefore, $|\Bcal|\leq|\Acal|$, i.e. $\Bcal$ is also finite. Let $b_1,b_2,\ldots,b_m$ ($m\leq n$) be all elements of $\Bcal$. An injective on $\{b_i\}$ homomorphism  $\phi\colon\Bcal\to\Acal$ provides the embedding of $\Bcal$ into $\Acal$.

The proof of the ``if'' part of the statement immediately follows from the obvious discrimination of any subalgebra $\Bcal\subseteq\Acal$ by $\Acal$.
\end{proof}

\begin{theorem}
\label{th:discr_is_embedding_for_locally_finite}
Let $\Acal$ be an $\LL$-algebra from a locally finite variety $V$ (remark that $\Acal$ may not be finitely generated). Then a finitely generated   
$\LL$-algebra $\Bcal$ is the coordinate algebra of an irreducible algebraic set over $\Acal$ iff $\Bcal$ is embedded into $\Acal$. 
\end{theorem}
\begin{proof}
The proof of this theorem repeats the proof of Theorem~\ref{th:discr_is_embedding_for_finite}, since $\Bcal\in V$ (Remark~\ref{rem:coord_alg_from_the_same_variety}) that gives the finiteness of  $\Bcal$.
\end{proof}

Let us describe the coordinate algebras of irreducible algebraic sets over some algebras.

\begin{enumerate}
\item Let $\Acal$ be a torsion-free abelian group. Proposition~\ref{pr:coord_groups_over_torsion_free_group} provides that a group $\Bcal$ is approximated by $\Acal$ iff $\Bcal\cong\Zbb^n$ for some $n$. Let us show that $\Bcal$ is discriminated by $\Acal$. First of all we assume the following.
\begin{enumerate}
\item We assume that $\Acal\cong\Zbb$ (any torsion-free group $G$ contains a subgroup $\Zbb$, therefore the discrimination of $\Bcal$ by $\Zbb$ implies the discrimination by $G$).
\item We assume $\Bcal=\Zbb^2=\Zbb\oplus\Zbb$ (the general case $\Bcal\cong\Zbb^n$ is considered by the induction on $n$ and its proof is similar to the case $n=2$).
\end{enumerate}

Let $b_1,b_2,\ldots,b_n$ be different elements of $\Bcal\cong\Zbb^2$. Each $b_i$ is a pair $b_i=(b_{i1},b_{i2})$, where $b_{ij}\in\Zbb$. Let us consider a map $\phi_\al\colon\Bcal\to\Zbb$ $\phi_\al((x,y))=\al x+ y$, where a number $\al\in\Zbb\setminus\{0\}$ will be defined below. It is easy to see that $\phi_\al\in\Hom(\Bcal,\Acal)$. Let us choose $\al$ such that the homomorphism is $\phi_\al$ injective on $\{b_i\}$. An integer number $\al$ is said to be {\it forbidden} if $\phi_\al$ is not injective on the set $\{b_i\}$. Let us prove there exist at most finite set of forbidden numbers. Indeed, the equality $\phi_\al(b_i)=\phi_\al(b_j)$ implies
\[
\al b_{i1}+ b_{i2}=\al b_{j1}+ b_{j2}\Leftrightarrow \al(b_{i1}-b_{j1})+(b_{i2}-b_{j2})=0,
\]
and the last equation has a unique solution $\al_{ij}$. Choosing all pairs $i,j$, we obtain a finite set of forbidden numbers $\al_{ij}$. Thus, there exists an integer permitted number $\al$ such that the homomorphism $\phi_\al$ is injective on $\{b_i\}$, and therefore $\Bcal$ is discriminated by $\Acal$.

According to Theorem~\ref{th:coord_iff_discr}, we have the following result.

\begin{proposition}
A finitely generated abelian group $G$ is the coordinate group of an irreducible algebraic set over a torsion-free abelian group $A$ iff $G$ is isomorphic to $\Zbb^n$ for some $n$. 
\label{pr:irred_coord_groups_over_torsion_free_group}
\end{proposition} 

\item Similarly, one can prove that any submonoid in $\N^n$ is discriminated by $\N$, and Theorem~\ref{th:coord_iff_discr} gives the following result.

\begin{proposition}
A finitely generated commutative monoid $M$ is the coordinate monoid of an irreducible algebraic set over  $\N$ iff $M$ is embedded into a direct power of $\N$. 
\label{pr:irred_coord_groups_over_N}
\end{proposition}

\item Since the varieties of semilattices, left zero semigroups and rectangular bands are locally finite, irreducible coordinate semigroups in such varieties satisfy Theorem~\ref{th:discr_is_embedding_for_locally_finite}. 

\item Let $A$ be a finitely generated abelian group. By~(\ref{eq:abelian_group_presentation}), the group $A$ admits a presentation $\Acal=\Zbb^n\oplus T(A)$. If $n=0$ (i.e. $A=T(A)$), the group $A$ is finite, and Theorem~\ref{th:discr_is_embedding_for_finite} states that irreducible coordinate groups over $A$ are exactly the subgroups of $A$.

Suppose now $n>0$ and consider a finitely generated abelian group $B=\Zbb^m\oplus T(B)$ such that $B$ is discriminated by $A$. Since the subgroup $T(B)$ is finite, the discrimination gives that $T(B)$ is embedded into $T(A)$. By~\ref{pr:coord_groups_over_torsion_free_group}, the subgroup $\Zbb^m\subseteq B$ is discriminated by $\Zbb^n\subseteq A$, and we obtain the following result.

\begin{theorem}
A finitely generated abelian group $B=\Zbb^m\oplus T(B)$ is an irreducible coordinate group over a finitely generated abelian group $A=\Zbb^n\oplus T(A)$ iff the torsion $T(B)$ is embedded into $T(A)$ and $m=0$ for $n=0$.
\label{th:abelian_groups_irred_coord}
\end{theorem} 

\item Let us show (as we promised in Section~\ref{sec:irreducible_algebraic_sets}) that an algebraic set $Y=\V_{FS_r}(\{xy=yx,yz=zy\})$ over the free semigroup $FS_r$ is irreducible. Let us compute the coordinate semigroup of the set $Y$. Remark that the equation $xz=zx$ belongs to the radical of $Y$, therefore the coordinate semigroup $\Gamma_{FS_r}(Y)$ generated by the elements $x,y,z$ is commutative. Since there are not other relations between the elements $x,y,z\in \Gamma_{FS_r}(Y)$, the semigroup $\Gamma_{FS_r}(Y)$ is isomorphic to the free commutative semigroup of rank $3$, i.e. $\Gamma_{FS_r}(Y)\cong(\N\setminus 0)^3$. To prove the irreducibility of $Y$ it is sufficient to show that the semigroup $(\N\setminus 0)^3$ is discriminated by $FS_r$. The last property follows from Proposition~\ref{pr:irred_coord_groups_over_N}, where we proved that $(\N\setminus 0)^3$ is discriminated by $\N\setminus 0\subseteq FS_r$.

\end{enumerate}

\section{When all algebraic sets are irreducible}
\label{sec:co-domains}

There are $\LL$-algebras where all nonempty algebraic sets are irreducible. Following~\cite{uni_Th_IV}, such algebras are called \textit{co-domains}. Theorems ~\ref{th:coord_iff_approx},~\ref{th:coord_iff_discr} give the following criterion for an $\LL$-algebra to be a co-domain.

\begin{theorem}
An $\LL$-algebra $\Acal$ is a co-domain iff for any finitely generated $\LL$-algebra $\Bcal$ the following conditions are equivalent 
\begin{enumerate}
\item $\Bcal$ is approximated by $\Acal$; 
\item $\Bcal$ is discriminated by $\Acal$
\end{enumerate}
(remark that the second condition implies the first one for any $\LL$-algebra $\Acal$).
\label{th:co-domain_criterion}
\end{theorem}

Let us show examples of co-domains in different classes of algebras.

\begin{enumerate}
\item Let us show that an {\it infinite direct power of any $\LL$-algebra $\Acal$ is a co-domain}. 
Let $\Acal^\om=\prod_{i\in I}\Acal$ denote an infinite power of $\Acal$, and suppose that a finitely generated  $\LL$-algebra $\Bcal$ is approximated by $\Acal^\om$. By Theorem~\ref{th:coord_iff_embeds_into_power}, $\Bcal$ is embedded into a direct power  $\Acal^\om$, so $\Bcal$ is embedded into a direct power of $\Acal$. Thus, one can treat elements of $\Bcal$ as vectors with coordinates from $\Acal$. Let $b_1,b_2,\ldots,b_n$ be different elements of  $\Bcal$, i.e. there exists an index  $i_{jk}\in I$ such that the elements $b_j,b_k$ have different $i_{jk}$-th coordinates. Let $I^{\pr}=\{i_{kj}\mid 1\leq j<k\leq n\}$, and  $\Bcal^{\pr}$ be the projection $\pi(\Bcal)$ onto the coordinates $I^{\pr}$. Since $\Bcal^{\pr}$ is a finite direct power of $\Acal$, the algebra $\Bcal^{\pr}$ is embedded into $\Acal^\infty$ and the images $\pi(b_1),\pi(b_2),\ldots,\pi(b_n)$ are different in $\Acal^\infty$. Thus, $\Bcal$ is discriminated by $\Acal^\om$, and by Theorem~\ref{th:co-domain_criterion} the algebra $\Acal^\om$ is a co-domain.

\item According to Propositions~\ref{pr:irred_coord_groups_over_torsion_free_group},~\ref{pr:irred_coord_groups_over_N}, any torsion-free abelian group and the monoid $\N$ are co-domains.

\item Let $\Fcal$ be the free semilattice of infinite rank. According to semilattice theory, $\Fcal$ is isomorphic to the family of all finite subsets of natural numbers $\N$ relative to the operation of set union. Moreover, semilattice theory states that any finite semilattice is embedded into $\Fcal$. Therefore $\Fcal$  discriminates any finite semilattice (in the locally finite varieties of semilattices every finitely generated semilattice is finite). Thus $\Fcal$ is a co-domain.


\end{enumerate}

The following algebras are not co-domains.

\begin{enumerate}
\item Let $\Acal=\{a_1,a_2,\ldots,a_n\}$ be a finite $\LL$-algebra. Then the solution set of the $\LL$-system  $\Ss=\{x_1=x_1,x_2=x_2,\ldots,x_{n+1}=x_{n+1}\}$ is reducible:
\[
\V_\Acal(\Ss)=\bigcup_{i\neq j}\V_\Acal(\Ss\cup \{x_i=x_j\}).
\]

\item By the previous item, any cyclic group $\Zbb_{n}$ is not a co-domain. Moreover, below we prove that {\it any finitely generated abelian group with a nontrivial torsion is not a co-domain.}  Let $\Acal=\Zbb^n\oplus T(\Acal)$. Since $T(\Acal)$ is finite for a finitely generated $\Acal$, there exists a natural number $m$ such that $mx=0$ for each $x\in T(\Acal)$. Let $k$ denote the order of $T(\Acal)$ and $\Ss$ be an $\LL$-system 
\[
\{mx_i=0\mid 1\leq i\leq k+1\},
\]
which is obviously satisfied only by points  $(a_1,a_2,\ldots,a_{k+1})\in T^{k+1}(\Acal)$. Since $|T(\Acal)|=k$, there exists a pair of coordinates  $i,j$ such that $a_i=a_j$. Therefore the solution set of  $\Ss$ is reducible and admits a presentation
\[
\V_\Acal(\Ss)=\bigcup_{i\neq j}\V_\Acal(\Ss\cup \{x_i=x_j\}).
\]  

\item The free monoid $FM_r$ of rank $r>1$ is not a co-domain, since
\begin{multline*}
\V_{FM_r}(\{xy=yx,yz=zy\})=\V_{FM_r}(\{xy=yx,yz=zy,xz=zx\})\cup\\
\V_{FM_r}(\{xy=yx,yz=zy,y=1\})
\end{multline*}

\end{enumerate}

\section{The intervention of model theory}

\subsection{Quasi-identities and universal formulas}
\label{sec:quasi-identities_and_univ_formulas}

One of the unexpected result of Unifying theorems is the description of coordinate algebras by first-order formulas of special kind. Below we assume that a reader knows the definition of a first-order formula\Wiki of a language $\LL$ and main identities of the first-order logic. The operations of the conjunction, disjunction and negation are denoted by $\wedge,\vee,\neg$ respectively. Moreover, the expression $x\neq y$ is the abbreviation for $\neg(x=y)$.

A \textit{universal formula} of a language $\LL$ is a formula equivalent to  
\begin{equation}
\label{eq:universal_formula}
\forall x_1\forall x_2\ldots \forall x_n \varphi(x_1,x_2,\ldots,x_n),
\end{equation}
where $\varphi$ is a quantifier-free formula.
A \textit{quasi-identity} of $\LL$ is a universal formula, where the subformula $\varphi$ has the form 
\[
(t_1(X)=s_1(X))\wedge(t_2(X)=s_2(X))\wedge\ldots \wedge(t_m(X)=s_m(X))\to (t(X)=s(X)),
\]
where $t_i(X),s_i(X),t(X),s(X)$ are $\LL$-terms in variables $X=\{x_1,x_2,\ldots,x_n\}$.

Many properties of algebraic structures can be written as universal formulas or quasi-identities.

\begin{enumerate}
\item It is easy to see that any identity 
\[
\forall x_1\forall x_2\ldots \forall x_n \; t(X)=s(X)
\] 
of a language $\LL$ is a quasi-identity, since it is equivalent to
\[
\forall x_1\forall x_2\ldots \forall x_n \left((x_1=x_1)\to (t(X)=s(X))\right)
\] 

\item The cancellation property in the language $\LL_s$ is written as the following quasi-identity
\[
\forall x\forall y\forall z ((xz=yz)\to (x=y))\wedge ((zx=zy)\to (x=y)).
\]  
For a commutative semigroup of the language $\LL_{+s}$ the cancellation property is written as the following quasi-identity
\[
\forall x\forall y\forall z ((x+z=y+z)\to (x=y)).
\] 

\item One of the principal properties of the additive monoid of natural numbers can be written as the quasi-identity
\[
\forall x\forall y((x+y=0)\to (x=0)).
\]  
Following~\cite{MorShev}, this property is called the~\textit{positiveness}, and in Section~\ref{sec:logic_coord} it is used in the description of coordinate monoids over $\N$.

\item The injectivity of a function $f$ defined on an unar $U$ can be written as the following quasi-identity of the language $\LL_u$ 
\[
\forall x\forall y((f(x)=f(y))\to (x=y)).
\] 
Obviously, all free unars $FU_n$ satisfy this formula.

\item The transitivity of the commutation in a semigroup can be written as the following quasi-identity 
\[
\forall x\forall y\forall z((xy=yx)\wedge(yz=zy)\to (xz=zx)).
\]

\item If an abelian group $A$ satisfies the following list of quasi-identities
\[
\{\forall x ((nx=0)\to(x=0))\mid n\in\N\}
\]
the group $A$ is torsion-free.

\end{enumerate}

The following properties of $\LL$-algebras can be written as universal formulas .
\begin{enumerate}
\item If a semilattice $S$ satisfies the following formula of the language $\LL_s$
\[
\forall x\forall y\; ((x\leq y)\vee (y\leq x)),
\]
$S$ is linearly ordered (precisely, in this formula one should replace all occurrences of inequalities $a\leq b$ to $ab=a$, and after that we obtain a formula of the language $\LL_s$)

\item All free monoids satisfy the universal formula
\[
\forall x\forall y\forall z((xy=yx)\wedge(yz=zy)\to ((y=1)\vee(xz=zx))).
\]
\item The following universal formula states that an $\LL$-algebra $\Acal$ {\it contain at most $n$  elements}:
\[
\forall x_1\forall x_2\ldots \forall x_n\forall x_{n+1}(\bigvee_{i\neq j}(x_i=x_j)).
\]
Remark that the properties {\it ``an $\LL$-algebra $\Acal$ contain exactly $n$ elements''} and {\it ``an $\LL$-algebra $\Acal$ contains at least $n$ elements''} are not expressible by universal formulas.

\item The property `` in a group $G$ there are at most $k$ elements of the order $n$'' is written as the following universal formula of the language $\LL_g$
\[
\forall x_1\forall x_2\ldots \forall x_k\forall x_{k+1} (\bigwedge_{i=1}^{k+1}\bigwedge_{j=1}^{n-1}(x_i^{j}\neq 1)\bigwedge_{i=1}^{k+1}(x_i^n=1)\to\bigvee_{i\neq j}(x_i=x_j)).
\]

\end{enumerate} 

Let $\Acal$ be an $\LL$-algebra. {\it The quasi-variety $\qvar(\Acal)$ (universal closure $\ucl(\Acal)$)} is the class of all $\LL$-algebras $\Bcal$ such that $\Bcal$ satisfies all quasi-varieties (respectively, universal formulas) which are true in $\Acal$. 

Since any quasi-identity is a universal formula, we have the inclusion $\ucl(\Acal)\subseteq\qvar(\Acal)$.

The importance of the classes $\qvar(\Acal)$, $\ucl(\Acal)$ in universal algebraic geometry follows from the next two sections.

\subsection{Coordinate algebras}
\label{sec:logic_coord}
\begin{theorem}
If a finitely generated $\LL$-algebra $\Bcal$ is the coordinate algebra of a nonempty algebraic set over an $\LL$-algebra $\Acal$, then $\Bcal\in\qvar(\Acal)$.
\label{th:coord_then_qvar}
\end{theorem}
\begin{proof}
By Theorem~\ref{th:coord_iff_approx}, $\Bcal$ is approximated by $\Acal$. Suppose there exists a quasi-identity 
\[
\varphi\colon \forall x_1\forall x_2\ldots\forall x_n (\bigwedge_{i=1}^m t_i=s_i\to t=s)
\]
such that $\varphi$ is true in $\Acal$, but false in $\Bcal$ (above $t_i,s_i,t,s$ are $\LL$-terms in variables $X=\{x_1,x_2,\ldots,x_n\}$). Therefore, there exist elements $b_1,b_2,\ldots,b_n\in\Bcal$ such that $t_i(b_1,b_2,\ldots,b_n)=s_i(b_1,b_2,\ldots,b_n)$, but $t(b_1,b_2,\ldots,b_n)\neq s(b_1,b_2,\ldots,b_n)$. Since $\Bcal$ is approximated by $\Acal$, there exists a homomorphism  $\phi\in\Hom(\Bcal,\Acal)$ with 
\begin{equation}
\label{eq:ineq111}
\phi(t(b_1,b_2,\ldots,b_n))=t(\phi(b_1),\phi(b_2),\ldots,\phi(b_n)))\neq s(\phi(b_1),\phi(b_2),\ldots,\phi(b_n))=\phi(s(b_1,b_2,\ldots,b_n)). 
\end{equation}

Since $\phi$ preserves all equalities between elements of $\Bcal$, we have the following equalities in  $\Acal$:
\[
t_i(\phi(b_1),\phi(b_2),\ldots,\phi(b_n)))=s_i(\phi(b_1),\phi(b_2),\ldots,\phi(b_n)). 
\] 
Thus, the point $P=(\phi(b_1),\phi(b_2),\ldots,\phi(b_n))$ make the subformula 
\[
\bigwedge_{i=1}^m t_i=s_i
\]
true. Since $\varphi$ holds in $\Acal$, the equation $t(X)=s(X)$ should satisfy the point $P$, and we came to the contradiction with~(\ref{eq:ineq111}).
\end{proof}

The profit of Theorem~\ref{th:coord_then_qvar} is the following: for a given $\LL$-algebra $\Acal$ one can restrict the class of coordinate algebras over $\Acal$. For example, Theorem~\ref{th:coord_then_qvar} immediately gives that any coordinate monoid of algebraic set over the monoid $\N$ is positive, since $\N$ is positive and the property of positiveness can be written as quasi-identity of the language $\LL_{+m}$ (see above). 

Actually, for equationally Noetherian algebras it holds the converse to Theorem~\ref{th:coord_then_qvar} statement. 

\begin{theorem}
If a finitely generated  $\LL$-algebra $\Bcal$ belongs to the quasi-variety  $\qvar(\Acal)$ generated by an equationally Noetherian $\LL$-algebra $\Acal$, then $\Bcal$ is the coordinate algebra of a nonempty algebraic set over $\Acal$.
\label{th:qvar_then_coord}
\end{theorem}
\begin{proof}
Assume the converse: the $\LL$-algebra $\Bcal$ is not the coordinate algebra of a nonempty algebraic set over $\Acal$. By Theorem~\ref{th:coord_iff_approx} it follows that there exist two distinct elements $b,b^\pr\in\Bcal$ such that for any homomorphism $\phi\in\Hom(\Bcal,\Acal)$ it holds $\phi(b)=\phi(b^\pr)$. Suppose $\Bcal$ has a presentation $\lb b_1,b_2,\ldots,b_n\mid R\rb$. Let $\bar{b}$ denote the vector of the generators $(b_1,b_2,\ldots,b_n)$. There exist $\LL$-terms $t(\bar{b}),t^\pr(\bar{b})$ with $b=t(b_1,b_2,\ldots,b_n)$, $b^\pr=t^\pr(b_1,b_2,\ldots,b_n)$.

One can treat the set of defining relations $R=\{t_i(\bar{b})=s_i(\bar{b})\mid i\in I\}$ as an $\LL$-system in variables $b_1,b_2,\ldots,b_n$ over the algebra $\Acal$. Since $\Acal$ is equationally Noetherian, there exists a finite subsystem $R^\pr=\{t_i(\bar{b})=s_i(\bar{b})\mid i\in I^\pr\}$ such that $R^\pr\sim_\Acal R$ (in other words, the algebras $\Bcal^\pr=\lb b_1,b_2,\ldots,b_n\mid R^\pr\rb$ and $\Bcal$ have the same set of homomorphisms to the algebra $\Acal$). Let us show that the following quasi-identity 
\[
\varphi\colon \forall x_1\forall x_2\ldots\forall x_n (\bigwedge_{i\in I^\pr} t_i=s_i\to t=s).
\]
holds in $\Acal$.

If $\varphi$ does not holds in $\Acal$, there exist elements $a_1,a_2,\ldots,a_n$ such that 
\[
t_i(a_1,a_2,\ldots,a_n)=s_i(a_1,a_2,\ldots,a_n)\; (i\in I^\pr), t(a_1,a_2,\ldots,a_n)\neq s(a_1,a_2,\ldots,a_n).
\] 
Therefore the homomorphism $\phi\in\Hom(\Bcal^\pr,\Acal)=\Hom(\Bcal,\Acal)$ defined by $\phi(b_j)=a_j$ (one can check that this map is really homomorphism) satisfies  $\phi(b)=t(a_1,a_2,\ldots,a_n)\neq s(a_1,a_2,\ldots,a_n)=\phi(b^\pr)$. The last inequality contradicts the choice of the elements $b,b^\pr$. Thus, we obtain that $\varphi$ is true in $\Acal$.

Since $\Bcal\in\qvar(\Acal)$, the algebra $\Bcal$ satisfies the quasi-identity $\varphi$. The equalities $t_i(b_1,b_2,\ldots,b_n)=s_i(b_1,b_2,\ldots,b_n)$ ($i\in I^\pr$) hold in $\Bcal$, since they are defining relations of $\Bcal$.  The truth of $\varphi$ in $\Bcal$ implies $b=t(b_1,b_2,\ldots,b_n)=s(b_1,b_2,\ldots,b_n)=b^\pr$ that contradicts the choice of the elements $b,b^\pr$.
\end{proof}

Actually, one can prove Theorem~\ref{th:qvar_then_coord} for a wide class of  $\qq$-compact algebras (see details in Section~\ref{sec:compactness_classes}). 

The sense of Theorems~\ref{th:coord_then_qvar},~\ref{th:qvar_then_coord} is the following: any coordinate $\LL$-algebra $\Bcal$ of a nonempty algebraic set over an equationally Noetherian $\LL$-algebra $\Acal$ inherits the properties of $\Acal$ which can be written as quasi-identities of the language $\LL$. For example, a coordinate group over the cyclic group $\Zbb_{n}$ is not necessarily cyclic, since the property ``to be a cyclic group'' is not expressible as a quasi-identity of the language $\LL_{+g}$ (see Theorem~\ref{th:coord_abelian_groups_with_logic}). Moreover, Theorems~\ref{th:coord_then_qvar},~\ref{th:qvar_then_coord} explain the following phenomenon. One can consider the set $\Zbb=\{0,\pm 1,\pm, 2,\ldots\}$ as a commutative monoid in the language $\LL_{+m}$. Then the coordinate algebras of algebraic sets over the monoid $\Zbb$ are not necessarily abelian groups, since the property ``to be a group'' can not be written as a quasi-identity of the language $\LL_{+m}$ (see example below).

Thus, it follows from Theorems~\ref{th:coord_then_qvar},~\ref{th:qvar_then_coord} that the class of coordinate algebras over a given equationally Noetherian $\LL$-algebra $\Acal$ is logical, i.e. it can be defined by a set of logical formulas (quasi-identities). Let us give examples which demonstrate such property of coordinate algebras.

\begin{enumerate}
\item Let $A=\Zbb^k\oplus T(A)$ be a finitely generated abelian group. Below we define a set of quasi-identities $\Sigma_A$ which are true in $A$. Initially, we  put $\Sigma_A=\emptyset$.

A group $A$ has a {\it period} $n$ if for each $a\in A$ it holds $na=0$ and  $n$ is the minimal number with such property. For example, the group $A=\Zbb_2\oplus\Zbb_3$ has the period $n=6$. One can prove that {\it the period of a finitely generated abelian group exists iff $A=T(A)$}, and the period equals the least common multiple of the orders of direct summand of $T(A)$. For example, the periods of the groups $\Zbb_4\oplus\Zbb_3$, $A=\Zbb_2\oplus\Zbb_4\oplus\Zbb_8$ are $12$, $8$ respectively. 

Thus, if a group $A$ has the period $n$ we add to the set $\Sigma_A$ the following identity
\[
\forall x\; (nx=0)
\]
(remind that any identity is equivalent to a quasi-identity)

From the set of quasi-identities 
\begin{equation}
\Sigma_{p,n}=\{\forall x\; (p^nx=0\to \; p^{n-1}x=0)\mid p\mbox{ is prime },n\in\N\},
\label{eq:fghhs}
\end{equation}
let us add to $\Sigma_A$ the quasi-identies which are true in the group $A$. Let us explain, the sense of the set $\Sigma_{p,n}$ by the following examples (below $p,q$ are pimes).
\begin{enumerate}
\item If $A=\Zbb$ (one can take here any torsion-free group), then each quasi-identity of the form~(\ref{eq:fghhs}) is true in $A$ (the condition  $p^n x=0$ implies in $A$ the equality $x=0$, and $p^{n-1}x=0$ obviously holds).
 
\item Let $A=\Zbb_{q^n}$. All quasi-identities~(\ref{eq:fghhs}) for $p\neq q$ are true in $A$, since there are not nonzero elements of order $p^n$ in $A$, and the left part of the implication~(\ref{eq:fghhs}) becomes false, therefore the whole formula~(\ref{eq:fghhs}) becomes true. Moreover, any quasi-identity~(\ref{eq:fghhs}) for $p=q$ and $n<m$ is also true in $A$, since the condition  $q^mx=0$ implies that an element $x\in A$ becomes zero in a power less than $q^m$. By the other hand, the set $\Sigma_A$ does not contain all formulas with $p=q$ and $m\leq n$, since the condition $q^mx=0$ does not now imply  $q^{m-1}x=0$, because the group $\Zbb_{q^n}$ contains an element of order  $q^m$.

\item Let $A=\Zbb_{p^n}\oplus\Zbb_{q^m}$. By the reasonings from the previous item, one can prove that a quasi-identity $Q$ from $\Sigma_{p,n}$ is true in $A$ iff  $Q$ is true in $\Zbb_{p^n}$ and $\Zbb_{q^m}$ simultaneously.

\item Actually, the following general statement holds: {\it if the torsion $T(A)$ of a finitely generated abelian group $A$ does not contain $\Zbb_{q^m}$ as a subgroup of a direct summand, then $A$ satisfies the following quasi-identities 
\[
\forall x\; (q^nx=0\to \; q^{n-1}x=0)
\] 
for every $n\geq m$.}
\end{enumerate}

\begin{theorem}
Let $A,B$ be finitely generated abelian groups. Then the following conditions are equivalent:
\begin{enumerate}
\item $B$ is the coordinate group of a nonempty algebraic set over $A$; 
\item $B$ is a direct sum of groups from $Sub_\oplus(A)$.
\item $B$ satisfies all quasi-identities $\Sigma_A$.
\end{enumerate}
\label{th:coord_abelian_groups_with_logic}
\end{theorem}
\begin{proof}
The equivalence of the first two conditions was proved in Theorem~\ref{th:coord_abelian_groups}. By Theorem~\ref{th:coord_then_qvar}, we have $(a)\Rightarrow(c)$.  Thus, it is sufficient to prove $(c)\Rightarrow(b)$.

Assume the presentation~(\ref{eq:abelian_group_presentation}) of $B$ contains as a direct summand $\Zbb_{q^m}\notin Sub_\oplus(A)$. By the properties of the set $\Sigma_A$ the group $A$ satisfies the formulas
\[
\forall x\; (q^nx=0\to \; q^{n-1}x=0)
\] 
for any $n\geq m$. Since $B$ contains $\Zbb_{q^m}$ as direct summand, the equality $q^{m}x=0$ does not imply $q^{m-1}x=0$ in $\Zbb_{q^m}$. Thus, the quasi-identity
\[
\forall x\; (q^mx=0\to \; q^{m-1}x=0)
\]
is false in $B$, that contradicts the condition.
 
Therefore, for all direct summands $\Zbb_{q^n}$ of the group $B$ we proved  $\Zbb_{q^n}\in Sub_\oplus(A)$. Assume now that $\Zbb$ is a direct summand in $B$. If $\Zbb\notin Sub_\oplus(A)$ the group $A$ has a finite period $n$, and the definition of the set $\Sigma_A$ gives  $(\forall x\; (nx=0))\in \Sigma_A$. By the condition, such quasi-identity should hold in $B$, so $B$ has a finite period. We came to the contradiction, since $B$ contains $\Zbb$.

Thus, we proved that $B$ is a direct sum of groups from $Sub_\oplus(A)$. 
\end{proof}

\item Let us consider the set of integer numbers $\Zbb$ as a commutative monoid of the language $\LL_{+m}$. For coordinate algebras of algebraic sets over $\Zbb$ we have the following theorem.

\begin{theorem}
A finitely generated commutative monoid $M$ is the coordinate monoid of an algebraic set over the monoid $\Zbb$ iff $M$ satisfies the following quasi-identities:
\[
\forall x\forall y\forall z (x+z=y+z\to x=y)\; \mbox{(cancellation property)},
\] 
\[
\varphi_n\colon\;  \forall x\; (nx=0\to x=0),\; n\geq 1
\] 
\label{th:about_monoid_Z_coord_logic}
(the series $\{\varphi_n\}$ provides that $M$ does not contain elements of finite order).
\end{theorem}
\begin{proof}
The ``only if'' statement follows from Theorem~\ref{th:coord_then_qvar}, since all quasi-identities above hold in $\Zbb$. Let us prove the ``if'' statement of the theorem. Following Theorem~\ref{th:coord_iff_embeds_into_power}, it is sufficient to prove the embedding of $M$ into a direct power of the monoid $\Zbb$. By semigroup theory, a commutative monoid $M$ is embedded into an abelian group iff $M$ has the cancellation property. Thus, there exists an abelian group $A$ with a submonoid isomorphic to $M$. Since $M$ is finitely generated, so is $A$. Following~(\ref{eq:abelian_group_presentation}), $A$ is isomorphic to a direct sum $\Zbb^n\oplus T(A)$. Since the quasi-identities $\varphi_n$ hold in $M$, the monoid $M$ does not contain elements of finite order, and actually $M$ is embedded into a subgroup $\Zbb^n\subseteq A$.
\end{proof}

\item One can compare Theorem~\ref{th:about_monoid_Z_coord_logic} with the similar result for the monoid  $\N$. 

\begin{theorem}
A finitely generated commutative monoid $M$ is the coordinate monoid of an algebraic set over $\N$ iff $M$ satisfies the following quasi-identities:
\[
\forall x\forall y\forall z (x+z=y+z\to x=y)\; \mbox{(cancellation property)},
\] 
\[
\forall x\forall y (x+y=0\to x=0)\; \mbox{(positiveness)}.
\] 
\label{th:about_N_coord_logic}
\end{theorem}
\begin{proof}
By Proposition~\ref{pr:irred_coord_groups_over_N}, it is sufficient to prove that any positive commutative monoid with cancellations is embedded into a direct power of $\N^n$. One can see the proof of this fact in~\cite{grillet} or~\cite{MorShev} (remark that in~\cite{grillet} a positive monoid is called a {\it reduced monoid}).
\end{proof}

\item For unars of the language $\LL_u$ we have the following result. 

\begin{theorem}
For a nontrivial finitely generated unar $U$ the following conditions are equivalent:
\begin{enumerate}
\item $U$ is isomorphic to a free unar $FU_m$ for some $m\geq 1$;
\item $U$ is the coordinate unar of an algebraic set over the free unar  $FU_n$ ($n\geq 1$);
\item $U$ satisfies the quasi-identities 
\[
\varphi\colon\; \forall x \forall y (f(x)=f(y)\to x=y)\; (\mbox{the function $f$ is injective}),
\] 
\[
\varphi_n\colon\; \forall x \forall y (f^n(x)=x\to x=y) \; (n\geq 1)
\] 
(the quasi-identities $\varphi_n$ provides the following property: if an unar $U$ contains an element $x$ such that $x=f^n(x)$ for some $n\geq 1$ then the unar $U$ is trivial).  
\end{enumerate} 
\end{theorem}
\begin{proof}
One can see that a unar $U$ satisfying $\varphi,\varphi_n$ is either free or trivial. The trivial unar is the coordinate unar of the empty set, therefore $U$ is free.
\end{proof}

\end{enumerate}

\subsection{Irreducible algebraic sets}
\label{sec:logic_coord_irred}
Let us consider the role of the class $\ucl(\Acal)$ in algebraic geometry over an $\LL$-structure $\Acal$. First of all we prove the following auxiliary statement from model theory.

\begin{proposition}
An $\LL$-algebra $\Bcal$ belongs to the universal closure $\ucl(\Acal)$ of an $\LL$-algebra $\Acal$ iff  $\Bcal$ satisfies all formulas of the form
\begin{equation}
\label{eq:universal_formuila_special}
\forall x_1\forall x_2\ldots \forall x_n (\bigwedge_{i=1}^m(t_i=s_i)\to (\bigvee_{j=1}^l(t_j^\pr=s_j^\pr)))
\end{equation}
which are true in  $\Acal$ ($t_i,s_i,t^\pr_j,s_j^\pr$ are $\LL$-terms). 
\label{pr:ucl_aux}
\end{proposition}
\begin{proof}
According to the main identities of the first-order logic, the quantifier-free subformula $\varphi(x_1,x_2,\ldots,x_n)$ of an universal formula~(\ref{eq:universal_formula}) can be written as
\[
\varphi\colon \; \forall x_1\forall x_2\ldots \forall x_n\; (\bigwedge_{k=1}^rD_k),
\] 
where $D_k$ are disjunctions of atomic formulas and their negations (the precise structure $D_k$ is written below).
Using the logic identity
\[
\forall x\; (\varphi_1\wedge\varphi_2)\sim \forall x\; (\varphi_1)\wedge\forall x\; (\varphi_2),
\]
we obtain that $\varphi$ holds in $\Acal$ iff all formulas 
\begin{equation}
\forall x_1\forall x_2\ldots \forall x_n\; (D_k).
\label{eq:reytirweyto}
\end{equation}
are true in $\Acal$.

Thus the truth of~(\ref{eq:universal_formula}) in $\Bcal$ is equivalent to the simultaneous truth of formulas of the form~(\ref{eq:reytirweyto}). Thus, it is sufficient to prove that any formula~(\ref{eq:reytirweyto}) can be transformed to~(\ref{eq:universal_formuila_special}).

Indeed, the quantifier-free part $D_k$ of~(\ref{eq:reytirweyto}) is the following expression
\[
\bigvee_{i}(t_i\neq s_i)\bigvee_{j}(t_j^\pr=s^\pr_i).
\]   
Using the identities of the first-order logic, we obtain:
\[
\bigvee_{i}(t_i\neq s_i)\bigvee_{j}(t_j^\pr=s^\pr_i)\sim 
\neg(\bigwedge_{i}(t_i=s_i))\bigvee_{j}(t_j^\pr=s^\pr_i)\sim
\bigwedge_{i}(t_i=s_i)\to (\bigvee_{j}t_j^\pr=s^\pr_i).
\]   
\end{proof}

\begin{theorem}
If a finitely generated $\LL$-algebra $\Bcal$ is the coordinate $\LL$-algebra of an irreducible algebraic set over an $\LL$-algebra $\Acal$, then $\Bcal\in\ucl(\Acal)$.
\label{th:coord_then_ucl}
\end{theorem}
\begin{proof}
By Theorem~\ref{th:coord_iff_discr}, $\Bcal$ is discriminated by $\Acal$. Let us assume there exists a universal formula~(\ref{eq:universal_formuila_special}) which is true in $\Acal$ but false in $\Bcal$. Therefore, there exists a vector $\bar{b}=(b_1,b_2,\ldots,b_n)\in\Bcal^n$ such that $t_i(\bar{b})=s_i(\bar{b})$, $t_j^\pr(\bar{b})\neq s_j^\pr(\bar{b})$ for all $1\leq i\leq m$, $1\leq j\leq l$. Since $\Bcal$ is discriminated by $\Acal$ there exists a homomorphism $\phi\in\Hom(\Bcal,\Acal)$ with 
\begin{equation}
\label{eq:ineq222}
\phi(t_j^\pr(\bar{b}))=t_j^\pr(\phi(b_1),\phi(b_2),\ldots,\phi(b_n)))\neq s_j^\pr(\phi(b_1),\phi(b_2),\ldots,\phi(b_n))=\phi(s_j^\pr(\bar{b}))
\end{equation}
for all $1\leq j\leq l$.

By the definition of a homomorphism, we have the following equalities 
\[
t_i(\phi(b_1),\phi(b_2),\ldots,\phi(b_n)))=s_i(\phi(b_1),\phi(b_2),\ldots,\phi(b_n))
\] 
in $\Acal$ for each $1\leq i\leq m$.

Thus, the subformula 
\[
\bigwedge_{i=1}^m t_i=s_i
\]
becomes true for the vector $\phi(\bar{b})=(\phi(b_1),\phi(b_2),\ldots,\phi(b_n))\in\Acal^n$.

Since $\Acal$ satisfies the universal formula~(\ref{eq:universal_formuila_special}), the subformula
\[
\bigvee_{j=1}^l t^\pr_j=s^\pr_j
\]
should be true at $\phi(\bar{b})$. Therefore, there exists an index $j$ with $t^\pr_j(\phi(\bar{b}))=s^\pr_j(\phi(\bar{b}))$, a contradiction with~(\ref{eq:ineq222}).
\end{proof}

The role of Theorem~\ref{th:coord_then_ucl} is close to Theorem~\ref{th:coord_then_qvar}: for a fixed $\LL$-algebra $\Acal$ one can easily discover main properties of the irreducible coordinate algebras over $\Acal$. For example, Theorem~\ref{th:coord_then_ucl} states that any irreducible coordinate semilattice over the linearly ordered semilattice $L_n$ is always linearly ordered, since the property ``to be a linearly ordered semilattice'' is written as a universal formula of the language $\LL_{s}$ (see Section~\ref{sec:quasi-identities_and_univ_formulas}). 

Actually, for equationally Noetherian algebras one can prove the statement converse to Theorem~\ref{th:coord_then_ucl}. 

\begin{theorem}
If a finitely generated $\LL$-algebra $\Bcal$ belongs to the universal closure $\ucl(\Acal)$ generated by an equationally Noetherian $\LL$-algebra $\Acal$, then $\Bcal$ is isomorphic to the coordinate $\LL$-algebra of an irreducible algebraic set over $\Acal$.
\label{th:ucl_then_coord}
\end{theorem}
\begin{proof}
The proof of this theorem  is left to the reader, since it is close to the proof of Theorem~\ref{th:qvar_then_coord}.
\end{proof}

\begin{remark}
One can deduce the following principle from Theorem~\ref{th:ucl_then_coord}: if the universal theory of an $\LL$-algebra $\Acal$ is undecidable it is impossible to obtain a nice description of irreducible coordinate algebras over $\Acal$. For example, one can take the free nilpotent group $FN_n$ of rank $n>1$ and nilpotency class\Wiki $2$. It follows from model theory that $FN_n$ has an undecidable universal theory. Thus, it is impossible to obtain a description of irreducible coordinate groups over $FN_n$. Nevertheless, the coordinate groups of the solution sets of systems in one variable were described~\cite{Misch3}. 
\end{remark}

Actually, one can prove Theorem~\ref{th:ucl_then_coord} for any $\uu$-compact $\LL$-algebra $\Acal$, not merely for equationally Noetherian $\Acal$ (for more details see Section~\ref{sec:compactness_classes}).

According to Theorems~\ref{th:coord_then_ucl},~\ref{th:ucl_then_coord}, the class of irreducible coordinate algebras over an equationally Noetherian $\LL$-algebra $\Acal$ is a logical, i.e. it is defined by universal formulas. Let us give examples which show this property.

\begin{enumerate}
\item 

Let $A$ be an abelian group. Let us define a set of universal formulas $\Sigma^\pr$ as follows. Firstly, we put $\Sigma^\pr_A=\Sigma_A$, where the set of quasi-identities $\Sigma_A$ was defined in Section~\ref{sec:logic_coord}. Let us consider a set of universal formulas of the form
\begin{equation}
\varphi_{n,k}\colon\;\forall x_1\forall x_2\ldots\forall x_{n+1}(\bigwedge_{i=1}^{n+1}(p^kx_i=0)\bigwedge_{i=1}^{n+1}(p^{k-1}x_i\neq 0)\to \bigvee_{1\leq i<j\leq n+1}(x_i=x_j)).
\label{eq:n_elements_of_order_k}
\end{equation}
The formula $\varphi_{n,k}$ obviously states that there are at most $n$ elements of order $p^k$ in a group $A$.

Let us add to $\Sigma^\pr_A$ all formulas of the form $\varphi_{n,k}$ which are true in $A$.  

\begin{theorem}
Let $A=\Zbb^n\oplus T(A)$, $B=\Zbb^m\oplus T(B)$ be finitely generated abelian group. Then the following conditions are equivalent
\begin{enumerate}
\item $B$ is an irreducible coordinate group over $A$;
\item $T(B)$ is embedded into $T(A)$ and $m=0$ if $n=0$;
\item $B$ satisfies all universal formulas of the set $\Sigma^\pr_A$.
\end{enumerate}
\end{theorem}
\begin{proof}
The equivalence of the conditions $(a),(b)$ follows from Theorem~\ref{th:abelian_groups_irred_coord}. The implication $(a)\Rightarrow (c)$ follows from Theorem~\ref{th:coord_then_ucl}. Thus, it is sufficient to prove $(c)\Rightarrow(b)$. 

If $n=0$ the group $A$ has a finite period $d$ and the set $\Sigma_A^\pr$ contains a formula $\forall x\; (dx=0)$. Therefore $B$ has also a finite period and we have $m=0$.  

Let $T(B)$ is isomorphic to the following direct sum
\[
\Zbb_{n_1}\oplus\Zbb_{n_2}\oplus\ldots\oplus\Zbb_{n_t},
\]
where $n_i$ are powers of primes. Each group $\Zbb_{n_i}$ is generated by the element $1_{i}$ of order $n_i$. Since $B$ satisfies every formula $\varphi_{n,k}$~(\ref{eq:n_elements_of_order_k}) such that $\varphi_{n,k}$ is true in $A$, one can map the elements $1_i$ into different elements of  $T(A)$. Moreover, the elements $1_i,1_j$ ($i\neq j$) are mapped into different direct summands of $T(A)$. Following the formula~(\ref{eq:n_elements_of_order_k}), the elements $1_i$ and their images have the same order. Thus, $T(B)$ is embedded into $T(A)$.
\end{proof}

\item Since the monoid of natural numbers $\N$ is a co-domain (see Section~\ref{sec:co-domains}), Theorem~\ref{th:about_N_coord_logic} immediately gives the following result.

\begin{theorem}
A finitely generated commutative monoid $M$ is an irreducible coordinate monoid over $\N$ iff $M$ satisfies the following quasi-identities:
\[
\forall x\forall y\forall z (x+z=y+z\to x=y)\; \mbox{(cancellation property)},
\] 
\[
\forall x\forall y (x+y=0\to x=0)\; \mbox{(positiveness)}.
\] 
\label{th:about_N_coord_irred_logic}
\end{theorem}

\item 
\begin{theorem}
Let $L_n$ be the linearly ordered semilattice of order $n$. For a finite semilattice $S$ the following conditions holds $S$:
\begin{enumerate}
\item $S$ is an irreducible coordinate semilattice over $L_n$;
\item $S$ is embedded into $L_n$
\item $S$ satisfies the following universal formulas
\[
\forall x\forall y\; (x\leq y \vee y\leq x).
\]
\[
\forall x_1\forall x_2\ldots \forall x_{n+1}(\bigvee_{1\leq i<j\leq n+1}(x_i=x_j))
\]
\end{enumerate}
\end{theorem}
\begin{proof}
The equivalence of the conditions $(a),(b)$ follows from Theorem~\ref{th:discr_is_embedding_for_finite}. The implication $(a)\Rightarrow (c)$ follows from Theorem~\ref{th:coord_then_ucl}. Thus, it is sufficient to prove $(c)\Rightarrow(b)$. The first universal formula above states that $S$ is linearly ordered, and the second one asserts that $S$ contains at most $n$ elements. Thus, $S$ is isomorphic to $L_m$ for some $m\leq n$, therefore $S$ is obviously embedded into $L_n$.
\end{proof}
\end{enumerate}

\section{Geometrical equivalence}

Two $\LL$-algebras $\Acal_1,\Acal_2$ are \textit{geometrically equivalent} if for any  $\LL$-system $\Ss$ the coordinate $\LL$-algebras of the sets $\V_{\Acal_1}(\Ss)$, $\V_{\Acal_2}(\Ss)$ are isomorphic to each other. In other words, we have the radical equality $\Rad_{\Acal_1}(\Ss)=\Rad_{\Acal_2}(\Ss)$.

It follows from the definition that geometrically equivalent $\LL$-algebras $\Acal_1,\Acal_2$ have the same set of the coordinate algebras of algebraic sets. According to Theorem~\ref{th:coord_iff_approx} we obtain the following criterion of geometric equivalence.

\begin{theorem}
$\LL$-algebras $\Acal_1,\Acal_2$ are geometrically equivalent iff any finitely generated  $\LL$-algebra $\Bcal$ is approximated by $\Acal_1$ iff $\Bcal$ is approximated by $\Acal_1$.
\label{th:geom_equiv_first_criterion}
\end{theorem} 

One can simplify the statement of Theorem~\ref{th:geom_equiv_first_criterion} if $\Acal_1,\Acal_2$ are finitely generated.

\begin{theorem}
A finitely generated $\LL$-algebras $\Acal_1,\Acal_2$ are geometrically equivalent iff $\Acal_1$ is approximated by $\Acal_2$ and $\Acal_2$ is approximated by $\Acal_1$.
\label{th:geom_equiv_second_criterion}
\end{theorem} 
\begin{proof}
The ``only if'' part of the statement immediately follows from Theorem~\ref{th:geom_equiv_first_criterion}. Let us prove the ``if'' part. Let us consider a finitely generated $\LL$-algebra $\Bcal_1$ such that $\Bcal_1$ is approximated by $\Acal_1$, i.e. for arbitrary $b_1\neq b_2$ there exists a homomorphism   $\phi\in\Hom(\Bcal_1,\Acal_1)$ with $\phi(b_1)\neq\phi(b_2)$. By the condition, $\Acal_1$ is approximated by $\Acal_2$, therefore there exists a homomorphism $\phi^\pr\in\Hom(\Acal_1,\Acal_2)$ such that $\phi^\pr(\phi(b_1))\neq \phi^\pr(\phi(b_2))$. Thus, $\Bcal_1$ is approximated by $\Acal_2$, and Theorem~\ref{th:geom_equiv_first_criterion} gives the algebras $\Acal_1,\Acal_2$ are geometrically equivalent.
\end{proof}

Theorem~\ref{th:geom_equiv_second_criterion} allows us to apply Theorems~\ref{th:any_semilattice_is_coord},~\ref{th:unars_coord},~\ref{th:coord_abelian_groups} for the study of geometrically equivalent semilattices, unars and abelian groups.

\begin{enumerate}
\item All free unars $FU_n,FU_m$ are geometrically equivalent.
\item All finite nontrivial semilattices are geometrically equivalent .
\item A finitely generated abelian groups $A_1,A_2$ are geometrically equivalent iff $Sub_\oplus(A_1)=Sub_\oplus(A_2)$.
\item It follows from Section~\ref{sec:coord_alg_properties} that two rectangular bands are always geometrically equivalent except the following cases:
\begin{enumerate}
\item the the first band is a left zero semigroup, but the second one isn`t;
\item the the first band is a right zero semigroup, but the second one isn`t;
\item the the first band is trivial, but the second one isn`t;
\end{enumerate}

\item Let us prove that {\it a finitely generated commutative monoid with cancellations $M$ is geometrically equivalent to $\N$ iff $M$ is positive}. Indeed, in Theorem~\ref{th:about_N_coord_logic} it was proved that $M$ is approximated by $\N$. By the positiveness property, any one-generated submonoid in $M$ is isomorphic to $\N$, i.e. $\N$ is embedded into $M$. Thus, $\N$ is approximated by $M$ and Theorem~\ref{th:geom_equiv_second_criterion} concludes the proof.
\end{enumerate}

\section{Unifying theorems}
\label{sec:unify_theorems}

Above we describe coordinate algebras and irreducible coordinate algebras with different approaches: approximation, discrimination, embeddings, quasi-variety and universal closure. Actually, it is not a complete list of equivalent approaches in the description of (irreducible) coordinate algebras. In~\cite{uni_Th_II} it was formulated so-called Unifying theorems which contain seven (!) approaches in the description of coordinate algebras. Let us formulate these theorems (all unknown definitions can be found in~\cite{uni_Th_I,uni_Th_II}).

\begin{theorem}
\label{th:unify_approx}
Let $\Acal$ be an equationally Noetherian $\LL$-algebra. Then for a finitely generated $\LL$-algebra $\Bcal$ the following conditions are equivalent:
\begin{enumerate}
\item $\Bcal$ is the coordinate algebra of a nonempty algebraic set over $\Acal$; 
\item $\Bcal$ is approximated by $\Acal$;
\item $\Bcal$ is embedded into a direct power of  $\Acal$;
\item $\Bcal\in\qvar(\Acal)$;
\item $\Bcal$ belongs to the pre-variety generated by $\Acal$;
\item $\Bcal$ is a subdirect product of finitely many limit algebras over $\Acal$;
\item $\Bcal$ is defined by a complete atomic type of the theory $\mathrm{Th_{qi}}(\Acal)$.

\end{enumerate} 
\end{theorem}

\begin{theorem}
\label{th:unify_discr}
Let $\Acal$ be an equationally Noetherian $\LL$-algebra. Then for a finitely generated $\LL$-algebra $\Bcal$ the following conditions are equivalent:
\begin{enumerate}

\item $\Bcal$ is the coordinate algebra of an irreducible algebraic set over $\Acal$; 
\item $\Bcal$ is discriminated by $\Acal$;
\item $\Bcal$ is embedded into a ultrapower of $\Acal$;
\item $\Bcal\in\ucl(\Acal)$;
\item $\mathrm{Th_{\exists}}(\Acal)\supseteq \mathrm{Th_{\exists}}(\Bcal)$;
\item $\Bcal$ is a limit algebra over $\Acal$;
\item $\Bcal$ is defined by a complete atomic type of the theory $\mathrm{Th_{\forall}}(\Acal)$.

\end{enumerate} 
\end{theorem}

\begin{remark}
Fragments of Unifying theorems were appearing in many papers. In~\cite{BMR1,MR2} (in~\cite{Daniyarova1}) Theorems~\ref{th:unify_approx},~\ref{th:unify_discr} were proved for groups (for Lie algebras, respectively). Finally,~\cite{uni_Th_I,uni_Th_II} prove these theorems for an arbitrary algebra.
\end{remark}

\section{Appearances of constants}
\label{sec:appearing_of_const}

In the previous sections we consider groups monoid and semigroups in standard languages  $\LL_g,\LL_{m},\LL_s$. Commutative groups monoids and semigroups were considered in the additive languages $\LL_{+g},\LL_{+m},\LL_{+s}$. Thus, equations over groups, monoids and semigroups contain only trivial elements $1$ ($0$ in commutative case).  In the following sections we consider a general form of equations which may contain nontrivial constants.

Let $C=\{c_i\mid i\in I\}$ be a set of constant symbols. Below the extended language  $\LL\cup C$ is denoted by $\LL(C)$. 

For shortness we give the necessary definitions only for groups. The corresponding definitions for semigroups, monoid and algebras of an arbitrary language $\LL$ are similar. A group $G$ considered as an algebraic system of the language  $\LL_g(C)$ is called a {\it $C$-group}. A $C$-group $G$ differs from ``classic'' group $G$, since a $C$-group $G$ has a subgroup $C(G)$ generated by constants of the set $C$. If  $C(G)=G$ (i.e. any element of the group $G$ is a product of constants) a  $C$-group $G$ is called  \textit{Diophantine}.

\subsection{Equations with constants}

The language extension changes the classes of terms and equations. Let us show the  equations in groups and semigroups in languages with constants. 

\begin{enumerate}
\item Any $\LL_g(C)$-equation $t(X)=s(X)$ is equivalent over a Diophantine  $G$-group $G$ to an equality
\begin{equation}
w(X)=1,
\end{equation}
where $w(X)$ is a product of variables in integer powers and constants $g\in G$. Precisely, $w(X)$ is an element of the free product\Wiki $F(X)\ast G$, where $F(X)$ is the free group generated by the set $X$.

Similarly any $\LL_{+g}(C)$-equation over a Diophantine abelian $C$-group $A$ is equivalent to an  expression
\begin{equation}
\label{eq:abelian_group_equation_with_const}
k_1x_1+k_2x_2+\ldots+ k_nx_n=a,
\end{equation}
where $a\in A$.

\item By the definition, an $\LL_s(C)$-equation over a Diophantine $C$-semigroup $S$ is an expression
\begin{equation*}
t(X)=s(X),
\end{equation*}
where $t(X),s(X)$ are products of variables and constants $s\in S$. Any $\LL_{+m}(C)$-equation over a a commutative Diophantine $C$-monoid with cancellations $M$ is written as
\begin{equation}
\label{eq:cancellation_semigroup_equation_with_const}
\sum_{i\in I}k_ix_i+a=\sum_{j\in J}k_jx_j+a^\pr,
\end{equation}
where $a,a^\pr\in M$, $I,J\subseteq \{1,2,\ldots,n\}$, $I\cap J=\emptyset$ (i.e. each variable occurs in at most one part of the equation).
\end{enumerate}

\subsection{New algebraic sets}

In the previous section we showed that an extension of the language changes the classes of terms and equations. Therefore we obtain a wide class of algebraic sets over extended languages. Moreover, a language extension may change the class of irreducible algebraic sets. Let us give examples which show such properties of an extension (below $G$ is a Diophantine $C$-group).

\begin{enumerate}
\item {\it Any point $(a_1,a_2,\ldots,a_n)\in G^n$ is algebraic set over $G$} (actually, this fact holds for any Diophantine $C$-semigroup and $C$-monoid), since it is the solution set of the following $\LL_g(C)$-system $\Ss=\{x_i=a_i\mid 1\leq i\leq n\}$. Moreover, for $|G|>1$ the empty set is algebraic over $G$: $\V_{\Acal}(c=c^\pr)=\emptyset$, where $c,c^\pr$ are different elements of $G$.

\item {\it There is an algebraic set $Y\subseteq G^n$ such that $Y$ is irreducible over a group $G$, but $Y$ is reducible over a Diophantine $C$-group $G$}. Let $G=\Zbb_2=\{0,1\}$. The algebraic set $Y=G=\V_G(x=x)$ is irreducible over an abelian group $G$ of the language $\LL_{+g}$, since $\Gamma_{G}(Y)\cong\Zbb_2$ and $\Gamma_G(Y)$ is obviously discriminated by $G$. Let us extend the language $\LL_{+g}$ by new constants $C=\{0,1\}$ and obtain an abelian Diophantine $C$-group $G$. The set $Y$ becomes reducible over the $C$-group $G$, since $Y=\V_G(x=0)\cup \V_G(x=1)$.

Moreover, one can prove  that {\it a finite algebraic set $Y$ is irreducible over a Diophantine  $C$-group $G$ iff $|Y|=1$}.

\item {\it If a semigroup $S$ is a co-domain, the Diophantine $C$-semigroup $S$ is not always a co-domain.} Indeed, in Section~\ref{sec:co-domains} we proved that the commutative monoid $\N$ is a co-domain in the language $\LL_{+m}$. However in the language $\LL_{+m}(C)$ there are reducible algebraic  sets over the Diophantine $C$-monoid:
\[
\V_\N(x+y=1)=\V_\N(\{x=0,y=1\})\cup \V_\N(\{x=1,y=0\}).
\]

\end{enumerate}

\subsection{Equationally Noetherian algebras with constants}
\label{sec:noeth_const}

\begin{enumerate}

\item {\it If an $\LL$-algebra $\Acal$ is equationally Noetherian the $\LL(C)$-algebra $\Acal$ is not necessarily equationally Noetherian.}  Indeed, let $L_\infty=\{a_1,a_2,a_3,\ldots\}$ be an infinite linearly ordered semilattice with $a_i<a_{i+1}$. The semilattice $L_\infty$ is equationally Noetherian as an algebra of the language $\LL_s$ ($L_\infty$ belongs to a locally finite variety, see Section~\ref{sec:noeth}). However, for the Diophantine $C$-semilattice $L_\infty$ there is an infinite $\LL_s(C)$-system $\Ss=\{x\geq a_n\mid n\geq 1\}$ such that $\Ss$ is not equivalent to any finite subsystem. Thus, the Diophantine $C$-semilattice $L_\infty$ is not equationally Noetherian.

Moreover, for the variety of semilattices we have the following result.

\begin{proposition}\textup{~\cite{shevl_at_service}}
A Diophantine $C$-semilattice $S$ is equationally Noetherian iff $S$ is finite. 
\label{pr:noeth_criterion_for_semilattices_with_const}
\end{proposition}

\item Let us consider a Diophantine nilpotent $C$-group of the nilpotency class $2$ given by the presentation
\[
G=\lb a_1,a_2,\ldots,a_n,\ldots,b_1,b_2,\ldots,b_n,\ldots\mid [a_i,b_j]=1\; i\leq j\rb.
\] 
An infinite $\LL_{g}(C)$-system $\Ss=\{[x,a_i]=1\mid i\geq 1\}$ is not equivalent to any finite subsystem (a finite subsystem $\Ss^\pr=\{[x,a_i]=1\mid 1\leq i\leq n\}$ is satisfied by the point $x=b_n$, but $b_n\notin\V_G(\Ss)$). Thus, the Diophantine $C$-group $G$ is not equationally Noetherian. However, the group $G$ in the standard group language $\LL_g$ is equationally Noetherian (actually, one can prove that any nilpotent group of the nilpotency class $2$ is equationally Noetherian). 

\item {\it All classes of algebras considered in Section~\ref{sec:noeth} (abelian and linear groups, free semilattices, rectangular bands, left zero semigroups) remain equationally Noetherian in extended languages} $\LL(C)$. We give this fact with no proof. 

\end{enumerate}

\section{Coordinate algebras with constants}

In Section~\ref{sec:coord_algebras} it was given the definition of the coordinate $\LL$-algebra over an arbitrary language $\LL$ (including languages with an arbitrary number of constant symbols). Thus, the reader can apply this general definition for $C$-groups and $C$-semigroups. We remark below only the principal differences between coordinate algebras in constant-free languages and in languages with constants. {\it For simplicity we deal with groups in all examples below}.

Firstly, the coordinate $C$-group $\Gamma_G(Y)$ of an algebraic set $Y$ over a Diophantine $C$-group $G$ is not always Diophantine. Thus, all elements of a  $C$-group $\Gamma_G(Y)$ can be divided into three groups: 
\begin{enumerate}
\item {\it constants}, i.e. the images of interpretations of constant symbols of the language $\LL_{g}(C)$ (Proposition~\ref{pr:A_embedded_into_Gamma} describes below a subgroup generated by constants);
\item {\it ``transcendental'' elements}, i.e. the elements which cannot be obtained from constants;
\item results of functions applied to transcendental and constants elements. 
\end{enumerate}

Let us consider a nonempty algebraic set $Y$ over a Diophantine $C$-group $G$. By the definition, distinct constants  $a,a^\pr\in G$ are $\LL_{g}(C)$-terms and $a\nsim_Y a^\pr$ (i.e. $[a]_Y\neq [a^\pr]_Y$ in $\Gamma_G(Y)$). Thus, we obtain the following statement.

\begin{proposition}
The coordinate $C$-group $\Gamma_\Acal(Y)$ for $Y\neq\emptyset$ contains a subgroup isomorphic to $G$, and this subgroup is generated by constants (this statement holds for any Diophantine $C$-algebra not merely for a group).
\label{pr:A_embedded_into_Gamma}
\end{proposition}

Remark that the set $Y=\emptyset$ makes Proposition~\ref{pr:A_embedded_into_Gamma} false, since the relation $\sim_\emptyset$ over the set of $\LL_g(C)$-terms generates a unique equivalence class,  and the coordinate $C$-algebra $\Gamma_G(Y)$ is trivial.  

Thus, a constant subgroup of $\Gamma_G(Y)$ is isomorphic to $G$, so $G$ is embedded into $\Gamma_G(Y)$. The principal question is~\textit{a way of this embedding, because there is often at least one embedding of $G$ into  $\Gamma_G(Y)$}. The following example shows the differences in algebraic geometries for different embeddings of $G$ into $\Gamma_G(Y)$.

Let $G_1,G_2$ be two isomorphic copies of the abelian group $\Zbb$, but $G_1$ is Diophantine (i.e. all elements of $G_1$ are marked as constants) whereas the set of constants of $G_2$ is $2\Zbb=\{0,\pm 2,\pm 4,\ldots\}$. It is clear that $G_1,G_2$ are abelian $C$-groups in the language $\LL_{+g}(C)$ and the constant subgroups of  $G_1,G_2$ are isomorphic to $\Zbb$. The abelian $C$-group $G_1$ is approximated by the $C$-group $\Zbb$ (since $G_1\cong \Zbb$). However $G_2$ is not approximable by the $C$-group $\Zbb$ (by the definition of a homomorphism $\phi\colon G_2\to \Zbb$, the constants of $G_2$ should be mapped into the constants of $\Zbb$; therefore the relation $1+1=2$ becomes $\phi(1)+\phi(1)=1$, since $2\in G_2$ and $1\in\Zbb$ are the interpretations of the same constant symbol; but there is not any element $x$ in $\Zbb$ satisfying the equality $x+x=1$). Since Unifying theorem~\ref{th:unify_approx} was proved for an arbitrary language $\LL$, it also holds for abelian groups of the language $\LL_{+g}(C)$. Thus, the $C$-group $G_1$ is the coordinate  $C$-group of an algebraic set over the abelian Diophantine $C$-group $\Zbb$, but $G_2$ isn't.

Let us give examples of coordinate algebras over Diophantine $C$-groups and $C$-monoids.

\begin{enumerate}
\item 
\begin{proposition}
The coordinate $C$-group of an algebraic set $Y$ over a Diophantine $C$-group $G$ is isomorphic to $G$ iff $|Y|=1$ (actually, this statement holds for any Diophantine $C$-algebra, not merely for groups). 
\end{proposition}
\begin{proof}
If a set $Y$ consists of a single $P=(p_1,p_2,\ldots,p_n)$ the definition of the equivalence relation   $\sim_Y$ gives
\[
t(X)\sim_Y s(X)\Leftrightarrow t(P)=s(P).
\] 
Therefore for each $\LL_g(C)$-term $t(X)$ there exists a unique constant $a\in G$ with $t(X)\sim_Y a$ ($t(P)=a$), and the equivalence class $[t(X)]_Y$ is uniquely defined by a constant $a\in[t(X)]_Y$. Thus, the factor-algebra $\Gamma_G(Y)=\Tcal(X)/\sim_Y$ is isomorphic to the group generated by constants, i.e. $\Gamma_G(Y)\cong G$.

Let us assume that $Y$ contains a least two points $P=(p_1,p_2,\ldots,p_n)$, $Q=(q_1,q_2,\ldots,q_n)$. Without loss of generality we put $p_1\neq q_1$. Therefore an $\LL$-term $t(X)=x_1$ is not $\sim_Y$-equivalent to any constant $a\in G$. Thus, the coordinate $C$-group $\Gamma_G(Y)$ is not equivalent to $G$.
\end{proof}

\item Let us consider an equation $2x+3y+5z=5$ over the Diophantine $C$-monoid $\N$. Obviously, its solution set is $Y=\{(1,1,0),(0,0,1)\}$, therefore the equation $x=y$ belongs to the radical of $Y$. Thus, the coordinate $C$-monoid for the set $Y$ is a commutative $C$-monoid given by the following presentation
\begin{multline*}
\Gamma_\Acal(Y)=\lb x,y,z\mid \Rad_\N(2x+3y+5z=5)\rb\cong
\lb x,y,z\mid \Rad_\N(2x+3y+5z=5),x=y\rb\cong\\
\lb x,y,z\mid \Rad_\N(2x+3x+5z=5)\rb\cong
 \lb x,z\mid \Rad_\N(5x+5z=5)\rb\cong \lb x,z\mid \Rad_\N(x+z=1)\rb.
\end{multline*}
It is directly checked that $\Rad_\N(x+z=1)$ coincides with the congruent closure $[x+z=1]$, therefore
\[
\Gamma_\Acal(Y)\cong \lb x,z\mid x+z=1\rb.
\] 
The defined relation $x+y=1$ describes the way of an embedding of the constant monoid $\N$ into $\Gamma_\Acal(Y)$
\end{enumerate}

\subsection{Coordinate groups over abelian groups}

In this section we describe coordinate $C$-groups over a finitely generated abelian Diophantine $C$-group $A$. Moreover we compare algebraic geometries over Diophantine group $\Zbb$ and monoid $\N$.

By Proposition~\ref{pr:A_embedded_into_Gamma}, a group $A$ is embedded into any coordinate $C$-group of a nonempty algebraic set over $A$. The following propositions specify the properties of this embeddings.

A subgroup $A$ of an abelian group $G$ is {\it pure} if any consistent in $G$ equation $nx=a$ ($a\in A$) is also consistent in $A$. According to group theory, we have the following statement.

\begin{proposition}
Any  pure subgroup $A$ of a finitely generated abelian group $G$ is a direct summand of  $G$, i.e. $G\cong A\oplus G^\pr$ for some finitely generated group $G^\pr$. 
\label{pr:servant_property_from_group_theory}
\end{proposition}

\begin{proposition}
Let $\Gamma_A(Y)$ be the coordinate $C$-group of a nonempty algebraic set over an abelian Diophantine group $A$. Then $A$ is pure in $\Gamma_A(Y)$. 
\label{pr:coord_group_has_servant_A}
\end{proposition}
\begin{proof}
By Unifying theorem~\ref{th:unify_approx}, the $C$-group $\Gamma_A(Y)$ should satisfy all quasi-identities $\varphi$ of the language $\LL_{+g}(C)$ such that $\varphi$ holds in $A$. Since $A$ is Diophantine, one can explicitly use any element $a\in A$ in formulas. Let us consider the following quasi-identity for any equation $nx=a$ with $\V_A(nx=a)=\emptyset$
\begin{equation}
\varphi_{\emptyset,n,a}\colon \forall x\; (nx=a\to x=0).
\label{eq:quasi-identity_for_servant}
\end{equation}  
Let us show that the quasi-identity  $\varphi_{\emptyset,n,a}$ is true in $A$. Indeed, the left part of the implication in $\varphi_{\emptyset,n,a}$ does not hold for any $x\in A$ (since the equation $nx=a$ is inconsistent in $A$), therefore the whole implication is true in $A$.

By Theorem~\ref{th:unify_approx}, all formulas $\varphi_{\emptyset,n,a}$ should hold in  $\Gamma_A(Y)$, therefore any inconsistent over $A$ equation $nx=a$ remains inconsistent in $\Gamma_A(Y)$. Thus, $A$ is a pure subgroup of $\Gamma_A(Y)$.
\end{proof}

By Propositions~\ref{pr:servant_property_from_group_theory},~\ref{pr:coord_group_has_servant_A} we obtain the following statement about coordinate $C$-groups over an abelian Diophantine group $A$.

\begin{proposition}
Suppose an abelian $C$-group $\Gamma_A(Y)$ is the coordinate $C$-group of a nonempty algebraic set over a finitely generated abelian Diophantine group $A$. Then $\Gamma_A(Y)$ is isomorphic to a direct sum $A\oplus G^\pr$ for some finitely generated group $G^\pr$. 
\label{pr:coord_group_has_direct_summand_A}
\end{proposition}

Thus, we obtain the following theorem.

\begin{theorem}
Let $A,B$ be finitely generated abelian $C$-groups, and $A$ is Diophantine. Then the following conditions are equivalent:
\begin{enumerate}
\item $B$ is the coordinate $C$-group of a nonempty algebraic set over $A$; 
\item $B$ is a direct sum of groups from $Sub_\oplus(A)$ and $B=A\oplus B^\pr$ for some subgroup  $B^\pr\subseteq B$;
\item $B$ satisfies all quasi-identities $\Sigma_{AA}$, where $\Sigma_{AA}$ consists of the set $\Sigma_{A}$ (see Theorem~\ref{th:coord_abelian_groups}) and all quasi-identities of the form~(\ref{eq:quasi-identity_for_servant}) which hold in $A$.
\end{enumerate}
\label{th:coord_abelian_groups_const}
\end{theorem}
\begin{proof}
By Theorem~\ref{th:coord_then_qvar} we obtain the implication $(1)\Rightarrow(3)$.

Let us prove $(2)\Rightarrow(1)$. Following Unifying theorem~\ref{th:unify_approx}, it is sufficient to prove that $B$ is approximated by $A$. Let $b_1,b_2\in B$ be two distinct elements of $B$. We find a homomorphism $\phi\in\Hom(B,A)$ such that $\phi(b_1)\neq\phi(b_2)$ (recall that a homomorphism of  $C$-groups acts trivially on the subgroup of constants $A\subseteq B$). By the definition of a direct sum, the elements $b_i$ admit a unique presentation as sums $b_i=a_i+b_i^\pr$, $a_i\in A$, $b_i^\pr\in B^\pr$. If $a_1\neq a_2$ one can define $\phi$ as th projection onto $A$: $\phi(a+b^\pr)=a$. Let us assume now $a_1=a_2=a$. The group $B^\pr$ satisfies the conditions of Theorem~\ref{th:coord_abelian_groups},  therefore $B^\pr$ is approximated (as a group in the language $\LL_{+g}$) by  $A$. Thus, there exists a homomorphism $\psi\colon B^\pr\to A$ such that $\psi(b_1^\pr)\neq\psi(b_2^\pr)$, and the homomorphism of $C$-groups $\phi(a+b^\pr)=a+\psi(b^\pr)$ ($\phi$ is a homomorphism of abelian $C$-groups, since it acts trivially on the set $A\subseteq B$) has the property $\phi(b_1)\neq\phi(b_2)$.

Let us prove $(3)\Rightarrow(2)$. Since $B$ satisfies the formulas $\Sigma_A\subseteq\Sigma_{AA}$, Theorem~\ref{th:coord_abelian_groups_with_logic} states that $B$ is a direct sum of groups from the set $Sub_\oplus(A)$. The quasi-identities ~(\ref{eq:quasi-identity_for_servant}) from $\Sigma_{AA}$ provide the purity of $A$ in $B$, therefore $A$ is embedded into $B$ as a direct summand (Proposition~\ref{pr:servant_property_from_group_theory}).  
\end{proof}

\begin{corollary}
An abelian $C$-group $B$ is the coordinate $C$-group of a nonempty algebraic set over an abelian Diophantine $C$-group $A=\Zbb$ iff $B$ is isomorphic to a direct sum $A\oplus \Zbb^n$ for some $n\geq 0$.
\label{cor:about_Z_const}
\end{corollary}

Unlike the group $\Zbb$, the class of coordinate $C$-monoids over the Diophantine monoid of natural numbers $\N$ is not trivial (see~\cite{shevl_over_N_qvar} for the axiomatization of $\qvar(\N)$).

\subsection{Irreducible coordinate groups over abelian groups}

For irreducible coordinate $C$-groups over an abelian Diophantine $C$-group $A$ we have the following theorem.

\begin{theorem}
Let $A,B$ be a finitely generated abelian $C$-groups and $A$ is Diophantine. Then the following conditions are equivalent:
\begin{enumerate}
\item $B$ is the coordinate $C$-group of an irreducible algebraic set over $A$; 
\item $B=A\oplus B^\pr$, where
\[
B^\pr=\begin{cases}
\{0\}\mbox{ if  $A=T(A)$},\\
\Zbb^n\mbox{ if  $A$ contains $\Zbb$ as a direct summand}
\end{cases}
\]
\item $B$ satisfies all universal formulas of the set $\Sigma_{AA}\cup\Sigma_A^\pr$ (see Theorems~\ref{th:abelian_groups_irred_coord},~\ref{th:coord_abelian_groups_const}).
\end{enumerate}
\label{th:coord_abelian_groups_const_irred}
\end{theorem}
\begin{proof}
Theorem~\ref{th:coord_then_ucl} immediately proves the implication $(1)\Rightarrow(3)$.

Let us prove $(2)\Rightarrow(1)$. By Unifying theorem~\ref{th:unify_discr}, it is sufficient to prove that $B$ is discriminated by the Diophantine $C$-group $A$. If $B^\pr=\{0\}$ then $B\cong A$ and $B$ is obviously discriminated by $A$. Let $B=A\oplus \Zbb^n$ and $A=\Zbb\oplus A^\pr$. Therefore any element of  $B$ is written as $b=a^\pr+a+z$, where $a^\pr\in A^\pr$, $a\in\Zbb\subseteq A$, $z\in\Zbb^n$. Let us consider a set of elements $b_i\in B$ given by sums $b_i=a_i^\pr+a_i+z_i$ ($1\leq i\leq k$).

According to Theorem~\ref{th:abelian_groups_irred_coord}, for any finite set of elements $z_1,z_2,\ldots,z_k\in \Zbb^n$ there exists a group homomorphism $\zeta\colon \Zbb^n\to \Zbb$ with $\zeta(z_i)\neq\zeta(z_j)$ for $z_i\neq z_j$. Let us consider a map $\phi\colon B\to A$, $\phi(b)=\phi(a^\pr+a+z)=a^\pr+(a+d\cdot\zeta(z))$ (the brackets contain an elements of the subgroup $\Zbb\subseteq A$), where the integer $d$ will be defined below. It is clear that $\phi$ is a homomorphism of $C$-groups, since it preserves the operations $+,-$ and it acts trivially on $A\subseteq B$. Let us choose $d$ which makes $\phi$ injective on the set $\{b_i\}$. If the equality $\phi(z_i)=\phi(z_j)$ holds there appears a single constraint for $d$:
\[
a_i^\pr +a_i+d\zeta(z_i)=a_i^\pr+a_j+d\zeta(z_j)\Rightarrow d=\frac{a_i-a_j}{\zeta(z_j)-\zeta(z_i)}.
\]
Iterating all possible pairs $b_i,b_j$, we obtain at most finite number of constraints for $d$. Therefore there exists an integer value of $d$ such that $\phi$ is injective on $\{b_i\}$. Thus, $B$ is discriminated by $A$.

Let us prove $(3)\Rightarrow(2)$. The quasi-identities of the form~(\ref{eq:quasi-identity_for_servant}) from $\Sigma_{AA}$ provides the purity of the group $A$ in $B$. By Proposition~\ref{pr:servant_property_from_group_theory}, $A$ is embedded into $B$ as direct summand. According the quasi-identities $\Sigma_A\cup\Sigma_A^\pr$ and Theorem~\ref{th:coord_then_qvar}, we obtain the desired structure of $B$. 
\end{proof} 

\begin{corollary}
A finitely generated $C$-group $B$ is the coordinate $C$-group of an irreducible algebraic set over the Diophantine $C$-group $A=\Zbb$ iff $B\cong A\oplus\Zbb^n$ for some $n\geq 0$. Applying Corollary~\ref{cor:about_Z_const}, we obtain that the abelian Diophantine $C$-group $\Zbb$ is a co-domain in the language $\LL_{+g}(C)$.
\label{cor:about_Z_const_irred}
\end{corollary}

Unlike $\Zbb$, the Diophantine $C$-monoid of natural numbers $\N$ has a complicated set of irreducible $C$-monoids.

\begin{theorem}\textup{\cite{shevl_over_N_irred}}
Let $\N$ be the Diophantine $C$-monoid of natural numbers. Consider a finitely generated $C$-monoid $M$ of the language $\LL_{+m}$ such that the set of homomorphisms of $C$-monoids $\Hom(M,\N)$ is nonempty. The $C$-monoid $M$ is the coordinate $C$-monoid of an irreducible algebraic set over $\N$ iff the constant submonoid of $M$ is isomorphic to $\N$, and for any $a\in \N$ $M$ satisfies the following universal formula of the language $\LL_{+m}(C)$ 
\[
\varphi_{a}\colon\; \forall x\forall y\; (x+y=a\to\bigvee_{i=0}^a(x=i)).
\]
The sense of the formula $\varphi_a$ is the following: if a sum of elements $x,y\in M$ belongs to the submonoid $\N\subseteq M$, then both summands $x,y$ belong to $\N\subseteq M$. Equivalently, 
\[
x,y\in M\setminus\N\Rightarrow x+y\in M\setminus\N.
\]
\end{theorem} 
\begin{proof}
The ``only if'' part of the theorem follows from Theorem~\ref{th:coord_then_ucl}, since the formulas $\varphi_a$ obviously hold in $\N$. Let us prove the ``if'' part. By Unifying theorem~\ref{th:unify_discr}, it is sufficient to prove that for any finite set of distinct elements  $m_1,m_2,\ldots,m_n\in M$ there is a homomorphism $\phi\in\Hom(M,\N)$ injective on $\{m_i\}$.

Since $\N\subseteq M$, one can define an equivalence relation on $M$ as follows:
\[
m\sim m^\pr\Leftrightarrow \exists a_1,a_2\in\N\colon\; m+a_1=m^\pr+a_2.
\]

Let $M^\pr=M/\sim$ be the factor-monoid (the zero of $M^\pr$ is the equivalence class $[0]=\N\subseteq M$). Let $\chi\colon M\to M^\pr$ denote the canonical homomorphism between $M$ and $M^\pr$.

It is easy to check that the monoid $M^\pr$ of the language $\LL_{+m}$ is positive and has the cancellation property (the positiveness of $M^\pr$ follows from the implication $x,y\in M\setminus\N\Rightarrow x+y\notin \N$ and $[x]+[y]\neq [0]$). By Theorem~\ref{th:about_N_coord_irred_logic}, $M^\pr$ is discriminated (as a monoid of the language $\LL_m$) by the monoid $\N$. 

Therefore, there exists a monoid homomorphism $\zeta\colon M^\pr\to \N$ such that $\zeta(\chi(m_i))\neq \zeta(\chi(m_j))$ if $\chi(m_i)\neq\chi(m_j)$. Let us consider a map $\phi(m)=\psi(m)+d\zeta(\chi(m))$, where $\psi$ is an arbitrary element of the nonempty set $\Hom(M,\N)\neq\emptyset$ and a number $d$ will be defined below.

We have exactly two cases.

1) Let us assume that $\chi(m_i)\neq\chi(m_j)$ for any $i\neq j$.
Since $M$ is positive, $M$ has not a torsion and the equality
\[
\phi(m_i)=\psi(m_i)+d\zeta(\chi_A(m_i))=\psi(m_j)+d\zeta(\chi_A(m_j))=\phi(m_j).
\]
holds only for a unique $d$. Since the set $\{m_i\}$ is finite there exists a number $d$ such that $\phi$ is injective on $m_1,\ldots,m_k$.

2) If $\chi(m_i)=\chi(m_j)$ some $i\neq j$ then
$m_i+a=m_j+a^\prime$ for $a,a^\prime\in \N$, $a\neq a^\pr$. It is directly checked that the equality $\phi(m_i)=\phi(m_j)$ implies $\phi(a)=\phi(a^\prime)$. The last equality is impossible, since $\phi$ is a homomorphism of $C$-monoids, and $\phi$ acts trivially on the set of constants.

Thus, $\phi$ has the property $\phi(m_i)\neq\phi(m_j)$ ($i\neq j$), and the $C$-monoid $\N$ discriminates  $M$. 
\end{proof}

\section{Equational domains}
\label{sec:equational_domains}
Recall that in Section~\ref{sec:co-domains} we studied co-domains. A co-domain is an $\LL$-algebra $\Acal$, where any nonempty algebraic set is irreducible. In the current section we consider $\LL$-algebras where \textit{any} finite union of algebraic sets remains algebraic.

Let us give the main definition, following~\cite{uni_Th_IV}. An $\LL$-algebra $\Acal$ is called an {\it equational domain (e.d.)} if any finite union  $Y_1\cup Y_2\cup\ldots\cup Y_n$ of algebraic sets over $\Acal$ is an algebraic set. Let us give examples which show the properties of this definition.

\begin{enumerate}
\item Obviously, one can simplify the definition of e.d. a follows: an $\LL$-algebra $\Acal$ is called an {\it equational domain (e.d.)} if any union of two algebraic sets  $Y_1,Y_2$ over $\Acal$ is an algebraic set.

\item It is easy to see that {\it the trivial  $\LL$-algebra  $\Acal$ (i.e.  $|\Acal|=1$) is an e.d}. The proof of this fact immediately follows from  $|\Acal^n|=1$ and there is a unique algebraic set in the affine space  $\Acal^n$.
 
\item {\it Any nontrivial group $G$ of the language $\LL_g$ is not an e.d.} To  prove this fact we consider the union of the following algebraic sets over $G$
\[
M=\{(x,y)\in G^2\mid x=1\mbox{ or }y=1 \}=\V_\Acal(x=1)\cup\V_\Acal(y=1).
\]
Let us assume that $M$ is the solution set of some system $\Ss(x,y)$, and $w(x,y)=1$ be an arbitrary equation from $\Ss(x,y)$, where 
\[
w(x,y)=x^{n_1}y^{m_1}x^{n_2}y^{m_2}\ldots x^{n_k}y^{m_k},\; m_i,n_i\in\Zbb.
\]

Since any point $(1,y)\in M$ should satisfy the equation  $w(x,y)=1$, we obtain that the following identity holds in $G$
\[
y^{m_1+m_2+\ldots+m_k}=1
\] 
Similarly, the points $(x,1)\in M$ gives the truth of the identity
\[
x^{n_1+n_2+\ldots+n_k}=1.
\] 
in $G$.

Since $G$ is nontrivial, there exists an element $g\in G\setminus\{1\}$. Computing $w(g,g)=1$, we obtain
\[
w(g,g)=g^{n_1}g^{m_1}g^{n_2}g^{m_2}\ldots g^{n_k}g^{m_k}=(g^{n_1+n_2+\ldots+n_k})(g^{m_1+m_2+\ldots+m_k})=1\cdot 1=1,
\]
i.e. $w(x,y)=1$ is satisfied by $(g,g)$. Since we arbitrarily chose the equation $w(x,y)=1\in\Ss$, we obtain $(g,g)\in\V_G(\Ss)$ that contradicts $(g,g)\notin M=\V_G(\Ss)$.

\item In~\cite{shevl_domains_without_const} it was proved that  {\it any nontrivial semigroup of the language $\LL_s$ is not an e.d.} 

\item Let $\LL_r=\lb +^{(2)},-^{(2)},\cdot^{(2)},0\rb$ be the language of ring\Wiki theory. An arbitrary field\Wiki $F$ is an $\LL_r$-algebra with natural interpretations of the functional symbols. Applying the axioms of field theory, one can show that any equation in variables $X=\{x_1,x_2,\ldots,x_n\}$ is equivalent over $F$ to an expression of the form
\begin{equation}
\label{eq:field_equation}
\sum_i x_1^{m_{1i}}x_2^{m_{1i}}\ldots x_n^{m_{1n}}=0.
\end{equation}

Let us prove that {\it a field $F$ (as an $\LL_r$-algebra) is an e.d. } The absence of {\it zero-divisors} plays a key role in this proof. Recall that elements $a,b\neq 0$ are called zero-divisors if $ab=0$.

Consider a union $Y=Y_1\cup Y_2$ of algebraic sets $Y_1,Y_2$ over $F$. There exist systems $\Ss_1=\{t_i(X)=0\mid i\in I\}$,  $\Ss_2=\{s_j(X)=0\mid j\in J\}$ with the solution sets $Y_1,Y_2$ respectively. The direct check gives that the solution set of the system
\[
\Ss=\{t_i(X)s_j(X)=0\mid i\in I,j\in J\}
\]
equals $Y$. Thus, $Y$ is an algebraic set. 

The obtained result can be applied to a wide class of $\LL_r$-algebras. Precisely, {\it any commutative associative ring (as an $\LL_r$-algebra) with no zero-divisors is an e.d}.

\end{enumerate}

The main result of this section is the following theorem.

\begin{theorem}
An $\LL$-algebra $\Acal$ is an e.d. iff the set of points
\begin{equation}
\label{eq:domain_main_set}
M=\{(x_1,x_2,x_3,x_4)\in\Acal^4\mid x_1=x_2\mbox{ or }x_3=x_4 \}=\V_\Acal(x_1=x_2)\cup\V_\Acal(x_3=x_4)
\end{equation}
is algebraic over $\Acal$.
\label{th:domain_about_main_set}
\end{theorem}
\begin{proof}
The ``only if'' part of the theorem obviously follows from the definition of an e.d. Let us prove the ``if'' part of the statement. Suppose $M$ is the solution set of a system $\Ss(x_1,x_2,x_3,x_4)$. 

Firstly, we show that the union of the solution sets of two equations $t_1(X)=t_2(X)$, $t_3(X)=t_4(X)$ is algebraic over $\Acal$. Indeed, changing all variables $x_i$ in $\Ss(x_1,x_2,x_3,x_4)$ to the terms $t_i(X)$, we obtain
\[
\V_\Acal(t_1(X)=t_2(X))\cup\V_\Acal(t_3(X)=t_4(X))=\V_\Acal(\Ss(t_1(X),t_2(X),t_3(X),t_4(X))),
\]
and therefore the union of the solution set of two equations is an algebraic set over $\Acal$.

Let us consider the union of two algebraic sets $Y_1,Y_2$ given by systems of equations $\Ss_1=\{t_i(X)=s_i(X)\mid i\in I\}$,  $\Ss_2=\{p_j(X)=r_j(X)\mid j\in J\}$. In other words,
\[
Y_1=\bigcap_{i\in I}\V_\Acal(t_i(X)=s_i(X)),\;Y_2=\bigcap_{j\in J}\V_\Acal(p_j(X)=r_j(X)).
\]
By the distributive law of set intersection, we obtain
\[
Y=\bigcap_{i\in I,j\in J}(\V_\Acal(t_i(X)=s_i(X))\cup\V_\Acal(p_j(X)=r_j(X)))=\bigcap_{i\in I,j\in J}\Ss(t_i(X),s_i(X),p_j(X),r_j(X)),
\] 
i.e. $Y$ is algebraic.
\end{proof}

The statement of Theorem~\ref{th:domain_about_main_set} gives the following observation. \textit{Extending a language $\LL$ by new constants symbols $C$, we enrich the class of algebraic sets. Thus, it is possible that the set $M$~(\ref{eq:domain_main_set}) is not algebraic over a group (semigroup)  $\Acal$, but $M$ becomes algebraic over the Diophantine $C$-group ($C$-semigroup) $\Acal$. In other words, the class of e.d. in the languages $\LL_g(C)$ (respectively, $\LL_s(C)$) is wider than in the languages $\LL_g$ (respectively, $\LL_s$), and the following examples demonstrate this effect}.

\begin{enumerate}
\item In~\cite{uni_Th_IV} it was proved the following criterion for e.d. in the class of Diophantine $C$-groups.

\begin{theorem}
A Diophantine $C$-group $G$ is an e.d. iff there is not a pair of elements $x,y\in G\setminus\{1\}$ such that $[x,y^g]=1$ for every $g\in G$ (the denotation $y^g$ equals $g^{-1}yg$).
\label{th:domains_for_groups}
\end{theorem}

According to Theorem~\ref{th:domains_for_groups}, the following Diophantine $C$-groups are e.d. in the language $\LL_g(C)$.
\begin{enumerate}
\item {\it The free group $FG_n$ of rank $n>1$ is an e.d.} Let us prove this fact. Indeed, if assume the existence of elements $x,y\neq 1$ such that $[x,y^g]=1$ for all $g\in FG_n$ we obtain for $g=1$ that $x$ commutes with $y$. By the properties of the free groups, there exists $w\in FG_n$ with $x=w^k$, $y^g=w^l$ for some $k,l\in\Zbb$. However it is easy to find an element $h\in FG_n$ such that $y^h$ is not a power of $w$, and the equality $[x,y^{h}]=1$ fails.

\item {\it Any simple\Wiki non-abelian Diophantine $C$-group $G$ is an e.d.}. Let us assume the existence of elements $x,y\in G\setminus\{1\}$ such that $[x,y^g]=1$ for all $g\in G$. By the assumption, the set 
\[
G_x=\{y\mid [x,y^g]=1\mbox{ for all }g\in G\}
\]
is nonempty. If $y\in G_x$ then for $y^{-1}$ it holds
\[
[x,(y^{-1})^g]=[x,(y^g)^{-1}]=[y^g,x]^{(y^g)^{-1}}=1^{(y^g)^{-1}}=1
\] 
(we use the commutator identity $[a,b^{-1}]=[b,a]^{b^{-1}}$).
Thus, $y^{-1}\in G_x$.

If $y,z\in G_x$ then the identity $[a,cb]=[a,b][a,c]^b$ gives
\[
[x,(yz)^g]=[x,(y^gz^g)]=[x,z^g][x,y^g]^{z^g}=1\cdot 1=1.
\]
Thus, $yz\in G_x$ and we proved that $G_x$ is a subgroup in $G$. Let us show that $G_x$ is a normal\Wiki subgroup. If $z\in G$, $y\in G_x$ then 
\[
[x,(z^{-1}yz)^g]=[x,y^{zg}]=1,
\] 
and we obtain $z^{-1}yz\in G_x$. Thus, $G_x$ is a normal subgroup.

Since $G$ is simple, $G_x=G$. Therefore for $g=1$ we obtain that for any $z\in G$ it holds $[x,z]=1$. Thus $x$ belongs to the center\Wiki $Z$ of $G$. Since $Z$ is always a normal subgroup, we obtain $G=Z$. Thus, $G$ is abelian, a contradiction.  

\end{enumerate} 

\item The examples of e.d. in the class of groups allow us to define an e.d. in the class of semigroups of the language $\LL_{s}(C)$. Let us consider a finite Diophantine $C$-group $G$ such that $G$ is an e.d. in the language $\LL_{g}(C)$ (for example, one can choose any finite simple non-abelian group). According to Theorem~\ref{th:domain_about_main_set}, the set $M$ is the solution set of a system $\Ss$ of $\LL_{g}(C)$-equations. If all equations of $\Ss$ do not contain the operation of the inversion $\Ss$ becomes a system of $\LL_{s}(C)$-equations, and $G$ is an e.d. as a $C$-semigroup of the language $\LL_{s}(C)$. If $\Ss$  contains occurrences of ${}^{-1}$ one can replace all inversions to a sufficiently large positiver power as follows 
\[
x^{-1}=x^{|G|-1},
\] 
and $G$ becomes an e.d. as $C$-semigroup of the language $\LL_s(C)$.
\end{enumerate}  

It follows from Theorem~\ref{th:domains_for_groups} that the following classes of Diophantine $C$-groups are not e.d. in the language $\LL_{g}(C)$:
\begin{enumerate}
\item abelian groups;
\item groups with a nontrivial center (for example, nilpotent\Wiki groups)
\end{enumerate}

For semigroups of the language $\LL_{s}(C)$ we have the following results.

\begin{enumerate}
\item The free semigroup $FS_n$ and the free monoid $FM_n$ are not e.d. in the languages $\LL_{s}(C)$ $\LL_m(C)$ respectively (see the proof in~\cite{uni_Th_IV}).  

\item {\it Any nontrivial Diophantine $C$-semigroup $S$ with a zero\Wiki is not an e.d. in the language $\LL_{s}(C)$.} Let us prove this fact and consider the union of two algebraic sets $M_2=\{(x,y)\mid x=a\mbox{ or } y=a\}$, where $a$ is a nonzero element of $S$. Let us assume that $M_2$ is algebraic, i.e. there exists a system of  $\LL_{s}(C)$-equations $\Ss(x,y)$ with the solution set $M_2$. Since the point $(0,0)$ does not belong to $M_2$, there exists an equation in $\Ss$ which is not satisfied by $(0,0)$. It is clear that such equation should have the form $t(x,y)=c$, where $c$ is a nonzero constant. If the left part of $t(x,y)=c$ contain the variable $x$ (respectively, $y$) it is not satisfied by the point $(0,a)\in M_2$ (respectively,  $(a,0)\in M_2$). Thus, we obtained the contradiction with the choice of $\Ss$.

\end{enumerate}

\subsection{Comments}

The notion of an e.d. is the generalization of domains in commutative algebra (a domain is a commutative associative ring with no zero-divisors), since any domain is an e.d. in the ring language $\LL_r=\{+,-,\cdot,0\}$ (see above). Besides commutative rings and groups, in~\cite{uni_Th_IV} it was described e.d. in the classes of Lie algebras, anti-commutative algebras and associative algebras.

The first examples of e.d. in groups were the free non-abelian groups (it was initially proved by G.~Gurevich, see the proof in~\cite{makanin}) and finite non-abelian simile groups (see the proof above). 

We studied e.d. in various classes of semigroups. Above we mentioned about e.d. in the constant-free language $\LL_s$ (see~\cite{shevl_domains_without_const} for more details). Moreover, we described e.d. in the classes of completely simple~\cite{shevl_domains_c_s_finite}, inverse~\cite{shevl_domains_inverse} and completely regular~\cite{shevl_domains_clifford} semigroups. The classes of inverse and completely regular semigroups do not contain proper e.d. In other words, we proved the following theorems.

\begin{theorem}\textup{(\cite{shevl_domains_clifford})}
If a Diophantine completely regular $C$-semigroup $S$ of the language $\LL_s(C)$ is an e.d. then $S$ is completely simple.
\end{theorem}

\begin{theorem}\textup{(\cite{shevl_domains_inverse}}
If a Diophantine inverse $C$-semigroup $S$ of the language $\LL_{s-inv}(C)=\{\cdot,^{-1}\}\cup\{c_i\mid i\in I\}$ (all inverse semigroups admit the operation of the inversion) is an e.d. then $S$ is a group.
\label{th:domains_in_inverse_semigroups}
\end{theorem}

Remark that Theorem~\ref{th:domains_in_inverse_semigroups} generalizes the following result of ~\cite{rosenblat}.

\begin{theorem}\textup{\cite{rosenblat}}
If a set $M$~(\ref{eq:domain_main_set}) is the solution set of a single equation of the language  $\LL_{s-inv}(C)=\{\cdot,^{-1}\}\cup\{c_i\mid i\in I\}$ over an inverse semigroup $S$ then $S$ is a group.
\label{th:rosenblat}
\end{theorem}
 
The class of completely simple semigroups contains proper e.d. which are not groups. One can see such results in our papers~\cite{shevl_domains_c_s_finite} (for finite completely simple semigroups of the language $\LL_s(S)$) and~\cite{shevl_domains_c_s} (for arbitrary completely simple semigroups of the language $\LL_{s-inv}(C)=\{\cdot,^{-1}\}\cup\{c_i\mid i\in I\}$).  In~\cite{shevl_domains_finite} we describe all finite e.d. in the semigroup language $\LL_{s}(C)$.

In conclusion we remark that there is not any analogue of Theorem~\ref{th:domains_for_groups} for semigroups. In particular, we do not know e.d. among semigroups with the infinite minimal two-sided ideal and among semigroups with infinite descending chain of two-sided ideals.

\section{Types of equational compactness}
\label{sec:compactness_classes}
Actually, one can prove Unifying theorem~\ref{th:unify_approx},~\ref{th:unify_discr} for a wide class of $\LL$-algebras, not merely for equational Noetherian ones. Thus, in~\cite{uni_Th_III} there were defined the following generalizations of equationally Noetherian algebras.

An $\LL$-algebra $\Acal$ is {\it $\qq$-compact} if for any $\LL$-system $\Ss$ and an $\LL$-equation $t(X)=s(X)$ such that
\[
\V_\Acal(\Ss)\subseteq \V_\Acal(t(X)=s(X))
\]  
there exists a finite subsystem $\Ss^\pr\subseteq\Ss$ with
\[
\V_\Acal(\Ss^\pr)\subseteq \V_\Acal(t(X)=s(X)).
\]

An $\LL$-algebra $\Acal$ is {\it $\uu$-compact} if for any system $\Ss$ and $\LL$-equations $\{t_i(X)=s_i(X)\mid 1\leq i\leq m\}$ such that
\[
\V_\Acal(\Ss)\subseteq \bigcup_{i=1}^m\V_\Acal(t_i(X)=s_i(X))
\]  
there exists a finite subsystem $\Ss^\pr\subseteq\Ss$ with
\[
\V_\Acal(\Ss^\pr)\subseteq \bigcup_{i=1}^m\V_\Acal(t_i(X)=s_i(X)).s
\]  
Let $\Nbf,\Qbf,\Ubf$ denote the classes of equationally Noetherian, $\qq$-compact and $\uu$-compact $\LL$-algebras, respectively. By the definitions, it follows $\Nbf\subseteq \Ubf\subseteq \Qbf$. The importance of the classes $\Qbf,\Ubf$ follows from the following statement. 

\begin{theorem}\textup{\cite{uni_Th_III}}.
The statement of Theorem~\ref{th:unify_approx} (Theorem~\ref{th:unify_discr}) holds for an $\LL$-algebra $\Acal$ iff $\Acal$ is $\qq$-compact (respectively, $\uu$-compact). 
\label{unif_theorems_holds_iff_compactness}
\end{theorem}

Using the definition below, one can obtain a necessary condition for an $\LL$-algebra to be $\qq$- or $\uu$-compact.

An $\LL$-system $\Ss$ over an $\LL$-algebra $\Acal$ is called an {\it $E_k$-system} if  $|\V_\Acal(\Ss)|=k$, but for any finite subsystem $\Ss^\pr\subseteq\Ss$ it holds $|\V_\Acal(\Ss)|=\infty$. In particular, an $E_0$-system is an inconsistent system with consistent finite subsystems.  

\begin{theorem}\textup{\cite{kotov}}
If $\Acal\in\Qbf$ ($\Acal\in\Ubf$) then there are not $E_k$-systems over $\Acal$ for all $k\in\{0,1\}$ (respectively, $k\in\N$). 
\label{th:kotov_theorem_E_k}
\end{theorem} 
\begin{proof}

Let us prove the necessary condition of the $\qq$-compactness (one can see~\cite{kotov} for the whole proof of the theorem).
 
First of all we establish the following property of  $E_k$-systems. {\it If there exists an $E_k$-system $\Ss$ over an $\LL$-algebra $\Acal$ then there is a variable $x_i$ such that for any finite subsystem $\Ss^\pr\subseteq\Ss$ the variable $x_i$ has an infinite set of values in the set $\V_\Acal(\Ss^\pr)$}. (if there are not such variables there is a finite subsystem in $\Ss$ with a finite solution set that contradicts the definition of an $E_k$-system).

Let us prove there are not $E_0$- and $E_1$-systems over any $\qq$-compact $\LL$-algebra $\Acal$. Assume the converse: There exists an $E_0$-system $\Ss_0(X)$ (or $E_1$-system $\Ss_1(X)$) in variables  $X=\{x_1,x_2,\ldots,x_n\}$. By the definition, $\Ss_0(X)$ is inconsistent and $\Ss_1(X)$ has a unique solution $P=(p_1,p_2,\ldots,p_n)$.

Let $\Ss_i(Y)$ denote a system obtained from $\Ss_i(X)$ by the substitution $x_i\mapsto y_i$. Let $\Ss_{00}=\Ss_0(X)\cup\Ss_0(Y)$ (respectively, $\Ss_{11}=\Ss_1(X)\cup\Ss_1(Y)$). It is easy to see that the system $\Ss_{00}(X,Y)$ is inconsistent, and $\Ss_{11}(X,Y)$ has a unique solution $(P,P)$. By the properties of $E_k$-systems, there exists a pair $x_i,y_i$ such that $x_i,y_i$ has an infinite set of values in the solution set of any finite subsystem $\Ss_{00}^\pr(X,Y)\subseteq\Ss_{00}(X,Y)$ (respectively, $\Ss_{11}^\pr(X,Y)\subseteq\Ss_{11}(X,Y)$). 

Since $\V_\Acal(\Ss_{00}(X,Y))=\emptyset$ and $\V_\Acal(\Ss_{11}(X,Y))=(P,P)$ we have the set inclusions
\[
\V_\Acal(\Ss_{00}(X,Y))\subseteq\V_\Acal(x_i=y_i), \; \V_\Acal(\Ss_{11}(X,Y))\subseteq\V_\Acal(x_i=y_i).
\]
 
By the properties of the variables $x_i,y_i$, for any finite subsystem $\Ss^\pr_{00}(X,Y)\subseteq\Ss_{00}(X,Y)$, $\Ss^\pr_{11}(X,Y)\subseteq \Ss_{11}(X,Y)$ the inclusions
\[
\V_\Acal(\Ss_{00}^\pr(X,Y))\subseteq\V_\Acal(x_i=y_i), \; \V_\Acal(\Ss_{11}^\pr(X,Y))\subseteq\V_\Acal(x_i=y_i)
\]
do not hold. Thus, $\Acal$ is not $\qq$-compact.

\end{proof}

Let us give examples of $\qq$- and $\uu$-compact algebras.

\begin{enumerate}
\item Let $\LL_{s}(C)$ denote the semigroup language extended by a countable set of constant symbols $C=\{c_i\mid i\in\N\}$. Below we define three $C$-semilattices $\Acal,\Bcal,\Ccal$ such that $\Acal\notin\Qbf$, $\Bcal\in\Qbf\setminus\Ubf$, $\Ccal\in\Ubf\setminus\Nbf$.

Let us consider a linearly ordered $C$-semilattice $\Acal$ which consists of the elements $\{a_i\mid i\in\N\}$, $a_i<a_{i+1}$, $c_i^\Acal=a_i$. There exists an $E_0$-system $\Ss=\{x\geq c_i\mid i\in\N\}$  over $\Acal$ (obviously, $\V_\Acal(\Ss)=\emptyset$), therefore Theorem~\ref{th:kotov_theorem_E_k} states $\Acal\notin\Qbf$.

Let us add to $\Acal$ two non-constant elements $b_1,b_2$ with $b_1<b_2$, $a_i<b_j$ for all $i,j$ and denote the obtained linearly ordered $C$-semilattice by $\Bcal$.  The system $\Ss$ defined above is neither an $E_0$- nor $E_1$-system over $\Bcal$, since $\V_\Bcal(\Ss)=\{b_1,b_2\}$. In~\cite{shevl_at_service} it was proved that $\Bcal\in\Qbf$. However, $\Ss$ is an $E_2$-system, therefore Theorem~\ref{th:kotov_theorem_E_k} gives $\Bcal\notin\Ubf$. 

Let us add to $\Bcal$ a countable number of non-constant elements $b_3,b_4,\ldots,$ such that $b_i<b_{i+1}$ ($i\in\N$), $a_i<b_j$ for all $i,j$, and denote the obtained linearly ordered $C$-semilattice by $\Ccal$. Remark that the system $\Ss$ is not an $E_k$-system over $\Ccal$ for any $k\in\N$, since $\V_\Bcal(\Ss)=\{b_1,b_2,\ldots\}$. Actually, in~\cite{shevl_at_service} it was proved that $\Ccal\in\Ubf$.

\item There exists examples of $\qq$- and $\uu$-compact algebras in the languages with no constants. For instance, in~\cite{Plot3} it was proved that the ``Plotkin's monster'' (the direct product of {\it all} finitely generated group in the standard group language $\LL_g$)  is $\qq$-compact but not equationally Noetherian. The existence of a group from $\Ubf\setminus\Nbf$ follows from the following statement.

\begin{theorem}\textup{\cite{uni_Th_III}}
For any $\LL$-algebra $\Acal$ there exists an embedding of $\Acal$ into a $\uu$-compact $\LL$-algebra $\Bcal$.
\end{theorem}

\end{enumerate}

\bigskip

An $\LL$-algebra $\Acal$ is called \textit{consistently Noetherian} if any \textit{consistent} $\LL$-system $\Ss$ is equivalent to its finite subsystem. The class of all consistently Noetherian $\LL$-algebras is denoted by $\Nbf_c$.

In general, $\Nbf_c$ does not coincide with $\Nbf$, since we have the following statement.

\begin{theorem}\textup{\cite{shevl_free_semilattice}}
Let $\Fcal$ be the free Diophantine $C$-semilattice of infinite rank (recall that $\Fcal$ is isomorphic to the class of all finite subsets of some infinite set relative to the operation of set union). Then  $\Fcal\in\Nbf_c\setminus\Nbf$ (in other words, any consistent system over $\Fcal$ is equivalent to its finite subsystem, but there exists an infinite inconsistent system $\Ss$ such that all finite subsystem of $\Ss$ are consistent.
\end{theorem}

There are classes of $\LL$-algebras where the equality $\Nbf_c=\Nbf$ holds. For example, one can consider any class of $\LL$-algebras $\Cbf$ such that the empty set is not algebraic over any $\LL$-algebra $\Acal\in\Cbf$ (the class of all groups of the language $\LL_g$ satisfies this property). In this case all $\LL$-systems are consistent and the definition of the class $\Nbf_c$ becomes the definition of equationally Noetherian algebras. On the other hand, there exist classes of $\LL$-algebras with $\Nbf=\Nbf_c$, where the empty set is algebraic (see~\cite{shevl_boolean} where we study Boolean algebras in the language with constants).

\bigskip

An $\LL$-algebra $\Acal$ is \textit{weakly equationally Noetherian} if for any $\LL$-system $\Ss$ there exists a finite equivalent system $\Ss_0$ (we do not claim now for $\Ss_0$ to be a subsystem of $\Ss$). The class of all weakly equationally Noetherian $\LL$-algebras is denoted by $\Nbf^\pr$. Obviously, $\Nbf\subseteq\Nbf^\pr$.

\begin{proposition}
If the empty set is the solution set of some finite system over an $\LL$-algebra $\Acal$ then $\Acal\in\Nbf_c$ implies $\Acal\in\Nbf^\pr$. 
\end{proposition}
\begin{proof}
Suppose an $\LL$-algebra $\Acal$ is consistently Noetherian, and  $\Ss$ is an infinite $\LL$-system over $\Acal$. If $\Ss$ is consistent, $\Ss$ is equivalent to its finite subsystem. Otherwise ( $\V_\Acal(\Ss)=\emptyset$), the condition provides the existence of a finite inconsistent system $\Ss_0\sim_\Acal \Ss$.
\end{proof}

\begin{proposition}
For $\LL$-algebras we have
\[
\Qbf\cap\Nbf^\pr=\Nbf.
\]
\end{proposition}
\begin{proof}
Since the definitions of the classes above immediately give $\Nbf\subseteq \Qbf$, $\Nbf\subseteq \Nbf^\pr$, then $\Nbf\subseteq\Qbf\cap\Nbf^\pr$. Let us prove the converse inclusion.  Let an $\LL$-algebra $\Acal$ is $\qq$-compact and weakly equationally Noetherian. Consider an infinite $\LL$-system $\Ss$ over $\Acal$. Since $\Acal\in\Nbf^\pr$, $\Ss$ is equivalent to a some finite system $\{t_i(X)=s_i(X)\mid 1\leq i\leq n\}$. Therefore, for each equation  $t_i(X)=s_i(X)$ it holds
\[
\V_\Acal(\Ss)\subseteq\V_\Acal(t_i(X)=s_i(X)).
\]
Since $\Acal$ is $\qq$-compact, for each equation $t_i(X)=s_i(X)$ there exists a finite subsystem $\Ss_i\subseteq\Ss$ with
\[
\V_\Acal(\Ss_i)\subseteq\V_\Acal(t_i(X)=s_i(X)).
\]
Let $\Ss^\pr=\bigcup_{i=1}^n\Ss_i$ be a finite subsystem of  $\Ss$. By the choice of each $\Ss_i$ it is easy to prove that $\Ss^\pr$ is equivalent to $\Ss$ over $\Acal$. Thus, $\Acal$ is equationally Noetherian.
\end{proof}

In conclusion we give a picture which shows the inclusions of the classes $\Nbf,\Nbf^\pr,\Ubf,\Qbf$.

\begin{center}
\begin{picture}(150,100)
\qbezier(50,0)(69,0)(85,15) \qbezier(85,15)(100,31)(100,50)
\qbezier(100,50)(100,69)(85,85) \qbezier(85,85)(69,100)(50,100)
\qbezier(50,100)(31,100)(15,85) \qbezier(15,85)(0,69)(0,50)
\qbezier(0,50)(0,31)(15,15) \qbezier(15,15)(31,0)(50,0)

\qbezier(73,93)(70,90)(65,85) 
\qbezier(65,85)(50,69)(50,50)
\qbezier(50,50)(50,31)(65,15) 
\qbezier(65,15)(71,6)(70,4)

\qbezier(100,5)(114,5)(125,15) \qbezier(125,15)(135,26)(135,40)
\qbezier(135,40)(135,54)(125,65) \qbezier(125,65)(114,75)(100,75)
\qbezier(100,75)(86,75)(75,65) \qbezier(75,65)(65,54)(65,40)
\qbezier(65,40)(65,26)(75,15) \qbezier(75,15)(86,5)(100,5)

\put(23,50){$\Qbf$} \put(53,50){$\Ubf$}\put(83,50){$\Nbf$}
\put(113,50){$\Nbf^\prime$}
\end{picture}

\nopagebreak
\textbf{Fig. 1.}
\end{center}

\subsection{Comments}

The concept of $\qq$-compact ($\uu$-compact) algebra was given in~\cite{uni_Th_III} as a maximal class where Unifying Theorem~\ref{th:unify_approx} (respectively, both Unifying Theorems~\ref{th:unify_approx},~\ref{th:unify_discr}) remains true. Remark that in~\cite{Plot3} it was used the equivalent to the $\qq$-compactness definition of ``logically Noetherian algebras''.

Let us refer to papers devoted to $\qq$- and $\uu$-compact algebras. In~\cite{Plot3} it was defined the ``Plotkin's monster'', i.e. a $\qq$-compact but not equationally Noetherian group of the language $\LL_g$. Thus, the classes $\Qbf$, $\Nbf$ are distinct in the variety of groups. The examples of algebras from the classes $\Qbf\setminus\Nbf$, $\Qbf\setminus\Ubf$, $\Ubf\setminus\Nbf$ can be found in~\cite{kotov,shevl_at_service}. In~\cite{kotov} the examples were given in the language of infinite number of unary functions, while in~\cite{shevl_at_service} we find examples in the class of $C$-semilattices of the language $\LL_s(C)$.

The paper~\cite{kotov} gives the necessary conditions of $\qq$- and $\uu$-compactness (see Theorem~\ref{th:kotov_theorem_E_k}), therefore there arises a question: for what classes of $\LL$-algebras such conditions are sufficient? This problem was solved positively for Boolean algebras of the language $\LL_{C-bool}=\{\vee,\wedge,\neg,0,1\}\cup\{c_i\mid i\in I\}$ extended by constants  (see~\cite{shevl_boolean}) and for linearly ordered semilattices of the language $\LL_{s}(C)$  (\cite{shevl_linear_ordered_semi}). 

Let us note the paper~\cite{Shahryari3} by M. Shahryari where the $\qq$-compactness of an $\LL$-algebra $\Acal$ was described as a property of the lattice of congruences over $\Acal$.

\section{Advances of algebraic geometry and further reading}
\label{sec:researches}

In this section we give a brief survey of papers devoted to equations over some algebras and recommend papers for the further reading. 

\subsection{Further reading}

As we mentioned in Introduction, our lectures notes are based on the series of papers~\cite{uni_Th_I}--\cite{uni_Th_V} by DMR. Therefore, we strongly recommend you to read these papers. Let us give the content of each paper~\cite{uni_Th_I}--\cite{uni_Th_V}.

\begin{enumerate}
\item The first paper~\cite{uni_Th_I} develops universal algebraic geometry. The main disadvantage of this paper is the fundamental division of algebraic geometry into two parts: ``equations without constants'' and ``equations with constants''. Thus, the paper becomes complicated. This disadvantage was eliminated in the second paper~\cite{uni_Th_II}.  

\item The second paper~\cite{uni_Th_II} of the series contains the main definitions of universal algebraic geometry (algebraic sets, coordinate algebras, equationally Noetherian property). It generalizes the results of the first paper~\cite{uni_Th_I}, since the authors use the apt approach in the study of algebraic geometry. We follow this paper in all sections of our lectures notes. Thus, the reader may omit the paper~\cite{uni_Th_I} and begin the study of universal algebraic geometry with~\cite{uni_Th_II}. 

\item The paper~\cite{uni_Th_III} contains the definitions of $\qq$-, $\uu$-compact and weakly equationally Noetherian algebras and studies the properties of such classes. We follow this paper in Section~\ref{sec:compactness_classes}.  

\item The paper~\cite{uni_Th_IV} contains the descriptions of equational domains and co-domains. We used some results in Sections~\ref{sec:co-domains},~\ref{sec:equational_domains}.

\item The paper~\cite{uni_Th_V} deals with predicate languages and develops their algebraic geometry. Our paper~\cite{shevl_alg_geom_with_neq} applies the ideas of~\cite{uni_Th_V} to algebraic structures equipped the relation $\neq$.

\end{enumerate}

Moreover, the papers by M.\,Shahryari, P.\,Modabberi, H.\,Khodabandeh~\cite{Shahryari1,Shahryari2,Shahryari3,Shahryari4,Shahryari5} continue the study of general problems of universal algebraic geometry.

\subsection{Groups}

Firstly, we recommend the survey~\cite{RomankovSurvey} by V. Romankov devoted to equational problems in groups.

{\bf Free and hyperbolic groups}.

The algebraic geometry over free groups was studied in many papers. We refer only the principal ones~\cite{Lyndon,Appel,bb,Bryant,Makanin,Razborov1, Razborov2,GrigorchukKurchanov,Sela1,Sela2, Sela3,MR1, MRS,Remesl,KM1,KM2,KM3,KM4}.

The papers~\cite{KM1,KM2,KM3} contain the description of coordinate groups over the free group $FG_n$ of rank $n$. Precisely, it was proved that a finitely generated group $G$ is the coordinate group of an irreducible algebraic set over $FG_n$ iff $FG_n$ is embedded into the Lindon`s group $F^{\Zbb[t]}$. On the other hand, in~\cite{Sela1,Sela2,Sela3} it was obtained the description of limit groups over $FG_n$. According to Unifying theorems, the results of~\cite{Sela1,Sela2,Sela3} gives the another description of irreducible coordinate groups over $FG_n$.

The class of hyperbolic groups is close the class of free groups. The algebraic geometry over hyperbolic groups was studied in~\cite{BMR3,Groves1,Groves1-2,Groves2,Groves2-1}.

{\bf Partially commutative groups.}

Algebraic geometries over partially commutative groups in several varieties were studied in~\cite{CasKaz1, CasKaz2, CasKaz3, GTim, GTim2, GTim3, Misch2, Misch3, MischTim, Tim, Tim2}.

{\bf Nilpotent, solvable and metabelian groups.} 

Interesting and nontrivial results about equations in free metabelian group were obtained in the following papers~\cite{Chapuis,
Rem2, Rem-Shtor2, Rem-Shtor1, Rem-Rom1, Rem-Rom2, Rem-Tim, Rom1}.
The papers~\cite{GRom, MRom, Rom2, Rom3, Rom4, Rom5, Rom6, Rom8} are devoted to the algebraic geometry over solvable groups.

\subsection{Semigroups}

The description of coordinate $C$-monoids over the Diophantine $C$-monoid of natural numbers was obtained in~\cite{shevl_over_N_qvar,shevl_over_N_irred}. The consistency of systems of equations over $\N$ was studied in~\cite{Kryvyi}, and there were found the necessary and sufficient condition for a system of equations over $\N$ to be consistent. In~\cite{shevl_N_Nullstellensatz} it was developed the algorithm that for a given system $\Ss$ and an equation $t(X)=s(X)$ determines whether $t(X)=s(X)$ belongs to the radical $\Rad_\N(\Ss)$ or not. Moreover, in~\cite{shevl_N_Nullstellensatz} we offer the algorithm that for a given system $\Ss$ over $\N$ finds the irreducible components of $\V_\N(\Ss)$.

The paper~\cite{Ulyashev} contains the description of coordinate semigroup in various classes of completely simple semigroups.

In~\cite{shevl_free_semilattice} the irreducible coordinate $C$-semilattices over the Diophantine free $C$-semilattice were described.

A left regular band is a semigroup which satisfies the following identities
\[
\forall x\; xx=x,\; \forall x\forall y\; xyx=xy.
\]
Let $FLR_n$ denote the free left regular semigroup of rank $n$. Actually, in~\cite{steinberg} it was proved that $FLR_n,FLR_m$ are geometrically equivalent to each other, and the list of axioms of $\qvar(FLR_n)$ was obtained. In~\cite{shevl_free_left_regular} we described all finite subsemigroups of $FLR_n$ (since the variety of left regular semigroups is locally finite, we actually describe irreducible coordinate semigroups over $FLR_n$).

\subsection{Algebras}

The papers~\cite{ChSh, DKR1, DKR2, Daniyarova2,
DKR3, Daniyarova3, PoroshTim, Rem-Shtor3, RomSh} are devoted to algebraic geometry over Lie algebras. Let us refer to the paper~\cite{Rem-Shtor3}, where it was proved that even simple equations over free Lie algebra have the complicated solution sets. Thus, it is impossible to obtain a nice description of all coordinate algebras of algebraic sets over free Lie algebras. However, in~\cite{Daniyarova3} it was  obtained the description of {\it bounded} algebraic sets over free Lie algebras.  

Surprisingly, the solution sets of linear equations in free anti-commutative algebra is simpler than the corresponding sets over free Lie algebra (see~\cite{DaniyarovaOnskul}).

\end{document}